\documentclass[reqno, 12pt]{amsart}
\usepackage{amsmath}
\usepackage{amsthm}
\usepackage{amsfonts}
\usepackage{amssymb}
\usepackage{mathrsfs}
\usepackage{epsfig}
\usepackage{slashed}
\usepackage{mathtools}
\usepackage{enumerate}
\usepackage{fullpage}
 \usepackage{soul}
\usepackage{xcolor}

\newcommand{\mb}{\mathbf}
\newcommand{\mc}{\mathcal}

\renewcommand{\Re}{\mathrm{Re}\,}

\newcommand{\rg}{\mathrm{rg}\,}
\newcommand{\N}{\mathbb{N}}
\newcommand{\R}{\mathbb{R}}
\newcommand{\C}{\mathbb{C}}

\newcommand{\B}{\mathbb{B}}
\newcommand{\s}{\mathbb{S}}

\newcommand{\la}{\lambda}

\newcommand{\Hb}{\overline{\mathbb{H}}}
\newcommand{\ra}{\rightarrow}
\newcommand{\ve}{\varepsilon}

\newcommand{\rst}[1]{\ensuremath{{\mathbin |}%
\raise-.5ex\hbox{$#1$}}} 

\newcommand{\norm}[1]{{\left\vert\kern-0.25ex\left\vert\kern-0.25ex\left\vert #1 
    \right\vert\kern-0.25ex\right\vert\kern-0.25ex\right\vert}}

\newtheorem{lemma}{Lemma}[section]
\newtheorem{theorem}[lemma]{Theorem}
\newtheorem{corollary}[lemma]{Corollary}
\newtheorem{proposition}[lemma]{Proposition}
\theoremstyle{remark}
\newtheorem{remark}[lemma]{Remark}

\theoremstyle{definition}

\numberwithin{equation}{section}

\newcommand{\lastcfrac}[2]{%
	\vphantom{\cfrac{#1}{#2}}%
	\raisebox{\dimexpr1ex-\height}{%
		$\displaystyle
		\raisebox{.5\height}{$\ddots$}+\cfrac{#1}{#2}
		$%
	}%
}

\usepackage{hyperref}
\usepackage{cleveref}

\title[]{Co-dimension one stable blowup for the supercritical cubic wave equation}

\author{Irfan Glogi\'c}

\address{The Ohio State University, Department of Mathematics,
	231 West 18th Ave, Columbus, USA}

\address{Universit\"at Wien, Fakult\"at f\"ur Mathematik,
Oskar-Morgenstern-Platz 1, A-1090 Vienna, Austria}
\email{irfan.glogic@univie.ac.at}

\author{Birgit Sch\"orkhuber}
\address{Karlsruhe Institute of Technology, Institute for Analysis,  Englerstra{\ss}e 2, D-76131 Karlsruhe, Germany}

\address{Leopold-Franzens Universit\"at Innsbruck, Institut f\"ur Mathematik, Technikerstraße 13, 6020 Innsbruck, Austria}
\email{Birgit.Schoerkhuber@uibk.ac.at}

\thanks{Irfan Glogi\'c is supported by the Austrian Science Fund FWF, Projects P 30076 and P 34378. Also, support from The Ohio State University Graduate School via Presidential Fellowship (during which a part of this work was completed) is gratefully acknowledged. Birgit Sch\"orkhuber acknowledges funding by the Deutsche Forschungsgemeinschaft (DFG, German Research Foundation) - Project-ID 258734477 - SFB 1173}

\begin{document}
\begin{abstract}
For the focusing cubic wave equation, we find an explicit, non-trivial self-similar blowup solution $u^*_T$,  which is defined on the whole space and exists in all supercritical dimensions $d \geq 5$. For $d=7$, we analyze its stability properties without any symmetry assumptions and prove the existence of a set of perturbations which lead to blowup via $u^*_T$ in a backward light cone.  Moreover, this set corresponds to a co-dimension one Lipschitz manifold modulo translation symmetries in similarity coordinates.
\end{abstract}

\maketitle

\section{Introduction}
We consider the focusing wave equation
\begin{align}\label{Eq:cubicNLW}
(\partial^2_{t} - \Delta_{x}) u(t,x) = u(t,x)^3
\end{align}
where $(t,x) \in I \times \R^d$ and $I \subset \R$ is an interval containing zero.
The equation is invariant under the rescaling $u \mapsto u_{\lambda}$, 
\[u_{\lambda}(t,x) = \lambda^{-1} u(t/\lambda , x/\lambda ) \]
which leaves invariant the  $\dot H^{s_c}(\R^d)-$ norm, $s_c = \frac{d}{2} - 1$. Thus, the model is energy supercritical in $d \geq 5$ and we restrict ourselves to this case. It is well-known that Eq.~\eqref{Eq:cubicNLW} has solutions that blow up in finite time from smooth, compactly supported initial data 
\begin{align*}
 u[0]= (u(0,\cdot), \partial_t u(0,\cdot)),
\end{align*}
see \cite{Lev74}. Locally, the stable blowup behavior for Eq.~\eqref{Eq:cubicNLW} is described by the ODE blowup 
\begin{align}\label{Eq:ODE_blowup}
 u_T(t,x) = \frac{\sqrt{2}}{T-t}, 
\end{align}
see \cite{DonnSchoerk2012}, \cite{DonningerSchoerkhuber2017} and \cite{ChaDon19}.
We note that $u_T$ is an example of a self-similar blowup solution with trivial spatial profile. For the supercritical focusing wave equation in three dimensions it is well-known that non-trivial, smooth, self-similar solutions exist, see \cite{BizonMaisonWasserman2007}. However, in contrast to this case, where none of these solutions are known in closed form, Eq.~\eqref{Eq:cubicNLW} has the explicit self-similar solution
 
\begin{equation}\label{Eq:CritSol}
u^*(t,x)=\frac{1}{t} U\left(\tfrac{|x|}{t}\right), \quad U(\rho) =  \frac{2 \sqrt{2(d-1)(d-4)}}{d-4+3 \rho^2}.
\end{equation} 
To the best of our knowledge this solution has not been known so far. By using symmetries of the equation, we obtain a non-trivial self-similar blow up solution
\begin{equation}\label{Eq:CritSol_Blowup}
u_T^*(t,x)=(T-t)^{-1} U\left(\tfrac{|x|}{T-t}\right), \quad T > 0,
\end{equation} 
which is defined on all of $\R^d$ and becomes singular in forward time as $t \to T$. The aim of this paper is to investigate the stability properties of $u_T^*$
and to show that it has exactly one genuine unstable direction. 

\subsection{The main result}
Note that Eq.~\eqref{Eq:cubicNLW} is invariant under shifts in space and time
\[
S_{T,x_0}(t,x):=(t-T,x-x_0),
\]
where $T\in\R$ and $x_0=(x_0^j)_{j=1\dots d}\in\R^d$, 
as well as under time reflection
\[
R(t,x):=(-t,x)
\]
and Lorentz transformations 
\[  \Lambda(a) := \Lambda^d(a^d) \circ \Lambda^{d-1}(a^{d-1}) \circ \dots \circ  \Lambda^1(a^1),   \]
where $a =  (a^{j})_{j=1,\dots, d}\in\R^d $ and the Lorentz boost in $j$-th direction  $\Lambda^j(a^j)$ is given by
\[
\left \{
\begin{aligned}
t &\mapsto t \cosh(a^j)+x^j\sinh(a^j)\\
x^j &\mapsto t \sinh(a^j)+x^j\cosh(a^j)\\
x^k &\mapsto x^k \quad (k\neq j).
\end{aligned}
\right.
\]
By composing these transformations we set  
\[
\Lambda_{T,x_0}(a):=R \circ \Lambda(a) \circ S_{T,x_0}.
\]
This allows us to define a $(2d+1)$-parameter family of solutions to Eq.~\eqref{Eq:cubicNLW}
\[
u^*_{T,x_0,a}(t,x):=u^*\circ \Lambda_{T,x_0}(a)(t,x).
\]

Note that $(0,0) = \Lambda_{T,x_0}(a)(T,x_0)$ and the lightcones emanating from $(T,x_0)$ are mapped into the ones emanating from the origin. Hence, for $(t', x') = \Lambda_{T,x_0}(a)(t,x)$ we have the identity
\[ |x'|^2 - t'^2 = |x - x_0|^2 - (t-T)^2.\]
Also,
\begin{align*}
t' = A_0(a)(T-t) - A_j(a)(x^j - x_0^j),
\end{align*}
where 
\begin{align*} 
A_{0}(a) &:=\cosh (a^{1})  \cdots   \cosh (a^{d})\\
A_{1 }(a) &:= \sinh(a^1) \cosh (a^{2}) \dots \cosh (a^d) \\
A_{2 }(a) &:= \sinh(a^2) \cosh (a^{3}) \dots \cosh (a^d) \\
\vdots \\
A_{d}(a) &:= \sinh(a^d).
\end{align*} 

In this paper, for the reasons that are explained below, we focus on the case $d=7$. In this case, our solution $u^*_{T,x_0,a}$ assumes the following explicit form 
\begin{align}\label{Eq:CritSol_general}
u^*_{T,x_0,a}(t,x) = \frac{1}{T-t} \psi^*_{a}\left (\frac{x-x_0}{T-t}\right),
\end{align}
 where 
\begin{align}\label{Eq:BlowupSol_selfsimilar}
\psi_a^*(\xi) = \frac{4 \gamma(\xi,a) }{2 \gamma(\xi,a) ^2+ |\xi|^2 - 1},
\end{align}
with $\gamma(\xi,a) = A_0(a) - A_j(a)\xi^j$.  For $a  \in \R^7$ sufficiently small,  $A_0(a) = 1 + O(|a|)$ and $ A_j(a) = O(|a|)$ and thence  $\psi^*_a  \in C^{\infty}(\R^7)$. From Eq.~\eqref{Eq:CritSol_general} and the scaling of Sobolev norms we obtain 
\begin{align}\label{Eq:Blowup_behavior}
 \|u^*_{T,x_0,a} (t,\cdot) \|_{\dot H^k(\mathbb B^7_{T-t}(x_0))}  = (T-t)^{\frac{5}{2}- k} \|\psi^*_{a} \|_{\dot H^k(\mathbb B^7)},
 \end{align}
for fixed $k \in \N_0$ and all $0 \leq t <T$, i.e., the solution blows up in the backward lightcone of $(T,x_0)$ for $k > s_c=\frac{5}{2}$.
The main result of this paper is the following stability property of $u^*_{T,x_0,a}$.
 
\begin{theorem}\label{Th:Main}
Let $d=7$ and define $\mb h := (h_1, h_2)$ by 
\begin{align}\label{Eq:Def_UnstableDir}
h_1(x) = \frac{1}{(1+|x|^2)^2} , \quad h_2(x) = \frac{4}{(1+|x|^2)^3}.
\end{align} 
There exist constants $\omega >0$, $\delta > 0$ and $c > 0$ such that for all  $\mb v = (v_1,v_2) \in C^{\infty}(\overline{\B_2^7}) \times C^{\infty}(\overline{\B_2^7})$, satisfying
\[ \|\mb v\|_{\mc Y} := \| (v_1,v_2) \|_{H^4 \times H^3(\B_2^7)} \leq \frac{\delta}{c}\]
the following holds: There are parameters $a(\mb v) \in \overline{\mathbb B^{7}_{c \delta/\omega}}$, $x_0(\mb v) \in \overline{\mathbb B^{7}_{\delta}}$, $T(\mb v) \in [1 - \delta, 1+\delta]$ and $\alpha(\mb v) \in [-\delta,\delta]$ depending Lipschitz continuously on $\mb v$ with respect to $\mc Y$ such that for initial data
\begin{align}\label{Eq:InitialValue_NLW_Codim1}
u[0] = u^*_{1,0,0}[0] + \mb v + \alpha(\mb v)  \mb h 
\end{align}
there exists a unique solution $u$ to Eq.~\eqref{Eq:cubicNLW} on $\bigcup_{t\in
	[0,T)}\{t\}\times \overline{ \mathbb B^7_{T-t}}$. Furthermore, $u$ blows up at $t=T$, $x = x_0$ and converges to $u^*_{T,x_0, a}$ according to 
\begin{align*}
(T-t)^{k - \frac{5}{2}} & \| u[t] - u^*_{T,x_0, a}[t] \|_{\dot H^k \times \dot H^{k-1}  (\mathbb B_{T-t}^7(x_0))} \lesssim (T-t)^{\omega} 
\end{align*}
for $k=1,2,3$. 
\end{theorem}

Some remarks on the result are in order.  \\

\noindent{\it 1. Co-dimension 1 stability.} Our result can be interpreted as co-dimension one stability \textit{modulo symmetries}. This is elaborated in Section \ref{Sec:Intro_Self_Sim_Prop} below, see in particular Proposition \ref{Prop:Main} and the subsequent remarks.  \\

\noindent{\it 2. On the choice of the dimension.} Large parts of the proof can be generalized to other space dimensions with obvious adjustments. The key problem is the spectral analysis for the operator corresponding to the linearization around $u^*_{T,x_0, a}$. In seven dimensions, we can exploit some special structural properties that simplify the problem and allow us to solve the spectral problem \textit{rigorously}, see Section \ref{Sec:Overview_proof} for a detailed discussion.  However, we expect that a result analogous to Theorem \ref{Th:Main} can be established in all space dimensions $d \geq 5$. In fact, the numerical investigations that we performed together with Maliborski in \cite{GloMalSch20} suggest that the qualitative structure of the spectrum is the same in all space dimensions.   \\

\noindent{\it 3. Assumptions on the initial data/Regularity of the solution.} The assumption on the initial data, in particular the smallness of $\mb v$ in $\mc Y$, allows for an elementary proof of the Lipschitz continuity of the parameters $(a,T,x_0, \alpha)$.  The same is expected to hold for $\mb v$ small in $H^3 \times H^2(\B_2^7)$, but here more structure must be utilized and we will address this elsewhere.  \\
As in \cite{DonSch14b}, the solution in Theorem \ref{Th:Main} has to be understood as a lightcone solution, i.e., a solution of the abstract evolution equation  corresponding to a reformulation of Eq.~\eqref{Eq:cubicNLW} in self-similar coordinates, see below. \\

\subsection{Dynamical system formulation in adapted coordinates}\label{Sec:Intro_Self_Sim_Prop}

To prove Theorem \ref{Th:Main}, we consider small perturbations  of the blowup initial data 
\[u^*_{1,0,0}[0] = ( u^*_{1,0,0}(0,\cdot), \partial_t u^*_{1,0,0}(0,\cdot)),\]
where 
\[  u^*_{1,0,0}(0,x) = \psi^*_0(x) = U(|x|), \quad \partial_t u^*_{1,0,0}(t,x)|_{t=0} =  \psi^*_0(x)  + x^j \partial_j  \psi^*_0(x)  = U(|x|) + |x| U'(|x|), \] 
i.e., we study the initial value problem for Eq.~\eqref{Eq:cubicNLW} with data of the form
\begin{align}\label{Eq:InitianData}
u[0] = u^*_{1,0,0}[0]+(f,g).
\end{align}
Our aim is to identify a suitable class of functions $(f,g)$ such that the corresponding solution blows up in forward time $T >0$  at some point $x_0 \in \R^7$ and converges to $ u^*_{T,x_0,a}$ for some $a \in \R^7$ as $t \to T$ in a backward light cone
$\bigcup_{t\in[0,T)}\{t\} \times \overline{\mathbb B_{T-t}^{7}(x_0)}$. In general, the  blowup parameters $(T,x_0,a)$ will depend on the data.  Thus, for (yet unknown)  $T > 0$ and $x_0 \in \R^7$ we define adapted coordinates 
\begin{align*}
\tau = -\log(T-t) + \log T, \quad   \xi = \frac{x-x_0}{T-t},
\end{align*} 
and we set
\begin{align}\label{Eq:new_variable}
\psi(\tau,\xi) = T e^{-\tau} u(T-Te^{-\tau},Te^{-\tau} \xi + x_0)
\end{align}
such that Eq.~\eqref{Eq:cubicNLW}  reads
\begin{align}\label{Eq:cubicNLW_7d_selfsimilar}
\left[ \partial^2_{\tau} + 3 \partial_{\tau} + 2 \xi^j \partial_{\xi^j} \partial_{\tau} - (\delta^{jk} - \xi^j \xi^k)\partial_{\xi^j}\partial_{\xi^k} + 4 \xi^j \partial_{\xi^j} + 2 \right ] \psi(\tau,\xi)  = \psi(\tau,\xi)^3.
\end{align}

Now, for each $a \in \R^7$, $\psi_a^*$ defined in Eq.~\eqref{Eq:BlowupSol_selfsimilar} is a static solution to Eq.~\eqref{Eq:cubicNLW_7d_selfsimilar}. 
We write  Eq.~\eqref{Eq:cubicNLW_7d_selfsimilar} as a first order system by setting 
\begin{align}
\psi_1(\tau,\xi) := \psi(\tau,\xi) , \quad \psi_2(\tau,\xi) := \partial_{\tau} \psi(\tau,\xi) + \xi^j \partial_j \psi(\tau,\xi) + \psi(\tau,\xi)
\end{align}
and $\Psi(\tau) = (\psi_1(\tau, \cdot),\psi_2(\tau, \cdot))$, which yields
\begin{align}\label{Eq:Abstract_ODE_NLW_formal}
\begin{split}
\partial_{\tau} \Psi(\tau) & =   \tilde{\mb L} \Psi(\tau) + \mb N(\Psi(\tau)) \\
\end{split}
\end{align}
with
\begin{align*}
  \tilde{\mb L}  \mb u(\xi):=  \left ( \begin{array}{c} - \xi^j \partial_j u_1(\xi) -  u_1(\xi) + u_2(\xi)\\ 						
	\partial_j\partial^j u_1(\xi)  - \xi^j \partial_j u_2(\xi) - 2  u_2(\xi)  \end{array} \right),
\end{align*}
and $ \mb N(\mb u) = ( 0, u_1^3)$. Eq.~\eqref{Eq:Abstract_ODE_NLW_formal} now has the explicit static solution $\Psi^*_a = (\psi^*_{1,a}, \psi^*_{2,a})$ whereby $\psi^*_{1,a} =  \psi^*_a$ and $\psi^*_{2,a}(\xi) = \xi^j \partial_j \psi_a^*(\xi) + \psi_a^*(\xi)$.  We study Eq.~\eqref{Eq:Abstract_ODE_NLW_formal} as an abstract initial value problem on 
\[\mc H = H^3(\B^7) \times H^2(\B^7)\]
equipped with the usual norm
\[ \|\mb u \| := \|(u_1,u_2)\|_{H^3 \times H^2(\B^7)}. \]
We consider the differential operator $  \tilde{\mb L}: \mc D(  \tilde{\mb L} ) \subset \mc H \to \mc H$ for $ \mc D(  \tilde{\mb L} ) = C^4(\overline{\B^7}) \times C^3(\overline{\B^7})$ such that $\tilde{\mb L}$ is densely defined. In Section \ref{Sec:FreeTime_Evol} we prove that $ \tilde{\mb L}$ is closable and denote this closure by $(\mb L,\mc D(\mb L))$. Furthermore, we show in Section \ref{Sec:Nonlin_Pert} that the nonlinearity $\mb N$ is well-defined as a mapping on $\mc H$, see Eq.~\eqref{Eq:Trlin_Est}. The following result is crucial to prove Theorem \ref{Th:Main}. 

\begin{proposition}\label{Prop:Main}
There exists a manifold $\mc M \subset \mc H$ of co-dimension $9$ in a small neighborhood around zero  such that the following holds: There is an $\omega >0$ such that for any $\Phi_0 \in \mc M \cap \mc D(\mb L)$ there is $a \in \B^7$ for which the initial value problem
\begin{align}\label{Eq:Abstract_ODE_NLW}
\begin{split}
\partial_{\tau} \Psi(\tau) & =   \mb L \Psi(\tau) + \mb N(\Psi(\tau)) \\
\Psi(0)  & = \Psi^*_0 + \Phi_0
\end{split}
\end{align}
has a unique solution $\Psi \in C^1([0,\infty), \mc H)$ given by 
\[ \Psi(\tau) = \Psi^*_{a} +\tilde  \Phi(\tau) \]
where $\| \tilde \Phi(\tau) \| \lesssim e^{- \omega \tau}$ for $\tau \to \infty$.  
\end{proposition}

The number of co-dimensions of the manifold $\mc M$ is related to the spectral properties of a linearized operator, the spectrum of which is confined to the open left half plane except for eigenvalues $\lambda_0 = 0$, $\lambda_1 = 1$ and $\lambda^* = 3$. The first two eigenvalues are not genuine instabilities as they are artifacts of the symmetries of the problem, i.e., the Lorentz invariance and the space-time translation symmetry. The eigenvalue $\lambda^*$ is however a true instability; the associated eigenspace is one-dimensional and spanned by $\mb h$ defined in Eq.~\eqref{Eq:Def_UnstableDir}.
When linearizing around $\Psi^*_{a}$, the parameter $a$ occurs in the resulting evolution equation for the perturbation. Hence, the center direction determined by $\lambda_0$ can be controlled by a modulation argument  (i.e., by letting $a$ depend on $\tau$). However, the parameters $T$ and $x_0$ cannot be treated in this manner since they do not appear in the equation (only in the transformed initial data), as is already obvious from the above calculations. Hence, these unstable directions can only be controlled by restricting the initial data to a suitable subset, which is the manifold $\mc M$.
Hence, on the level of Eq.~\eqref{Eq:Abstract_ODE_NLW} a classical co-dimension one stability result cannot be obtained due to additional instabilities caused by translation invariance. 
However, one can show that the intersection of  $\mc M$ with a small ball in $\mc H$ around zero is contained in  $\mc M_1 \cap \mc M_2$, where $\mc M_2$ and $\mc M_1$ are manifolds of co-dimension $1$ and  $8$, respectively, the latter being associated with the translation instability. Hence, Proposition \ref{Prop:Main} is interpreted as co-dimension one stability \textit{modulo symmetries}. \\

\begin{remark} We note that if we would consider the stability problem for a fixed value of $T, x_0$ and $a$, we could directly apply the results of Bates and Jones \cite{BatJon89}. However, having to allow for a variation of the symmetry parameters depending on the initial data complicates the situation. 
\end{remark}

In a second step we translate the result of Proposition \ref{Prop:Main} to physical coordinates $(t,x)$. First, we show that the transformed initial data given by  \eqref{Eq:InitianData} can be written as 
\[\Psi(0) = \Psi^*_0 + \mb U((f,g) ,T,x_0).\]
In Section \ref{Sec:Codim1} we prove that $\mb U$ maps into $\mc H$ for every $T$ close to one, $x_0$ small enough and sufficiently smooth $(f,g)$ defined on a suitably larger ball, say $\B^7_2$. Moreover, one can show the following: For every initial data $(f,g)$ sufficiently small in  $H^3(\B^7_2) \times H^{2}(\B^7_2)$, there exist $T$ and $x_0$ depending on the data such that $\mb U((f,g),T,x_0) \in \mc M_1$ (the manifold associated to the symmetry eigenvalue $1$). In particular, this holds for an \textit{open set} of initial data in the physical space. The additional requirement $\mb U((f,g),T,x_0) \in \mc M_2$ imposes a co-dimension one condition. We prove that by adding a suitable multiple of $\mb h $ to the initial data, this can always be achieved. By assuming additional regularity we can prove that the parameters $T, x_0$ and $a$ depend Lipschitz continuously on the data. The precise statement is given in the following proposition. Theorem $1.1$ is obtained by translating this back into physical coordinates. 

\begin{proposition}\label{Prop:Main_Manifold}
There exist constants $c > 0$ and $\delta > 0$ such that for every $\mb v \in \mc Y$ with $\|\mb v\|_{\mc Y}  \leq  \frac{\delta}{c}$
the following holds: There are $x_0 \in \overline{\mathbb B^{7}_{\delta}}$, $T \in [1 - \delta, 1+\delta]$ and $\alpha \in [-\delta,\delta]$, depending Lipschitz continuously on $\mb v$ with respect to $\mc Y$, such that
\[\mb U(\mb v +\alpha \mb h  ,T, x_0) \in \mc M.\]
\end{proposition}

\subsection{Discussion of related results}
There is an extensive amount of work on the focusing nonlinear wave equation
\begin{equation}\label{Eq:Focusing}
	(\partial^2_{t} - \Delta_{x}) u(t,x) =|u(t,x)|^{p-1}u(t,x), \quad x\in\R^d, \quad p>1.
\end{equation}
A simple way to exhibit blowup for Eq.~\eqref{Eq:Focusing} is to construct self-similar solutions. Apart from the spatially homogeneous fundamental solution \begin{equation}\label{Def:FundSS}
	u_T(t,x)= (T-t)^{-\frac{2}{p-1}} \kappa, \quad \kappa=\left(\tfrac{2(p+1)}{(p-1)^2}\right)^{\frac{1}{p-1}}, \quad T>0,
\end{equation}
the existence of other self-similar solutions with non-trivial profiles has been proved for certain combinations of parameters $d$ and $p$. For example, for $d=3$ and $p\geq3$ an odd integer not equal to five, the existence of infinitely many such solutions to Eq.~\eqref{Eq:Focusing} has been established in \cite{BizonBreitenlohnerMaisonWasserman2010} and \cite{BizonMaisonWasserman2007}. However, to the best of our knowledge, apart from the ODE blowup  \eqref{Def:FundSS}, no other self-similar solution to Eq.~\eqref{Eq:Focusing} has ever been found in closed form and $u^*$ is the first example of this type. \\

To investigate the role of $u^*_T$ in the generic time evolution, we performed numerical simulations in a parallel work with Maliborski \cite{GloMalSch20}. Our findings suggest the following picture: For small radial initial data, the solution exists globally in time and decays to zero, while for large data the solution blows up via the ODE profile \eqref{Eq:ODE_blowup}. For data fine-tuned to the threshold, the evolution approaches $u^*_T$ for some intermediate period before one of the two above scenarios occurs. We note that a similar phenomenology has been observed in numerical experiments for Eq.~\eqref{Eq:Focusing} in the supercritical case $d=3$, $p=7$ in \cite{BizonChmajTabor2004} and for supercritical wave maps \cite{BizonChmajTabor2000,BiernatBizonMaliborski2016}. These observations are especially interesting due to the striking similarity with critical phenomena in gravitational collapse, where threshold solutions are typically self-similar \cite{GunMar07}.
From an analytic point of view, however, the description of threshold phenomena for supercritical wave equations in terms of self-similar solutions is completely open and Theorem \ref{Th:Main} is the very first step in this direction.  \\

We note that much more is known in the energy critical case. There, the threshold for blowup is described in terms of the Talenti-Aubin solution $W$, see the seminal work by Kenig and Merle \cite{KenigMerle2008}. Our main result is somehow in the spirit of the work by Krieger and Schlag \cite{KriegerSchlag2007}, where a co-dimension one stability result was obtained for the family $\{W_{\la}\}$ of rescalings of $W$ in a topology stronger than the energy (in the radial setting). The threshold character of the constructed manifold was then described by Krieger, Nakanishi and Schlag in \cite{KriegerSchlagNakanishi2014}. Their work was inspired by previous numerical observations made by Bizo\'n, Chmaj and Tabor in \cite{BizonChmajTabor2004}.
Later on, in a series of papers \cite{KriegerSchlagNakanishi2013_2}, \cite{KriegerSchlagNakanishi2013_1}, \cite{KriegerSchlagNakanishi2015}, Krieger, Nakanishi and Schlag finally gave a fairly complete characterization of the threshold dynamics in the energy space around a co-dimension one manifold containing also type II blowup solutions. In this context, we also refer to the  recent papers by Krieger \cite{Krieger_preprint2017}, respectively, by Burzio and Krieger \cite{BurzioKrieger_preprint2017}.  \

From a physical point of view, the supercritical case is more interesting and we are only at the very beginning of understanding these types of phenomena. However, a description of threshold dynamics as detailed as in the critical case seems completely out of reach at the moment.  \\

We conclude the discussion with some general comments. Obviously, the blowup profile is defined on all of $\R^d$ and it would be interesting to investigate its stability properties on the whole space, see the recent paper by Biernat, Donninger and the second author \cite{BieDonSch19}. Furthermore, $u^*$ as defined in Eq.~\eqref{Eq:CritSol} is a global smooth solution to Eq.~\eqref{Eq:cubicNLW} for all  $t \geq 1$. Its self-similar character implies that it has a time-independent critical norm. In fact $\| (u^*(t, \cdot), \partial_t u^*(t, \cdot)) \|_{\dot H^{s_c}\times \dot H^{s_c-1}(\R^d)} = \infty$,
which is in agreement with the characterization of global, scattering solutions obtained by Dodson and Lawrie in \cite{DodsonLawrie2015}.

We also mention some other results that are known for Eq.~\eqref{Eq:cubicNLW}. In $d=3$, the problem is subcritical and conformally invariant and blowup has been analyzed in the seminal work of Merle and Zaag, see e.g. \cite{MerleZaag2005}. Stable ODE blowup has been shown by Donninger and the second author in \cite{DonnSchoerk2012}. Non-scattering, global solutions to Eq.~\eqref{Eq:cubicNLW} have been studied by Donninger and Zengino\u{g}lu \cite{DonZen14} as well as by Donninger and Burtscher \cite{DonningerBurtscher2017}. In the energy critical dimension $d=4$, non-trivial self-similar solutions can be excluded. Instead, non self-similar blowup (type II) occurs, see the corresponding construction by Hillairet and Rapha\"{e}l \cite{HillairetRaphael2012}. Finally, in higher space dimensions the result of Collot \cite{Collot2018} on the existence of (at least co-dimension two unstable) type II blowup solutions applies to Eq.~\eqref{Eq:cubicNLW}, which makes the picture even more complex.

\subsection{Overview of the paper}\label{Sec:Overview_proof}
The strategy of the proof of Theorem \ref{Th:Main} builds on on the approach developed by Donninger and the second author in \cite{DonSch14b} to prove, without symmetry assumptions, the nonlinear asymptotic stability of the ODE blowup for the focusing supercritical wave equation in three space dimensions. However, as opposed to the ODE blowup solution, $u^*_T$ is non-trivial and unstable. Both aspects introduce new difficulties that were not present in previous works and demand the development of new techniques. One major difficulty is the rigorous analysis of the underlying spectral problem since the approach developed  in 
\cite{CosDonGloHua16,CosDonGlo17} cannot be applied in our case. Nevertheless, we introduce a new and efficient method to treat such problems, see Sec.~\ref{Sec:ODE_analyis}. In particular, this method yields simpler and shorter proofs of similar problems that have already been resolved by different means, see \cite{CosDonXia16,CosDonGloHua16,CosDonGlo17,DonGlo19}. 
\\

In the following, we give a brief overview of the main steps in the proof of Theorem \ref{Th:Main}, respectively, Propositions \ref{Prop:Main} and \ref{Prop:Main_Manifold}. \\

In Section \ref{Sec:Problem_formulation}, we insert the modulation ansatz $\Psi(\tau) =\Psi^*_{a(\tau)} + \Phi(\tau)$ into Eq.~\eqref{Eq:Abstract_ODE_NLW_formal} and by assuming that $a(\tau) \to a_{\infty}$ as $\tau \to \infty$, we write the resulting equation for $\Phi$ as 
\begin{align}\label{Eq:Intro_Nonlin_ModEq}
 \partial_{\tau} \Phi(\tau)  =  [\mb L     + \mb L'_{a_{\infty}}] \Phi(\tau) + \mb G_{a(\tau)}(\Phi(\tau))  - \partial_{\tau} \Psi^*_{a(\tau)} 
\end{align}
with $ \mb L'_{a_{\infty}}$ containing a potential  term and $\mb G_{a(\tau)}$, the remaining nonlinearity. 
Sections \ref{Sec:Semigroup_theory} - \ref{Sec:Spectrum_growthbounds} are devoted to the analysis of operators of the type $\mb L_a = \mb L    + \mb L'_{a}$ for sufficiently small $a$. They are defined as unbounded operators $\mb L_a : \mc D(\mb L_a ) \subset \mc H \to \mc H$, where $\mc H = H^k(\B^7) \times H^{k-1}(\B^7)$ for $k=3$.\\
 
In Section \ref{Sec:Semigroup_theory} we prove that $\mb L$ generates a strongly continuous semigroup $(\mb S(\tau))_{\tau \geq 0}$.
We point out that in order to be able to prove exponential decay of the semigroup it is indispensable to work in a topology strictly above scaling, i.e, $k > s_c$ and we chose the lowest possible integer value. In addition, we introduce an equivalent, taylor-made inner product on $\mc H$ to prove the desired growth bounds for the semigroup, see Remark \ref{Rem:norm} below for an extended discussion.\\

In Sections \ref{Sec:ODE_analyis}  and \ref{Sec:Spectrum_growthbounds} we analyze the spectrum of $\mb L_a$ and prove suitable growth bounds for the associated semigroup $(\mb S_a(\tau))_{\tau \geq 0}$. While in  \cite{DonSch14b}, \cite{ChaDon19} the spectral problem associated to the linearization around the trivial ODE profile could easily be solved by elementary methods, we are facing new challenges here. Since $u^*_T$ has a non-trivial spatial profile we obtain a non-trivial potential in the linearization, which severly complicates the problem. Furthermore, in addition to the eigenvalues $\la_0 = 0$ and $\la_1  =1$ which are induced by the Lorentz symmetry and the space-time translation invariance of Eq.~\eqref{Eq:cubicNLW} there is another genuinely unstable eigenvalue $\la^*$. The symmetry eigenvalues $\la_0=0$ and $\la_1=1$ are the same in all space dimensions and their geometric multiplicity is equal to $d$, respectively to $d+1$ (note that this is different from \cite{DonSch14b}, \cite{ChaDon19}, where no spatial translations had to be taken into account).  For $\la^*$, on the other hand, numerics indicate that it is real\footnote{Note that, due to non-selfadjoint character of the underlying spectral problem, eigenvalues are not necessarily real.} and decreases as the dimension $d$ grows. In fact, it appears that $\la^*(d) \to 2$ as $d \to \infty$.  As it turns out, $\la^* =3$ in $d=7$, and the corresponding eigenfunction is \textit{explicit} and given by $\mb h$, see Eq.~\eqref{Eq:Def_UnstableDir}. This coincidence is due to the invariance of the spectral equation under conformal transformations, which in this particular dimension relates one of the eigenfunctions associated to $\la  =1$ to $\mb h$ by conformal symmetry, see also \cite{GloMalSch20} for a detailed discussion.

For $a =0$, the potential is radial and a decomposition into spherical harmonics reduces the eigenvalue problem to a family of ODEs, which, due to the non-trivial potential, are of Heun-type, i.e., they have four regular singular points and the understanding of such problems is rather limited. 
Nervertheless, for $d=7$, we are able to analyze these equations by improving the techniques developed by the first author together with Costin and Donninger  \cite{CosDonGlo17}, see also  \cite{CosDonGloHua16}. Here, the explicit knowledge of the eigenfunction $\mb h$ is crucial.

The case $a \neq 0$ is then treated perturbatively. More precisely, we show that there is $\tilde \omega > 0$ such that for the spectrum of $\mb L_a$ we have
\[ \sigma(\mb L_a) \subset \{\la \in \C : \mathrm{Re} \la \leq - \tilde \omega \} \cup \{0, 1, 3\} \]
for all $a \in\R^7$ that are sufficiently small.
Based on the information on the spectrum of $\mb L_a$ and suitable bounds on the resolvent, we conclude Section \ref{Sec:Spectrum_growthbounds} with estimates for the semigroup $\mb S_a(\tau)$ that guarantee exponential decay on a stable subspace defined via spectral projection. Another indispensable ingredient of the proof of the main result are Lipschitz estimates with respect to $a$, for basically all relevant quantities that occur at the linear (and later at the nonlinear) level. \\

In Section \ref{Sec:Nonlin_Pert} we turn to the analysis of Eq.~\eqref{Eq:Intro_Nonlin_ModEq}. In fact, we study the associated integral equation
\begin{align}\label{Eq:Intro_Integral_Equation}
\Phi(\tau) = \mb S_{a_{\infty}}(\tau) \mb u + \int_0^{\tau} 
 \mb S_{a_{\infty}}(\tau - \sigma)[ \mb G_{a(\sigma)}(\Phi(\sigma)) - \partial_{\sigma} \Psi^*_{a(\sigma)} ] d \sigma
\end{align}
for $\mb u \in \mc H$. The function spaces for $\Phi$ and $a$ are chosen in a way that reflects their desired properties, i.e., that $\Phi(\tau)$ decays exponentially and that $a(\tau) \to a_{\infty}$ as $\tau \to \infty$. 
In Section \ref{Sec:Modulation_CorrectedData} we proceed along the lines of \cite{DonSch14b}: First, we derive a modulation equation for $a(\tau)$ and show by a fixed point argument that $a$ can be chosen such that the center direction induced by the Lorentz transform is controlled. Since the parameters $T$ and $x_0$ do not occur in the equation, the translation instability is dealt with in the same manner as the instability caused by $\lambda^*$: we modify the initial data by subtracting suitable elements of the corresponding unstable
 subspaces to obtain a solution $\Phi$ to  Eq.~\eqref{Eq:Intro_Integral_Equation}, with $\mb u$ replaced by 
\[  \mb u - \mb C(\Phi, a, \mb u),\]
provided that $\mb u$ and $a$ are sufficiently small. 
Here, $\mb C = \mb C_1 + \mb C_2$ serves to suppress the instabilities originating from $\la =1$ and $\la^* = 3$. Furthermore, by using similar arguments as in \cite{DonningerBurtscher2017}, we explicitly construct a co-dimension 9 manifold $\mc M$ determined by the vanishing of the correction term. This yields Proposition \ref{Prop:Main}.
In order to prove Proposition \ref{Prop:Main_Manifold}, we consider initial data for Eq.~\eqref{Eq:cubicNLW} of the form 
\[ u[0] = u^*_{1,0,0}[0] + \mb v + \alpha  \mb h , \]
for $\alpha \in \R$, which transforms into initial data 
\[ \Phi(0) = \mb U(\mb v + \alpha \mb h, T, x_0) \]
for Eq.~\eqref{Eq:Intro_Nonlin_ModEq}. We prove suitable Lipschitz estimates for $\mb U$ (here the additional regularity assumption on $\mb v$ comes into play) and apply a fixed point argument to show that for every suitably small $\mb v \in \mc Y$, there are parameters $(a,T,x_0, \alpha)$ close to $(0,1,0,0)$ depending Lipschitz continuously on $\mb v$ such that $\mb U(\mb v + \alpha \mb h, T, x_0)  \in \mc M$, in particular
\[ \mb C(\Phi, a, \mb U(\mb v + \alpha \mb h, T, x_0) ) = 0.\]
 Hence, there exists a solution $\Phi$ satisfying
\begin{align*}
\Phi(\tau) = \mb S_{a_{\infty}}(\tau) \mb U(\mb v + \alpha \mb h, T, x_0)  + \int_0^{\tau} 
 \mb S_{a_{\infty}}(\tau - \sigma)[ \mb G_{a(\sigma)}(\Phi(\sigma)) - \partial_{\sigma} \Psi^*_{a(\sigma)} ] d \sigma
\end{align*}
with $\|\Phi(\tau) \|_{H^3 \times H^2(\B^7)} \lesssim e^{-\omega \tau}$ for all $\tau \geq 0$ and some $\omega > 0$. Reverting to the original coordinates yields Theorem \ref{Th:Main}

\subsection{Notation and Conventions}
Throughout the whole paper the Einstein summation convention is in force, i.e., we sum over repeated upper and lower indices, where latin indices run from $1,\dots,d$.

We write $\N$ for the natural numbers $\{1,2,3, \dots\}$, $\N_0 := \{0\} \cup \N$. Furthermore, $\R^+ := \{x \in \R: x >0\}$. Also, $\Hb$ stands for the closed complex right half-plane.
By $\mathbb B_R^d(x_0)$ we denote the open ball of radius $R >0$ in $\R^d$ centered at $x_0 \in \R^d$. The unit ball is abbreviated by $\B^d := \B_1^d(0)$
and $\mathbb S^{d-1} := \partial \B^d$.

By $L^2(\B_R^d(x_0))$ and $H^{k}(\B_R^d(x_0))$, $k \in \N_0$, we denote the  Lebesgue  
and Sobolev spaces obtained from the completion of $C^{\infty}(\B_R^d(x_0))$ with respect to the usual norm
\[ \|  u \|^2_{H^k(\B_R^d(x_0))} := \sum_{ |\alpha| \leq k} \|\partial^{\alpha}  u \|^2_{L^2(\B_R^d(x_0))},\]
with $\alpha \in \N_0^d$ denoting a multi-index and $\partial^{\alpha} u = \partial_1^{\alpha_1} \dots \partial_d^{\alpha_d} u$, where $\partial_i u(x) = \partial_{x_j} u(x)$. 
For vector-valued functions, we use boldface letters, e.g., $\mb f = (f_1,f_2)$ and we sometime write $[\mb f]_1 := f_1$ to extract a single component. 
Throughout the paper,
$W(f,g)$ denotes the Wronskian of two functions $f,g \in C^{1}(I)$, $I \subset \R$, where we use the convention
$W(f,g)=fg'-f'g$ with $f'$ denoting the first derivative.

On a Hilbert space $\mc H$ we denote by $\mc B(\mc H)$ the set of bounded linear operators.
For a closed linear operator $(L, \mc D(L))$ on $\mc H$, we define the resolvent set $\rho(L)$ as the set of all $\la\in\mathbb{C}$ such that $R_{L}(\lambda):=(\lambda- L)^{-1}$ exists as a bounded operator on the whole underlying space.
Furthermore, the spectrum of $L$ is defined as $\sigma(L):=\mathbb{C}\setminus \rho(L)$ and the point spectrum is denoted by $\sigma_p(L) \subset  \sigma(L)$.
The notation $a\lesssim b$ means $a\leq Cb$ for an absolute constant $C>0$ and we write $a\simeq b$ if $a\lesssim b$ and $b \lesssim a$.  	
If $a \leq C_{\varepsilon} b$ for a constant $C_{\varepsilon}>0$ depending on some parameter $\varepsilon$, we write $a \lesssim_{\varepsilon} b$.

\subsubsection*{Spherical harmonics}
For fixed $d \geq 3$ we denote by  $Y_{\ell}:  \mathbb S^{d-1} \to \C$ a spherical harmonic function of degree $\ell \in \N_0$ (i.e., the restriction of a harmonic homogeneous polynomial $H_{\ell}(x_1, \dots,x_d)$ of degree $\ell$ in $\R^d$ to the $(d-1)$-sphere). In particular,   $Y_{\ell}$ is an eigenfunction for the Laplace-Beltrami operator  on $\mathbb S^{d-1}$ with eigenvalue $\ell(\ell + d -2)$ and 
\[ \int_{\s^d} Y_{\ell}(\omega) \overline{Y_{\ell'}(\omega)} d \sigma(\omega)  = \delta_{\ell \ell'}.  \]
For each $\ell \in \N$, we denote by $M_{d,\ell} \in \N$ the number of linearly independent spherical harmonics, and we designate by $\{Y_{\ell, m} : m \in \Omega_{\ell} \}$, $\Omega_{\ell} = \{1,\dots, M_{d,\ell}\}$ a linearly independent orthonormal set such that 
\[ \int_{\s^d} Y_{\ell,m}(\omega) \overline{Y_{\ell,m'}(\omega)}d \sigma(\omega)  = \delta_{m m'}. \]
Obviously, one has $\Omega_{0} = \{1 \}$, $\Omega_{1} = \{1, \dots, d\}$,
and $Y_{0,1}(\omega) = c_1$, $Y_{1, m}(\omega) = \tilde  c_{m} \omega_m$ 
for suitable normalization constants $c_1, \tilde c_{m} \in \R$. For $g \in C^{\infty}(\mathbb S^{d-1})$, we define
$\mc P_{\ell}: L^2(\mathbb S^{d-1}) \to L^2(\mathbb S^{d-1})$ by
\begin{align*}
\mc P_{\ell} g(\omega)  := \sum_{m \in \Omega_{\ell}} (g|Y_{\ell, m} )_{L^2(\mathbb S^{d-1})} Y_{\ell, m}(\omega).
\end{align*}
It is well-known, see e.g. \cite{Atkinson}, that $\mc P_{\ell}$ defines a self-adjoint projection on $L^2(\mathbb S^{d-1})$ and that 
$ \lim_{n \to \infty}   \| g  - \sum_{\ell = 0}^{n} \mc P_{\ell} g \|_{L^2(\mathbb S^{d-1})} = 0 $.
This can be extended to Sobolev spaces, in particular,  $\lim_{n \to \infty} \|g-   \sum_{\ell = 0}^{n} \mc P_{\ell}  g  \|_{H^k(\mathbb S^{d-1})} \to 0$
for all $g \in C^{\infty}(\mathbb S^{d-1})$, see e.g.~\cite{DonSch14b}, Lemma A.1. Furthermore, for $f \in C^{\infty}(\overline{\B^d})$ and
\begin{align}\label{Decomp:Projection}
[ P_{\ell} f](x)  := \sum_{m \in \Omega_{\ell}} (f(|x| \cdot)|Y_{\ell, m} )_{L^2(\mathbb S^{d-1})} Y_{\ell, m} \left (\tfrac{x}{|x|}\right)
\end{align}
we have, see for example Lemma A.2,
\begin{align}\label{Decomp:SpherHarm_Hk}
\lim_{n \to \infty} \|f- \sum_{\ell = 0}^{n} P_{\ell} f  \|_{H^k(\B^d)} \to 0.
\end{align}

\section{Evolution equation for the perturbation in similarity coordinates}\label{Sec:Problem_formulation}

In Section \ref{Sec:Intro_Self_Sim_Prop} we reformulated the cubic wave equation \eqref{Eq:cubicNLW} with the perturbed blowup initial data 
\[u[0] = u^*_{1,0,0}[0]+(f,g)\]
as a first order system  in similarity coordinates given by
\begin{align}\label{NLW_selfsim_sys_withdata}
\begin{split}
\partial_{\tau} \Psi(\tau) & =   \tilde{\mb L} \Psi(\tau) + \mb N(\Psi(\tau)), \quad \tau > 0    \\
 \Psi(0)  & =\Psi_0^* +  \mb U((f,g), T, x_0) 
\end{split}
\end{align}
where $ \mb N(\mb u) = ( 0, u_1^3)$,
\begin{align}\label{Def:tildeL}
 \tilde{\mb L}  \mb u(\xi) =  \left ( \begin{array}{c} - \xi^j \partial_j u_1(\xi) -  u_1(\xi) + u_2(\xi)\\ 					\partial_j\partial^j u_1(\xi)  - \xi^j \partial_j u_2(\xi) - 2  u_2(\xi)  \end{array} \right),
\end{align}
and 
\begin{align}\label{Eq:InitialData_Op}
 \mb U((f,g), T, x_0) := \mc R((f,g), T, x_0) + \mc R(\Psi^*_0, T, x_0) -  \Psi^*_0
\end{align}
with
\begin{align*}
 \mc R((f,g), T, x_0) :=   \left ( \begin{array}{c}  T f(T\cdot+x_0)   \\ T^2 g(T\cdot+x_0) 	 \end{array} \right).
 \end{align*}
 
 In the following, we assume  that 
 \[\Psi(\tau) = \Psi^*_a +\Phi(\tau),\]
where $\Phi$ is a small perturbation. The parameter $a \in \R^7$ depends on the initial data in general. We therefore insert the modulation ansatz  
\[ \Psi(\tau) = \Psi^*_{a(\tau)} + \Phi(\tau),\]
into  Eq.~\eqref{NLW_selfsim_sys_withdata} assuming that $a(0) = 0$ and $\lim_{\tau \to \infty} a(\tau) = a_{\infty} \in \R^7$. As a result, we obtain the equation
\begin{align}\label{Eq:ODE_NLW_Pert}
\begin{split}
\partial_{\tau} \Phi(\tau)   & =  [\tilde{\mb L}   + \mb L'_{a(\tau)} ] \Phi(\tau)  + \mb F_{a(\tau)}(\Phi(\tau))  - \partial_{\tau} \Psi^*_{a(\tau)} , \quad \tau > 0    \\
  \Phi(0) & = \mb U((f,g), T, x_0) 
\end{split}
\end{align}
where 
\begin{align}\label{Def:Potential}
 \mb F_{a(\tau)}(\mb u) = 
  \left ( \begin{array}{c} 0 \\ u_1^3 + 3 \psi^*_{a(\tau)} u_1^2 \end{array} \right), \quad \mb L'_{a} \mb u  =  \left ( \begin{array}{c} 0 \\ V_{a} 	u_1 \end{array} \right), \quad 
 V_{a}(\xi) := 3 \psi^*_{a}(\xi)^2.
\end{align}
For the potential, recall from Eq.~\eqref{Eq:BlowupSol_selfsimilar} that there is a $\delta^* >0$ such that $\psi^*_a  \in C^{\infty}(\R^7)$ for all $a \in \overline{\mathbb B^7_{\delta^*}}$. Furthermore, $\psi^*_a$ depends smoothly on $a$ and by the fundamental theorem of calculus,
\[ \psi_a^*(\xi) - \psi_b^*(\xi) = (a^j - b^j) \int_0^{1} \partial_{\alpha_j} \psi_{\alpha(s)}^*(\xi) ds ,\]
for $\alpha(s) = b + s(a-b)$. This implies the Lipschitz estimate
\begin{align}\label{Eq:Selfsim_Sol_Lipschitz}
\| \psi^*_a - \psi^*_b \|_{\dot H^k(\mathbb B^7)} \lesssim_k |a - b|
\end{align}
for all $a,b, \in \mathbb B^7_{\delta^*}$ and $k \in \N_0$. Analogously, one can show that for $k \in \N_0$,
\begin{align*}
\| V_{a} - V_b \|_{\dot H^k(\mathbb B^7)} \lesssim_k |a - b|
\end{align*}
for all sufficiently small $a,b \in \R^7$. \\

To deal with the time-dependent potential in  Eq.~\eqref{Eq:ODE_NLW_Pert} we rewrite the equation as
\begin{align}\label{Eq:SelfSim_Perturb}
\begin{split}
\partial_{\tau} \Phi(\tau) & =  [\tilde {\mb L}     + \mb L'_{a_{\infty}}] \Phi(\tau) + \mb G_{a(\tau)}(\Phi(\tau))  - \partial_{\tau} \Psi^*_{a(\tau)} , \quad \tau > 0   \\
\Phi(0) &  = \mb U((f,g), T, x_0)
\end{split}
\end{align}
where 
\begin{align*}
 \mb G_{a(\tau)}(\Phi(\tau)) := [\mb L'_{a(\tau)}  - \mb L'_{a_{\infty}}]\Phi(\tau) + \mb F_{a(\tau)}(\Phi(\tau)). 
 \end{align*}
 
The evolution equation  \eqref{Eq:SelfSim_Perturb} will be investigated in the Sobolev space 
\[\mc H = H^3(\B^7) \times H^2(\B^7)\]
equipped with the norm
\[ \|\mb u \| = \|(u_1,u_2)\|_{H^3 \times H^2(\B^7)}. \]

\section{Semigroup theory}\label{Sec:Semigroup_theory}

\subsection{The free time evolution}\label{Sec:FreeTime_Evol}
The differential operator $ \tilde{\mb L}$ given in \eqref{Def:tildeL} describes the free wave evolution in similarity coordinates. As an operator on $\mc H$ we define it as 
$  \tilde{\mb L}: \mc D(  \tilde{\mb L} ) \subset \mc H \to \mc H$ for $ \mc D(  \tilde{\mb L} ) = C^4(\overline{\B^7}) \times C^3(\overline{\B^7})$ such that it is densely defined. 
In the following, we prove that the closure of $ \tilde{\mb L} $ is the generator of a strongly-continuous semigroup. However, to prove  a suitable growth bound, we work with a different norm on our function space defined in the following, see also Remark \ref{Rem:norm} below.

\subsubsection*{Equivalent norm on $H^3 \times H^2(\B^7)$.} On $C^3(\overline{\B^7}) \times C^2(\overline{\B^7})$ we define 
\begin{align*}
(\mb u| \mb v)_1 & := \int_{\B^7} \partial_i \partial_j \partial_k u_1(\xi) \overline{\partial^i \partial^j \partial^k v_1(\xi)} d\xi
+ \int_{\B^7}  \partial_i \partial_j  u_2(\xi) \overline{\partial^i \partial^j v_2(\xi) } d\xi 
\\
& + \int_{\mathbb S^6}  \partial_i \partial_j  u_1(\omega)\overline{ \partial^i \partial^j v_1(\omega) }d\sigma(\omega) \\
(\mb u| \mb v)_2 & := \int_{\B^7} \partial_i \Delta u_1(\xi) \overline{\partial^i  \Delta  v_1(\xi)} d\xi
+ \int_{\B^7}  \partial_i \partial_j  u_2(\xi) \overline{\partial^i \partial^j v_2} (\xi) d\xi 
+ \int_{\mathbb S^6}  \partial_i u_2(\omega)\overline{ \partial^i  v_2(\omega) } d\sigma(\omega) 
\end{align*}
and set
\begin{align*}
(\mb u | \mb v)_{\mc H}  & :=
 4 (\mb u| \mb v)_1 +  (\mb u| \mb v)_2 \\
 &  +  
 \int_{\mathbb S^6}  \partial_i u_1(\omega) \overline{ \partial^i  v_1(\omega)} d\sigma(\omega) + \int_{\mathbb S^6}   u_1(\omega)  \overline{v_1(\omega)} d\sigma(\omega) 
+ \int_{\mathbb S^6}  u_2(\omega)  \overline{v_2(\omega) }d\sigma(\omega).   
\end{align*}
Furthermore, we let 
 \[ \| \mb u\|_{\mc H} := \sqrt{(\mb u | \mb u)_{\mc H}}\]
and show that this defines an equivalent norm on $H^3 \times H^2(\B^7)$. 
\begin{lemma}\label{Le:Equiv_Norm1} 
We have  
\[ \| \mb u \|_{\mc H} \simeq  \|\mb u \|  \]
for all $\mb u \in C^3(\overline{\B^7}) \times C^2(\overline{\B^7})$.
In particular, $\|\cdot \|_{\mc H}$ defines an equivalent norm on $\mc H$.
\end{lemma}

\begin{proof}
By the definition
\begin{align*}
 \|u\|^2_{H^2(\B^7)}  = \| u \|^2_{L^2(\B^7)} +  \| \nabla u \|^2_{L^2(\B^7)}  + \int_{\B^7} \partial_i \partial_j u(\xi) \overline{\partial^i \partial^j u(\xi)} d\xi 
 \end{align*}
and 
\begin{align*}
 \|u\|^2_{H^3(\B^7)}  = \| u \|^2_{L^2(\B^7)} +  \| \nabla u \|^2_{L^2(\B^7)}  + \int_{\B^7} \partial_i \partial_j u(\xi) \overline{\partial^i \partial^j u(\xi)} d\xi + 
  \int_{\B^7} \partial_i \partial_j \partial_k u(\xi) \overline{\partial^i \partial^j  \partial^k u(\xi)} d\xi.
 \end{align*}

To prove the statement we use that for $u \in C^1(\overline{\B^d})$
\begin{equation}\label{Eq:Rev_Trace}
	\| u \|^2_{L^2(\B^7)} \lesssim  \| \nabla u \|^2_{L^2(\B^7)} +  \|u\|^2_{L^2(\s^6)}.
\end{equation}
This follows simply from the divergence theorem. Namely,
\begin{equation}\label{Eq:Diver}
\int_{\mathbb{S}^6}|u(\omega)|^2d\sigma(\omega) =  \int_{\B^7}\text{div}\big(\xi |u(\xi)|^2\big)d\xi = \int_{\B^7} \big( 7|u(\xi)|^2 + \xi^i u(\xi) \overline{\partial_iu(\xi)} + \xi^i \overline{u(\xi)}  \partial_iu(\xi)\big) d\xi,
\end{equation}
and therefore
\begin{multline*}
	7\| u \|^2_{L^2(\B^7)} \leq \int_{\B^7} 2| \xi^i \Re \big(u(\xi) \overline{\partial_iu(\xi)}\big)|d\xi + \| u \|^2_{L^2(\mathbb{S}^6)}\\ \leq \int_{\B^7} 2 |\xi| |u(\xi)||\nabla u(\xi)| d\xi + \| u \|^2_{L^2(\mathbb{S}^6)}
	 \leq \| u \|^2_{L^2(\B^7)} + \| \nabla u \|^2_{L^2(\B^7)} + \| u \|^2_{L^2(\mathbb{S}^6)},
\end{multline*}
from which we get the estimate \eqref{Eq:Rev_Trace}. For $u \in C^2(\overline {\B^7})$ we then have that 
\[\| \partial_j u \|^2_{L^2(\B^7)} \lesssim   \int_{\B^7}  \partial_i \partial_j  u(\xi) \overline{\partial^i \partial_j u (\xi)} d\xi  +  \int_{\s^6} \partial_j u(\omega) \overline{ \partial_j u(\omega)} d\sigma(\omega)\]
for every $j = 1, \dots, 7$, and therefore 
\[\| \nabla u \|^2_{L^2(\B^7)} \lesssim   \int_{\B^7}  \partial_i \partial_j  u(\xi) \overline{\partial^i \partial^j u (\xi)} d\xi  +  \int_{\s^6} \partial_j u(\omega) \overline{ \partial^j u(\omega)} d\sigma(\omega).\]
Thus,
\[ \| u \|^2_{H^2(\B^7)} \lesssim \int_{\B^7}  \partial_i \partial_j  u(\xi) \overline{\partial^i \partial^j u (\xi)} d\xi  +  \int_{\s^6} \partial_j u(\omega) \overline{ \partial^j u(\omega)} d\sigma(\omega) +  \int_{\s^6} |u(\omega)|^2 d\sigma(\omega).\]
Similarly, for $u \in C^3(\overline {\B^7})$
\[\| \partial_j  \partial_k u \|^2_{L^2(\B^7)} \lesssim   \int_{\B^7}  \partial_i \partial_j   \partial_k u(\xi) \overline{\partial^i \partial_j  \partial_k u (\xi)} d\xi  +  \int_{\s^6} \partial_j  \partial_k u(\omega) \overline{ \partial_j  \partial_k u(\omega)} d\sigma(\omega),\]
which implies that 
\begin{align*}
 \int_{\B^7}  \partial_i \partial_j  u(\xi) \overline{\partial^i \partial^j u (\xi)} d\xi
  \lesssim \int_{\B^7}  \partial_i \partial_j   \partial_k u(\xi) \overline{\partial^i \partial^j  \partial^k u (\xi)} d\xi  +  \int_{\s^6} \partial_i  \partial_j u(\omega) \overline{ \partial^i  \partial^i u(\omega)} d\sigma(\omega).
\end{align*}
Hence, 
\begin{align*}
\| u \|^2_{H^3(\B^7)} &  \lesssim  \int_{\B^7}  \partial_i \partial_j   \partial_k u(\xi) \overline{\partial^i \partial_j  \partial_k u (\xi)} d\xi   +  \int_{\s^6} \partial_i  \partial_j u(\omega) \overline{ \partial^i  \partial^i u(\omega)} d\sigma(\omega) \\
&  +  \int_{\s^6} \partial_j u(\omega) \overline{ \partial^j u(\omega)} d\sigma(\omega) 
+  \int_{\s^6} |u(\omega)|^2 d\sigma(\omega),
\end{align*}
which proves that 
\[  \| \mb u \|^2_{H^3 \times H^2(\B^7)} = \|u_1\|^2_{H^3(\B^7)} + 
\|u_2\|^2_{H^2(\B^7)} \lesssim \|\mb u \|_{\mc H} \]
for all $\mb u = (u_1,u_2) \in C^3(\overline {\B^7}) \times C^2(\overline {\B^7})$.
The reverse inequality follows from the trace theorem, which simply follows from the divergence theorem \eqref{Eq:Diver}, and which asserts that 
\[ \int_{\s^6} |u(\omega)|^2 d \sigma(\omega) \lesssim \| u \|^2_{H^1(\B^7)}, \]
for all $u \in  C^1(\overline {\B^7})$,   
\end{proof}

By an approximation argument, the result extends to all of $H^3 \times H^2(\B^7)$, where the boundary integrals are understood in the sense of traces. 
This new inner product is tailor-made to derive the following estimate for the operator $\tilde{\mb L}$.

\begin{lemma}\label{Le:LuPhi_H}
We have
$\mathrm{Re} (  \tilde{\mb L} \mb u | \mb u )_{\mc H}  \leq  - \frac12 \| \mb u \|_{\mc H}^2$
for all $\mb u  \in \mc D(  \tilde{\mb L} )$.
\end{lemma}

\begin{proof}
In the following, we frequently use the identities
\begin{align}\label{Eq:Id1}
2 \mathrm{Re}[\xi^j \partial_j f(\xi) \overline f(\xi)]  = \partial_{\xi^j}[ \xi^j | f(\xi)|^2] - d |f(\xi)|^2
\end{align}
for $d=7$ and 
\begin{align*}
\partial_{i_1} \partial_{i_2} \dots  \partial_{i_k} [\xi^{i} \partial_i f(\xi) ]= k  \partial_{i_1} \partial_{i_2} \dots  \partial_{i_k} f(\xi) + \xi^i \partial_i  \partial_{i_1} \partial_{i_2} \dots  \partial_{i_k} f(\xi) 
\end{align*}
for all $k \in \N$. With this the divergence theorem implies that
\begin{align*}
\mathrm{Re} & \int_{\B^7} \partial_i \partial_j \partial_k [  \tilde{\mb L} \mb u]_1(\xi) \overline{\partial^i \partial^j \partial^k u_1(\xi)} d\xi   = - \frac12 
\int_{\B^7} \partial_i \partial_j \partial_k u_1(\xi)  \overline{\partial^i \partial^j \partial^k u_1(\xi) }d\xi  \\
&  +  \mathrm{Re}  \int_{\B^7} \partial_i \partial_j \partial_k u_2(\xi) \overline{\partial^i \partial^j \partial^k u_1(\xi)} d\xi 
 - \frac12
\int_{\mathbb S^6} \partial_i \partial_j \partial_k u_1(\omega) \overline{\partial^i \partial^j \partial^k u_1(\omega)} d\sigma(\omega) 
\end{align*}
Similarly, 
\begin{align*}
\mathrm{Re}&  \int_{\B^7} \partial_i \partial_j  [  \tilde{\mb L} \mb u]_2(\xi) \overline{\partial^i \partial^j  u_2(\xi)} d\xi   = - \frac12 
\int_{\B^7} \partial_i \partial_j  u_2(\xi)  \overline{\partial^i \partial^j  u_2(\xi) }d\xi \\
&  - \mathrm{Re}   \int_{\B^7} \partial_i \partial_j \partial_k u_2(\xi) \overline{\partial^i \partial^j \partial^k u_1(\xi)} d\xi 
  - \frac12 
\int_{\mathbb S^6} \partial_i \partial_j  u_2(\omega) \overline{\partial^i \partial^j u_2(\omega)} d\sigma(\omega) \\ & 
+ \mathrm{Re} 
\int_{\mathbb S^6} \omega^k \partial_k \partial_i \partial_j  u_1(\omega) \overline{\partial^i \partial^j u_2(\omega)} d\sigma(\omega)
\end{align*}
and 
\begin{align*}
 \mathrm{Re}  & \int_{\mathbb S^6}   \partial_i \partial_j   [  \tilde{\mb L} \mb u]_1(\omega)\overline{ \partial^i \partial^j u_1(\omega) }d\sigma(\omega) = - 3 \int_{\mathbb S^6}  \partial_i \partial_j   u_1(\omega)\overline{ \partial^i \partial^j u_1(\omega) }d\sigma(\omega) \\
& - \mathrm{Re}  \int_{\mathbb S^6} \omega^k \partial_k \partial_i \partial_j   u_1(\omega)\overline{ \partial^i \partial^j u_1(\omega) }d\sigma(\omega) + 
\mathrm{Re}  \int_{\mathbb S^6}  \partial_i \partial_j   u_2(\omega)\overline{ \partial^i \partial^j u_1(\omega) }d\sigma(\omega).
\end{align*}
Hence, 
\begin{align}\label{Eq:Est1}
\mathrm{Re} (  \tilde{\mb L} \mb u | \mb u )_1 = - \frac{1}{2} (\mb u | \mb u )_1 - 2 \int_{\mathbb S^6}  \partial_i \partial_j   u_1(\omega)\overline{ \partial^i \partial^j u_1(\omega) }d\sigma(\omega) + \int_{\mathbb S^6} A(\omega) d\sigma(\omega) 
\end{align}
for 
\begin{align*}
A(\omega)& =  - \frac12
 \partial_i \partial_j \partial_k u_1(\omega) \overline{\partial^i \partial^j \partial^k u_1(\omega)}  - \frac12
 \partial_i \partial_j  u_2(\omega) \overline{\partial^i \partial^j u_2(\omega)} 
 - \frac{1}{2}   \partial_i \partial_j   u_1(\omega)\overline{ \partial^i \partial^j u_1(\omega) }\\
&+ \mathrm{Re} \,
 \omega^k \partial_k \partial_i \partial_j  u_1(\omega) \overline{\partial^i \partial^j u_2(\omega)} - 
  \mathrm{Re}\,   \omega^k \partial_k \partial_i \partial_j   u_1(\omega)\overline{ \partial^i \partial^j u_1(\omega) }  + 
\mathrm{Re} \,   \partial_i \partial_j   u_2(\omega)\overline{ \partial^i \partial^j u_1(\omega) }.
\end{align*}
An application of the inequality
\[ \mathrm{Re}(a \overline{b}) + \mathrm{Re}(a \overline{c}) - \mathrm{Re}(b \overline{c}) \leq \tfrac{1}{2} ( |a|^2 + |b|^2 + |c|^2 ), \quad a,b,c \in \C, \]
shows that $A(\omega) \leq 0$. Analogously, one can show that 
\begin{align}\label{Eq:Est2}
\mathrm{Re} (  \tilde{\mb L} \mb u | \mb u )_2 = - \frac{1}{2} (\mb u | \mb u )_2 - 2 \int_{\mathbb S^6}  \partial_i  u_2(\omega)\overline{ \partial^i  u_2(\omega) }d\sigma(\omega) + \int_{\mathbb S^6} B(\omega) d\sigma(\omega) 
\end{align}
with $B(\omega) \leq 0$.
Next, we consider
\begin{align*}
 \mathrm{Re}  \int_{\mathbb S^6}  \partial_i [  \tilde{\mb L} \mb u]_1(\omega) \overline{ \partial^i  u_1(\omega)} d\sigma(\omega)& = 
 -2  \int_{\mathbb S^6}    \partial_i u_1(\omega) \overline{ \partial^i  u_1(\omega)} d\sigma(\omega)  \\
 & - \mathrm{Re} 
  \int_{\mathbb S^6}  \omega^k \partial_k  \partial_i u_1(\omega) \overline{ \partial^i  u_1(\omega)} d\sigma(\omega)
   + \mathrm{Re} 
  \int_{\mathbb S^6}  \partial_i u_2(\omega) \overline{ \partial^i  u_1(\omega)} d\sigma(\omega).
\end{align*}
The Cauchy-Schwarz inequality implies that 
\begin{align*}
 \mathrm{Re} & 
  \int_{\mathbb S^6}  [ \partial_i u_2(\omega)   - \omega^k \partial_k  \partial_i u_1(\omega)] \overline{ \partial^i  u_1(\omega)} d\sigma(\omega) \leq  
\frac{1}{2} \int_{\mathbb S^6}  \partial_i  u_1(\omega) \overline{ \partial^i  u_1(\omega)} d\sigma(\omega) \\
& +
 \int_{\mathbb S^6}  \partial_i  u_2(\omega) \overline{ \partial^i  u_2(\omega)} d\sigma(\omega) +
  \int_{\mathbb S^6}  \partial_i \partial_j u_1(\omega) \overline{ \partial^i \partial^j  u_1(\omega)} d\sigma(\omega)
 \end{align*}
such that 
\begin{align*}
 \mathrm{Re}   \int_{\mathbb S^6}  \partial_i [  \tilde{\mb L} \mb u]_1(\omega) \overline{ \partial^i  u_1(\omega)} d\sigma(\omega)& \leq 
 -\frac32  \int_{\mathbb S^6}    \partial_i u_1(\omega) \overline{ \partial^i  u_1(\omega)} d\sigma(\omega)  +  \int_{\mathbb S^6}  \partial_i  u_2(\omega) \overline{ \partial^i  u_2(\omega)} d\sigma(\omega) \\
 & +
  \int_{\mathbb S^6}  \partial_i \partial_j u_1(\omega) \overline{ \partial^i \partial^j  u_1(\omega)} d\sigma(\omega).
\end{align*}
Similarly, 
\begin{align*}
  \mathrm{Re}  \int_{\mathbb S^6}   [  \tilde{\mb L} \mb u]_2(\omega)  \overline{u_2(\omega) }d\sigma(\omega)  &  =  -2  \int_{\mathbb S^6}    |u_2(\omega)|^2 d\sigma(\omega)  +  \mathrm{Re} \int_{\mathbb S^6}    \Delta u_1 (\omega) \overline{  u_2(\omega)}  d\sigma(\omega)   \\
& - \mathrm{Re}
 \int_{\mathbb S^6}\omega^k \partial_k u_2(\omega) \overline{u_2(\omega)}   d\sigma(\omega) \\
 & 
  \leq -\frac32 \int_{\mathbb S^6}    |u_2(\omega)|^2 d\sigma(\omega)  +  \int_{\mathbb S^6}   | \Delta u_1 (\omega)|^2  d\sigma(\omega) +  \int_{\mathbb S^6} \partial_i u_2(\omega) \overline{ \partial^i u_2(\omega)}   d\sigma(\omega),
\end{align*}
and
\begin{align*}
 \mathrm{Re}  &  \int_{\mathbb S^6}    [  \tilde{\mb L} \mb u]_1(\omega)  \overline{u_1(\omega)} d\sigma(\omega)    = 
 -  \int_{\mathbb S^6}    |u_1(\omega)|^2 d\sigma(\omega)
 - \mathrm{Re} 
  \int_{\mathbb S^6}  \omega^k \partial_k  u_1(\omega) \overline{   u_1(\omega)} d\sigma(\omega)  \\
  &  + \mathrm{Re} 
  \int_{\mathbb S^6}  u_2(\omega) \overline{u_1(\omega)} d\sigma(\omega)
   \leq  - \frac12  \int_{\mathbb S^6}    |u_1(\omega)|^2 d\sigma(\omega)
  +  \int_{\mathbb S^6}    \partial_i u_1(\omega) \overline{  \partial^i u_1(\omega) } d\sigma(\omega)  + \int_{\mathbb S^6}    |u_2(\omega)|^2 d\sigma(\omega).
  \end{align*}
  
In view of Eqn.~\eqref{Eq:Est1} and \eqref{Eq:Est2} we obtain
\begin{align*}
  \mathrm{Re} (  \tilde{\mb L} \mb u | \mb u ) & = 4 \mathrm{Re} (  \tilde{\mb L} \mb u | \mb u )_1 +  \mathrm{Re} (  \tilde{\mb L} \mb u | \mb u )_2 + 
   \mathrm{Re}   \int_{\mathbb S^6}  \partial_i [  \tilde{\mb L} \mb u]_1(\omega) \overline{ \partial^i  u_1(\omega)} d\sigma(\omega)v\\
   & +
   \mathrm{Re}  \int_{\mathbb S^6}   [  \tilde{\mb L} \mb u]_2(\omega)  \overline{u_2(\omega) }d\sigma(\omega)
   + \mathrm{Re}    \int_{\mathbb S^6}    [  \tilde{\mb L} \mb u]_1(\omega)  \overline{u_1(\omega)} d\sigma(\omega)  \\ 
   & \leq  - 2 (\mb u | \mb u )_1 - \frac{1}{2} (\mb u | \mb u )_2  
  \\
  &    -\frac12  \int_{\mathbb S^6}    \partial_i u_1(\omega) \overline{ \partial^i  u_1(\omega)} d\sigma(\omega) -\frac12 \int_{\mathbb S^6}    |u_2(\omega)|^2d\sigma(\omega) -\frac12 \int_{\mathbb S^6}    |u_1(\omega)|^2 d\sigma(\omega) \\
     &  - 7 \int_{\mathbb S^6}  \partial_i \partial_j   u_1(\omega)\overline{ \partial^i \partial^j u_1(\omega) }d\sigma(\omega)   + \int_{\mathbb S^6}   | \Delta u_1 (\omega)|^2  d\sigma(\omega)
 \end{align*}
An application of the Cauchy-Schwarz inequality shows that 
\begin{align*}
  |\Delta u(\xi)|^2 =\left | \sum_{i=1}^{7} \partial^2_i u(\xi)\right |^2 \leq 7 \sum_{i=1}^{7} |\partial^2_i u(\xi)|^2 \leq   7 \sum_{i,j=1}^{7} |\partial_i \partial_j u(\xi)|^2 =  7 \partial_i \partial_j u(\xi) \overline{\partial^i \partial^j u(\xi)}.
\end{align*}
This finishes the proof.
\end{proof}

\begin{remark}\label{Rem:norm}
By arguments similar to the ones used above in the proof of Lemma \ref{Le:LuPhi_H} one can easily show that for the standard $\dot H^{k} \times \dot H^{k-1}(\B^7)$-seminorms one obtains
\begin{align}\label{Eq:dotH_bound}
\mathrm{Re} (  \tilde{\mb L} \mb u | \mb u )_{\dot H^{k} \times \dot H^{k-1}(\B^7)}  \leq  (\tfrac{5}{2} - k ) \| \mb u \|_{\dot H^{k} \times \dot H^{k-1}(\B^7)}
\end{align}
for all $\mb u  \in \mc D(  \tilde{\mb L} )$ and $k \geq 1$. In fact, this bound can be anticipated by recalling that $\tilde{\mb L}$ represents the free linear wave equation in different coordinates. Generic solutions of $\Box u(t,x) = 0$ satisfy  $\|u(t,\cdot) \|_{\dot H^{k}(\B^7_{T-t})}\lesssim 1$ for $k \geq 1$ by standard energy estimates.
For the rescaled variable defined in Eq.~\eqref{Eq:new_variable} this implies that 
\[\| \psi(-\log(T-t)+ \log T,\cdot) \|_{\dot H^{k}(\B^7)}  =  (T-t)^{-\frac{5}{2}+k} \| u(t,\cdot) \|_{\dot H^{k}(\B^7_{T-t})}\lesssim  (T-t)^{-(\frac{5}{2}-k)}, \]
which explains \eqref{Eq:dotH_bound}. In particular, in these seminorms one can show decay for $t \to T$ only if $k \geq 3$, which explains our choice of the function space. However, it is obvious that a naive use of the standard norm on $H^{3} \times H^{2}(\B^7)$ does not work, which is why we work instead with the equivalent norm $\| \cdot \|_{\mc H}$ that includes integrals over $\B^7$ only on the $\dot H^{3} \times \dot H^{2}-$level and lower order terms are substituted by suitable boundary terms. \\
Finally, we note that by the transformation \eqref{Eq:new_variable} the constant solution $u(t,x) = 1$ to the linear wave equation is transformed into $\psi(\tau,\rho) = T e^{-\tau}$. Hence, by increasing the regularity to $H^{k} \times H^{k-1}(\B^7)$ for $k \geq 4$ the anticipated optimal bound is
\begin{align*}
\mathrm{Re} (  \tilde{\mb L} \mb u | \mb u )_{H^{k} \times H^{k-1}(\B^7)}  \leq  - \| \mb u \|_{H^{k} \times H^{k-1}(\B^7)}.
\end{align*}
which cannot be improved further.
\end{remark}

\begin{lemma}\label{Le:LuPhi_DenseRange}
	Let $\lambda = \frac{5}{2}$. Then  
	\[\mathrm{rg}(\lambda - \tilde{\mb L}) \subset \mc H\]
	is a dense subset.
\end{lemma}
\begin{remark}
		For the anticipated application of this lemma it is enough to prove the statement for some $\la > -\frac{1}{2}$, and our particular choice might seem arbitrary at the first sight. We made it since we can thereby in part reduce our analysis to a problem that has already been treated elsewhere. This will become apparent from the proof.
	\end{remark}
	\begin{proof}
		We have to prove that for all $\mb f$ in a dense subset of $\mc H$ the equation
			$(\lambda - \tilde{\mb L}) \mb u = \mb f$
			is solvable in $\mc D(\tilde{\mb L})$.
		Let $\tilde {\mb f} \in \mc H$ and $\varepsilon > 0$ be arbitrary. By density, there is an $\mb f \in C^{\infty}(\overline{\mathbb B^{7}}) \times C^{\infty}(\overline{\mathbb B^{7}})$ such that $\|\mb  f - \tilde {\mb f} \| < \frac{\varepsilon}{2}$. Furthermore, for $n \in \N$, define $\mb f_n = (f_{1,n},  f_{2,n})$ by
		\begin{equation*}\label{Eq:Harmonics_N}
		f_{1,n} :=  \sum_{\ell = 0}^{n} P_{\ell} f_1, \quad f_{2,n} := \sum_{\ell = 0}^{n} P_{\ell} f_2.
		\end{equation*}
		Eq.~\eqref{Decomp:SpherHarm_Hk} implies that there is an $N  \in \N$ such that $\| \mb f_N - \mb f \| < \frac{\varepsilon}{2}$. Hence, it suffices to consider the equation
		\begin{equation}\label{Eq:Eigenv_N}
		(\lambda - \tilde{\mb L}) \mb u = \mb f_N 
		\end{equation}
		and to construct a solution $\mb u \in \mc D(\tilde{\mb L})$. 
			The dense subset of $\mc H$ we implicitly defined is convenient for the following reason. For $\mb f$ whose coordinate functions are finite sums of spherical harmonics the equation  $(\lambda - \tilde{\mb L}) \mb u = \mb f$ can be decoupled into finitely many ODEs, all of which are of the hypergeometric type and can therefore be solved by existing ODE methods. Now, Eq.~\eqref{Eq:Eigenv_N} is equivalent to two equations satisfied by its coordinate functions.
		Namely, for $u_2$, we obtain
		\begin{align}\label{Eq:DenseRg1}
		u_2(\xi) = \xi^i \partial_i u_1(\xi) + (\lambda + 1) u_1(\xi) - f_{1,N}(\xi),
		\end{align}
		while $u_1$ satisfies the degenerate elliptic problem
		\begin{align}\label{Eq:Elliptic}
		- (\delta^{ij} - \xi^i \xi^j) \partial_i \partial_j u_1(\xi) + 2(\lambda +2) \xi^i \partial_i u_1(\xi) + (\lambda + 1)(\lambda + 2) u_1(\xi) = g_N(\xi),
		\end{align}
		for $g_N \in C^{\infty}(\overline{\B^7})$ given by
		\[ g_N(\xi) = \xi^i \partial_i f_{N,1}(\xi) + (\lambda +2) f_{N,1}(\xi) + f_{N,2}(\xi). \]
		For $\lambda = \frac{5}{2}$, Eq.~\eqref{Eq:Elliptic} reduces to 
		\begin{align}\label{Eq:Elliptic1}
		- (\delta^{ij} - \xi^i \xi^j) \partial_i \partial_j u_1(\xi) + 9 \xi^i \partial_i u_1(\xi) + \frac{63}{4} u_1(\xi) = g_N(\xi)
		\end{align}
		and by introducing polar coordinates $\rho = |\xi|$, $\omega = \frac{\xi}{|\xi|}$ the right hand side can be written as   
		\[g_N(\rho \omega) =  \sum_{\ell = 0}^{N} \sum_{m  \in \Omega_{\ell} } g_{\ell,m}(\rho) Y_{\ell,m}(\omega),  \]
		with $g_{\ell,m} \in C^{\infty}[0,1]$.
			We then try a solution ansatz of the same form
		\begin{align}\label{Eq:DenseRg1_2}
		u_1(\rho \omega) =  \sum_{\ell = 0}^{N} \sum_{m  \in \Omega_{\ell} } u_{\ell,m}(\rho) Y_{\ell,m}(\omega).  
		\end{align}
			In addition, from the following relation
			\begin{equation}\label{Eq:Partial_polar_form}
			\partial_{\xi^i} = \frac{\delta_i^j-\omega^j \omega_i}{\rho}\partial_{\omega^j} + \omega_i \partial_{\rho}
			\end{equation}
			we get the polar form of the differential operator in Eq.~\eqref{Eq:Elliptic1}
			\begin{equation}\label{Eq:PolarForm}
			- (\delta^{ij} - \xi^i \xi^j) \partial_{\xi^i} \partial_{\xi^j} + 9 \xi^i \partial_{\xi^i} = -(1- \rho^2) \partial_{\rho}^2 - \frac{6}{\rho}\partial_{\rho} +9\rho\partial_\rho - \frac{1}{\rho^2}\triangle_\omega^{\mathbb{S}^6}.
			\end{equation}
			Here $\triangle_\omega^{\mathbb{S}^6}$ is the Laplace-Beltrami operator on the 6-sphere, namely
			\begin{equation*}
			\triangle_\omega^{\mathbb{S}^6}=(\delta^{ij}-\omega^i \omega^j) \partial_{\omega^i} \partial_{\omega^j} - 6\omega^i\partial_{\omega^i}.
			\end{equation*}
			Now, by means of Eq.~\eqref{Eq:PolarForm}, the fact that  $Y_{\ell,m}$ is an eigenfunction of $-\triangle_\omega^{\mathbb{S}^6}$ with eigenvalue $\ell(\ell+5)$, and the fact that $Y_{\ell,m}$ are orthogonal to each other, Eq.~\eqref{Eq:Eigenv_N}
		decouples into a system of ODEs 
		\begin{align}\label{Eq:DenseRg2}
		\left [-(1- \rho^2) \partial_{\rho}^2 - \frac{6}{\rho} \partial_{\rho} + 9 \rho \partial_{\rho} + \frac{\ell(\ell + 5)}{\rho^2} + \frac{63}{4} \right ]u_{\ell,m}(\rho) = g_{\ell,m}(\rho),
		\end{align}
		for $\ell = 0, \dots, N$ and $m \in \Omega_{\ell}$ (however, note that the coefficients only depend on $\ell$).
		We set
		\begin{equation}\label{Def:v_lm}
			v_{\ell,m}(\rho) := \rho^2  u_{\ell, m}(\rho)
		\end{equation} 
	such that  Eq.~\eqref{Eq:DenseRg2} transforms into
		\begin{align}\label{Eq:DenseRg3}
		\left [-(1- \rho^2) \partial_{\rho}^2 - \frac{2}{\rho} \partial_{\rho} + 5 \rho \partial_{\rho} + \frac{(\ell+2)(\ell + 3)}{\rho^2} + \frac{15}{4} \right ]v_{\ell,m}(\rho) = \rho^2 g_{\ell,m}(\rho).
		\end{align} 
			Problems of this type have already been treated in
			\cite{DonZen14}, \cite{DonSch14b} and \cite{ChaDon19}, and we borrow some results from there. In particular, from \cite{ChaDon19} we know that the homogeneous version of Eq.~\eqref{Eq:DenseRg3} has a fundamental system of solutions $\{ \psi_{\ell,0},\psi_{\ell,1}  \}$ where $\psi_{\ell,j}(\rho)=\rho^{\ell+2}\phi_{\ell,j}(\rho^2)$ for $j \in \{0,1  \}$ and
			\begin{align}
			\phi_{\ell,0}(z)&=\tfrac{1}{\sqrt{1-z}}\left(\tfrac{2}{1+\sqrt{1-z}}\right)^{\frac{5}{2}+\ell}, \label{Eq:Expl1}\\
			\phi_{\ell,1}(z)&= \tfrac{1}{\sqrt{1-z}}\left[\left(\tfrac{1}{1-\sqrt{1-z}}\right)^{\frac{5}{2}+\ell}-\left(\tfrac{1}{1+\sqrt{1-z}}\right)^{\frac{5}{2}+\ell}\right].\label{Eq:Expl2}
			\end{align}
			Furthermore, the Wronskian is $W(\psi_{\ell,0},\psi_{\ell,1})(\rho) = C_\ell ( 1- \rho^2)^{-\frac32} \rho^{-2}$ for some non-zero $C_\ell$.
			Expressions \eqref{Eq:Expl1} and \eqref{Eq:Expl2} are obtained by first expressing Eq.~\eqref{Eq:DenseRg3} in a hypergeometric form, and then using explicit solution formulas, e.g.~from \cite{NIST10}. Of course, unrelated to the origin of these expressions, one can check by a straightforward calculation that $\psi_{\ell,0},\psi_{\ell,1}$ are indeed solutions. Now, by the variation of constants formula we obtain a solution to Eq.~\eqref{Eq:DenseRg3}
			\begin{align}
			v_{\ell,m}(\rho) &= - \psi_{\ell,0}(\rho) \int_{\rho}^{1} \frac{\psi_{\ell,1}(s)}{W(\psi_{\ell,0},\psi_{\ell,1})(s)} \frac{s^2 g_{\ell,m}(s)}{1-s^2} ds 
			-  \psi_{\ell,1}(\rho)  \int_{0}^{\rho} \frac{\psi_{\ell,0}(s)}{W(\psi_{\ell,0},\psi_{\ell,1})(s)} \frac{s^2 g_{\ell,m}(s)}{1-s^2} ds \nonumber\\
			&=- \psi_{\ell,0}(\rho) \int_{\rho}^{1} \psi_{\ell,1}(s)\sqrt{1-s}h_{\ell,m}(s) ds 
			-  \psi_{\ell,1}(\rho)  \int_{0}^{\rho} \psi_{\ell,0}(s)\sqrt{1-s}h_{\ell,m}(s) ds, \label{Eq:Nonhom_sol}
			\end{align}
			where $h_{\ell,m} \in C^\infty([0,1])$. We claim that for $u_{\ell,m}$ defined by Eq.~\eqref{Def:v_lm} the function $u_1$ defined by Eq.~\eqref{Eq:DenseRg1_2} belongs to $C^\infty(\overline{\B^d})$. To show this, we first prove that $v_{\ell,m} \in C^\infty((0,1])$. To this end, we express $v_{\ell,m}$ a bit differently.
			First, note that the set of Frobenius indices of Eq.~\eqref{Eq:DenseRg3} at $\rho=1$ is $\{-\tfrac{1}{2},0\}$. Therefore, since $\psi_{\ell,1}$ is the analytic Frobenius solution at $\rho=1$, the other Frobenius solution has the following form $(1-\rho)^{-1/2}\psi_{\ell,2}(\rho)$ for some $\psi_{\ell,2}$ which is analytic at $\rho=1$. Furthermore, by linearity we have that
			\begin{equation}\label{Eq:Frob_sol1}
			\psi_{\ell,0}(\rho)= c_{\ell,1}\psi_{\ell,1}(\rho) + c_{\ell,2}\frac{\psi_{\ell,2}(\rho)}{\sqrt{1-\rho}}
			\end{equation}
			for some constants $c_{\ell,1},c_{\ell,2}$.
			Also, note that the second integral in Eq.~\eqref{Eq:Nonhom_sol} converges as $\rho \rightarrow 1^-$, and we denote its value by $\alpha_{\ell,m}$. Now, by using this and substituting Eq.~\eqref{Eq:Frob_sol1} in Eq.~\eqref{Eq:Nonhom_sol} we get
			\begin{align*}
			v_{\ell,m}(\rho) =
			-c_{\ell,2} \frac{\psi_{\ell,2}(\rho)}{\sqrt{1-\rho}} &\int_{\rho}^{1} \psi_{\ell,1}(s)\sqrt{1-s}h_{\ell,m}(s) ds \\
			&-\alpha_{\ell,m}\psi_{\ell,1}(\rho)
			+  c_{\ell,2}\psi_{\ell,1}(\rho)  \int_{\rho}^{1} \psi_{\ell,2}(s)h_{\ell,m}(s) ds,
			\end{align*}
			which is smooth near $\rho=1$. Indeed, the second and the third term are manifestly so, while for the first one this can be easily seen by means of substitution $s=\rho+(1-\rho)t$. Namely, in this way, the first term becomes
			\begin{equation}
			-c_{\ell,2}\psi_{\ell,2}(\rho)(1-\rho) \int_{0}^{1}\psi_{\ell,1}\big(t+(1-t)\rho\big)h_{\ell,m}\big(t+(1-t)\rho\big)\sqrt{1-t}dt;
			\end{equation}
			note that for $\rho$ away from zero, one can differentiate under the integral sign and smoothness up to $\rho=1$ follows. As a consequence, from Eq.~\eqref{Eq:DenseRg1_2} we have that 
			\begin{equation}\label{Eq:Regul_away_from_0}
			u_1 \in C^\infty(\overline{\B^7}\setminus \{ 0\}),
			\end{equation} with $u_1$ solving Eq.~\eqref{Eq:Elliptic1} away from zero classically.
			To understand the behavior of $u_{1}$ near zero we look back at Eq.~\eqref{Eq:Nonhom_sol}. Based on the asymptotic behavior of the integrands, we conclude that both $u_{\ell,m}$ and $u_{\ell,m}'$ are bounded near zero. Therefore, based on 
			Eqs.~\eqref{Eq:DenseRg1_2} and \eqref{Eq:Partial_polar_form} we have that $u_1 \in H^1(\B^7)$. In addition, $u_1$ is a weak solution to Eq.~\eqref{Eq:Elliptic1} on $\B^7$, and by elliptic regularity we conclude that $u_1 \in C^\infty(\B^7)$. This together with \eqref{Eq:Regul_away_from_0} and \eqref{Eq:DenseRg1} yields $u_1,u_2 \in C^\infty(\overline{\B^7})$, which in turn implies that the solution to Eq.~\eqref{Eq:Eigenv_N} we constructed belongs to $\mc D(\tilde{\mb L})$, and this finishes the proof.
\end{proof}

\subsubsection*{The free time-evolution}
In view of Lemma \ref{Le:LuPhi_H}, Lemma \ref {Le:LuPhi_DenseRange}, the Lumer-Phillips Theorem \cite{engel}, p.~83, Theorem 3.15 and the equivalence of norms, we obtain the following result for the free time-evolution. 

\begin{proposition}\label{Prop:FreeEvol}
The operator $\tilde{\mb L}: \mc D(  \tilde{\mb L} ) \subset \mc H \to \mc H$ is closable and its closure $(\mb L , \mc D( \mb L))$ generates a strongly-continuous one-parameter semigroup $\mb S: [0,\infty) \to \mc B(\mc H)$ satisfying
\[ \|\mb S(\tau)\| \leq M e^{-\frac{1}{2} \tau}\]
for all $\tau \geq 0$ and some constant $M > 1$. This implies that the spectrum of  $\mb L$ is contained in a left half-plane,
\[\sigma(\mb L) \subset \{\lambda \in \C: \mathrm{Re}\la \leq -\tfrac12 \} \]
and the resolvent is bounded by 
\[ \| \mb R_{\mb L}(\la) \| \leq \frac{M}{ \mathrm{Re} \la +  \frac12 }  \]
for all $\la \in \C$ with $\mathrm{Re} \la > -\frac12$.
\end{proposition}

For the last two statements see \cite{engel}, p.~55, Theorem 1.10. 

\subsection{Perturbations of the free evolution}
In the following, we define $\mb L'_a$ according to Eq.~\eqref{Def:Potential}. We fix $\delta^* > 0$ such that $\psi^*_a$ is smooth for all $a \in \overline{\mathbb B^7_{\delta^*}}$. For the rest of the paper, we assume that $0 < \delta < \delta^*$. 

\begin{lemma}\label{Le:Perturbation}
Let $\delta > 0$ be sufficiently small, then for every $a \in \overline{\mathbb{B}^7_{\delta}}$ the operator $\mb L'_{a}: \mc H \to \mc H$
is compact. Furthermore, the family of operators $\mb L'_{a}$ s uniformly bounded with respect to $a \in \overline{\mathbb{B}^7_{\delta}}$ and Lipschitz continuous, i.e., there is a $K > 0$ such that 
\[ \|\mb L'_{a} - \mb L'_{b} \| \leq K |a-b| \]
for all $a,b  \in \overline{\mathbb{B}^7_{\delta}}$.
\end{lemma}

\begin{proof}
The uniform boundedness of $\mb L'_{a}$ follows from the smoothness of the potential $V_a \in C^{\infty}(\overline{ \B^7})$ with respect to $a$ and the assumption that $a \in \overline{\mathbb B^7_{\delta^*}}$.
For the Lipschitz estimate we use that 
\[ \| V_a - V_b \|_{W^{2,\infty}(\B^7)} \lesssim | a - b | \]
for all $a,b  \in \overline{\mathbb B^7_{\delta}}$. Hence, 
\[ \|(V_a - V_b)u_1 \|_{H^2(\B^7)} \lesssim |a-b| \|u_1\| _{H^2(\B^7)}  \lesssim |a-b| \|u_1\| _{H^3(\B^7)}  \]
for all $\mb u \in \mc H$. For fixed $a \in \overline{\mathbb{B}^7_{\delta}}$, the compactness of the operator $\mb L'_{a}$ is a consequence of the compact embedding $H^{3}(\B^7) \hookrightarrow H^{2}(\B^7)$.
\end{proof}

As a consequence of the Bounded Perturbation Theorem, see \cite{engel} p.~158, we obtain the following result. 
 
\begin{corollary}\label{Cor:TimeEvol_La}
Let $\delta > 0$ be sufficiently small. For any $a \in \overline{\mathbb{B}^7_{\delta}}$, the operator   
\[\mb L_a := \mb L + \mb L_a'. \quad  \mc D(\mb L_a) = \mc D( \mb L) \subset \mc H \to \mc H\] generates a strongly-continuous one-parameter semigroup $\mb S_a: [0,\infty) \to \mc B(\mc H)$.
Furthermore, 
\[ \| \mb L_a - \mb L_b \| \leq K |a - b| \]
for all $a,b  \in \overline{\mathbb{B}^7_{\delta}}$ and $K > 0$ as in Lemma \ref{Le:Perturbation}
\end{corollary} 

Unfortunately, the growth estimate obtained from this abstract argument is too weak for our purpose. Hence, we analyze the spectrum of the generator $\mb L_a$.  First, we consider the case $a = 0$. The general case will then be treated perturbatively. 

\section{ODE Analysis}\label{Sec:ODE_analyis}

For $a=0$, the potential is radial and it will be shown in Section \ref{Sec:Spetrum_L0} that the eigenvalue problem can be reduced to ODE problems by decomposition
into spherical harmonics. We define the formal differential operator
\begin{align}\label{Eq:ODE_L0}
	\begin{split}
		\mc T_{\ell}(\lambda) f(\rho) := (1-\rho^2) f''(\rho) + & \left(\tfrac{6}{\rho}   - 2(\lambda +2) \rho \right) f'(\rho) \\
		&   - \left( (\lambda + 1)(\lambda + 2) + \tfrac{\ell ( \ell +5)}{\rho^2} - \tfrac{48}{(1+\rho^2)^2} \right )f(\rho) 
	\end{split}
\end{align}

We are furthermore interested in the unstable eigenvalues of $\mb L_0$, namely the ones that belong to $\Hb:= \{ z \in \mathbb{C}: \Re z \geq 0 \}$. In this case, as it will be made clear in the proof of Proposition \ref{Prop:Spectrum_L0}, due to elliptic regularity and the asymptotic behavior of eigenfunctions at $\rho=1$  it suffices to study solutions of Eq.~\eqref{Eq:ODE_L0} that are smooth on $[0,1]$. Therefore, for each $\ell \in \N_0$, we define the set 
\begin{align*}
	\Sigma_{\ell}  := \{ \lambda \in \C: \mathrm{Re \la \geq 0} \text{ and there exist }  f_{\ell}(\cdot;\lambda) \in C^{\infty}[0,1] \text { satisfying } \mc T_{\ell}(\lambda) f_{\ell}(\cdot;\lambda) = 0 \}.
\end{align*}

The investigation of the structure of $\Sigma_{\ell}$ is based on a refinement of the approach developed in \cite{CosDonGlo17}. Since different strategies are required for different values of $\ell$, two results will be proved in the following. 

\begin{proposition}\label{prop:l >1_l=0}
	For all $\ell \geq 2$,  $\Sigma_{\ell} = \emptyset$. Furthermore, for $\ell = 0$,  $\Sigma_{0} = \{1,3\}$ with unique solutions 
	\begin{equation*}
		f_0(\rho;1) =\frac{1-\rho^2}{(1+\rho^2)^2} \quad \text{and} \quad f_0(\rho;3)=\frac{1}{(1+\rho^2)^2}.
	\end{equation*}
\end{proposition}

\begin{proposition}\label{prop:l=1}
	For $\ell =1$, we have $\Sigma_{1} = \{0,1\}$ with unique solutions 
	\begin{equation*}
		f_1(\rho;0) =\frac{3\rho-\rho^3}{(1+\rho^2)^2} \quad \text{and} \quad  f_1(\rho;1)  =\frac{\rho}{(1+\rho^2)^2}.
	\end{equation*} 
\end{proposition}

\subsection{Discussion of our approach and some heuristics}	
We start with an observation that solutions to Eq.~\eqref{Eq:ODE_L0} which are smooth on $[0,1]$ are in fact analytic there. Analyticity on $(0,1)$ follows from the analyticity of coefficient functions in the standard normalized form of Eq.~\eqref{Eq:ODE_L0}. Furthermore, $\rho=0$ and $\rho=1$ are regular singular points and, by Frobenius theory, solutions near these points are given as linear combinations of Frobenius solutions.

At $\rho=0$ the Frobenius solutions of Eq.~\eqref{Eq:ODE_L0} have the following forms
\begin{equation}
	f_{0,1}(\rho) = \rho^\ell \sum_{n=0}^{\infty}a_1(n)\rho^{2n} \quad \text{and} \quad f_{0,2}(r) = C\log(r)f_{0,1}(\rho) + \rho^{-\ell-5} \sum_{n=0}^{\infty}a_2(n)\rho^{2n}, 
\end{equation}
for $a_1(0)=a_2(0)=1$ and a certain (possibly vanishing) parameter $C$. Consequently, due to the singular behavior of $f_{0,2}$ at the origin,  analytic solutions at $\rho=0$ are precisely constant multiples of $f_{0,1}$. On the other hand, if $\la \in  \Hb \setminus \{ 0,1,2 \}$ then the Frobenius solutions at $\rho=1$ have the following forms\footnote{For $\la \in \{ 0,1,2 \} $ the problem has to be treated case by case, but this is done in  similar manner.}
\begin{equation}
	f_{1,1}(\rho) = \sum_{n=0}^{\infty}b_1(n)(1-\rho)^{n} \quad \text{and} \quad f_{1,2}(\rho) = C\log(r)f_{1,1}(\rho) + \rho^{2-\la} \sum_{n=0}^{\infty}b_2(n)(1-\rho)^{n}, 
\end{equation}
where $b_1(0)=b_2(0)=1$ and $C$ is a parameter which is (possibly) non-zero only if $\la \in \mathbb{N}$. Then, since Eq.~\eqref{Eq:ODE_L0} is linear we have that
\begin{equation}
	f_{0,1}= c_1(\la)f_{1,1} + c_2(\la)f_{1,2}.
\end{equation}
Furthermore, since $f_{1,2}$ is singular at $\rho=1$, $\la \in \Sigma_\ell$ precisely when $c_2(\la)=0$. Finding zeros of the connection coefficient $c_2$ falls into the category of the so-called connection problems in ODE theory. At the moment, these problems are resolved only for equations with at most three regular singular points, while Eq.~\eqref{Eq:ODE_L0} has six of them. Their number can in fact be reduced to four, and we do this later on, however this is still of no help in the context of using the existing theory to solve the connection problem.

We therefore undertake a different approach. Namely, we study $f_{0,1}$ only, and look for values of $\la$ for which this function can be analytically continued past $\rho=1$. One could naively look for $\la$ for which the series defining $f_{0,1}$ has radius of convergence larger than one. This would be fine only if $\rho=1$ was the only singular point on the unit circle. However, there are three more, $\rho=\pm i$ and $\rho=-1$. Nevertheless, remarkably, there is a change of variables that moves the problematic singularities away, while preserving the analyticity of solutions at $\rho=0$ and $\rho=1$, see \eqref{Def:Heun_change_var} below. In particular, this leads to an ``isospectral" equation with four regular singular points, $0,$ $1$, $2$ and $\infty$, see \eqref{Eq:ODEheun}. 

For the new equation we generate the normalized series solution which is analytic at the origin
\begin{equation*}\label{Eq:Exspa}
	y_{\ell,\la}(x)=\sum_{n=0}^{\infty}a_n(\ell,\la) x^n, \quad a_0(\ell,\la)=1.
\end{equation*}
In this new setting, $\la \in \Sigma_\ell$ precisely when the radius of convergence of this series is larger than one. We then naturally approach this problem through the analysis of the asymptotics of the series coefficients. Namely, we study the limiting behavior of the quotient
\begin{equation*}
	r_n(\ell,\la):=\frac{a_{n+1}(\ell,\la)}{a_{n}(\ell,\la)}.
\end{equation*} 
Since $y_{\ell,\la}$ solves a differential equation with four regular singular points (in its canonical form) the sequence of coefficients $a_n$ satisfies a three term recurrence relation 
\begin{equation}\label{Eq:RecRel_outline}
	a_{n+2}(\ell,\la)=A_n(\ell,\la)a_{n+1}(\ell,\la)+B_n(\ell,\la)a_{n}(\ell,\la),
\end{equation}
where, in this particular case, $A_n$ and $B_n$ have rational function forms. 
Subsequently, for the quotient $r_n$ the following nonlinear relation holds
\begin{equation}\label{Eq:RecRel2_outline}
	r_{n+1}(\ell,\la)=A_n(\ell,\la)+\frac{B_n(\ell,\la)}{r_n(\ell,\la)}.
\end{equation}
Now we concentrate on the $\ell \geq 2$ case, where $\Sigma_{\ell} = \emptyset$ is to be shown. For $\ell\in \{ 0,1 \}$, on the other hand, the set $\Sigma_{\ell}$ is non-empty, and in this case we first conveniently ``remove" its elements and then follow the $\ell \geq 2$ approach. We will explain the ``removal" process in the proofs of Propositions \ref{prop:l >1_l=0} and \ref{prop:l=1}.

Now, from the theory of difference equations, we conclude that if $a_n(\ell,\la)$ is not eventually zero (i.e., $y_{\ell,\la}$ is not a polynomial) then $\lim_{n \rightarrow \infty}r_{n}(\ell,\la)$ exists and it is equal to either $1$ or $1/2$. Hence, analyticity at $x=1$ fails precisely when
\begin{equation}\label{Eq:Lim_rn}
	\lim_{n \rightarrow \infty}r_{n}(\ell,\la)=1,
\end{equation}
and our objective is therefore to prove that this holds for all $\ell \geq 2$ and $\la \in \Hb$.

Unfortunately, the explicit form of $r_{n}(\ell,\la)$ becomes increasingly complicated even for small values of $n$, and proving that $\lim_{n \rightarrow \infty}r_{n}(\ell,\la)=1$ is quite difficult even for a specific choice of $\ell,\la$, and much more so for all $\ell \geq 2$, $\la \in \Hb$. 
We nevertheless circumvent this difficulty in the following way. First, we construct a (simple) approximate solution $\tilde{r}_n$ to Eq.~\eqref{Eq:RecRel2_outline}, whose limiting value is 1. Then by showing that the actual solution $r_n$ is in a neighborhood of $\tilde{r}_n$, we conclude \eqref{Eq:Lim_rn}.
%
To construct an approximation $\tilde{r}_n(\ell,
\la)$ that emulates $r_n(\ell,\la)$ globally in $\ell$ and $\la$, we carefully analyze the asymptotics of $r_n(\ell,\la)$ in different parameter regimes and then model this behavior by simple rational function expressions. In particular, for fixed $\la$, $r_n(\ell,\la)$ is linear in $\ell$, and for fixed $\ell$ it is quadratic in $\la$. With some more exploration of the behavior of $r_n(\ell,\la)$ for fixed $\ell,\la$ and large $n$, we end up with the following approximation
\begin{equation}\label{Eq:QuasiSol_outline}
	\tilde{r}_n(\ell,\la)=\frac{\la^2}{4(n+1)(2n+2\ell+7)}+\frac{(4n+2\ell+5)\la}{2(n+1)(2n+2\ell+7)}+\frac{n-1}{n+1}+\frac{3\ell}{8(n+1)}.
\end{equation}

Note that $\lim_{n \rightarrow \infty}\tilde{r}_n(\ell,\la)=1$, which is the behavior we want to prove for $r_n$. Furthermore, by some basic complex function theory and elementary calculations we show that the relative difference $r_n(\ell,\la)/\tilde{r}_n(\ell,\la)-1$
is globally small. This rules out the possibility of $\lim_{n \rightarrow \infty} r_n(\ell,\la)=1/2$, and the desired claim follows.
\subsection{Proof of Proposition \ref{prop:l >1_l=0}} 

\subsubsection{Preliminary transformations} 

We study solutions of the equation  
\begin{align}\label{Eq:ODEeigenv}
	\begin{split}
		(1-\rho^2) f''(\rho) + & \left(\tfrac{6}{\rho}   - 2(\lambda +2) \rho \right) f'(\rho) \\
		&   - \left( (\lambda + 1)(\lambda + 2) + \tfrac{\ell ( \ell +5)}{\rho^2} - \tfrac{48}{(1+\rho^2)^2} \right )f(\rho)  = 0
	\end{split}
\end{align}
First, we transform Eq.~\eqref{Eq:ODEeigenv} into an ``isospectral" equation\footnote{Isospectral in the sense that the set of values of $\la \in\overline{\mathbb{H}}$ that yield $C^{\infty}[0,1]$ solutions is the same for both equations.} with four regular singularities. This is achieved by a change of variables
\begin{equation}\label{Def:Heun_change_var}
	\rho=\sqrt{\tfrac{x}{2-x}} \quad \text{and} \quad f(\rho)=x^{\frac{\ell}{2}}(2-x)^{\frac{\la+1}{2}}y(x).
\end{equation}
Thereby, Eq.~\eqref{Eq:ODEeigenv} is transformed into a Heun equation in its canonical form, see \cite{NIST10},
\begin{multline}\label{Eq:ODEheun} 
	y''(x)+\left(\frac{2\ell+7}{2x}+\frac{\la-1}{x-1}+\frac{1}{2(x-2)}\right)y'(x)\\+\frac{(\la+\ell+7)(\la+\ell-3)x-3\ell^2-4(\la+3)\ell-(\la+13)(\la-3)}{4x(x-1)(x-2)}y(x)=0.
\end{multline}
By Frobenius theory, any $y_{\ell,\la}\in C^{\infty}[0,1]$ that solves Eq.~\eqref{Eq:ODEheun} is in fact analytic on $[0,1]$. Furthermore, the Frobenius indices of Eq.~\eqref{Eq:ODEheun} at $x=0$ are $s_1=0$ and $s_2=-\ell-\frac{5}{2}$. Therefore, without loss of generality, we can assume that $y_{\ell,\la}(x)$ has the following expansion
\begin{equation}\label{Eq:Expansion}
	y_{\ell,\la}(x)=\sum_{n=0}^{\infty}a_n(\ell,\la) x^n, \quad a_0(\ell,\la)=1,
\end{equation} 
near $x=0$. Since the finite singular points of Eq.~\eqref{Eq:ODEheun} are $x=0$, $x=1$ and $x=2$, $y_{\ell,\la}(x)$ fails to be analytic at $x=1$ precisely when the radius of convergence of series \eqref{Eq:Expansion} is equal to one. We therefore consider the sequence of coefficients $\{a_n(\ell,\la)\}_{n\geq 0}$ and aim to show that
\[ \lim\limits_{n \ra \infty}\frac{a_{n+1}(\ell,\la)}{a_n(\ell,\la)}=1, \]
whenever the ratio inside the limit is defined for large $n$.
To that end, we first derive the recurrence relation for coefficients $a_n.$ Namely, by putting \eqref{Eq:Expansion} into Eq.~\eqref{Eq:ODEheun} we obtain
\begin{equation}\label{Eq:RecRel}
	a_{n+2}(\ell,\la)=A_n(\ell,\la)a_{n+1}(\ell,\la)+B_n(\ell,\la)a_{n}(\ell,\la),
\end{equation}
where
\begin{gather*}
	A_n(\ell,\la)=\frac{12n^2+4(3\ell+2\la+12)n+\la^2+2(2\ell+9)\la+3(\ell^2+8\ell-1)}{4(2n+2\ell+9)(n+2)}, \\ B_n(\ell,\la)=-\frac{(\ell+\la+2n+7)(\ell+\la+2n-3)}{4(2n+2\ell+9)(n+2)},
\end{gather*}
and the initial condition is $a_{-1}(\ell,\la)=0, a_0(\ell,\la)= 1$. Now for every nonnegative integer $n$ we define 
\begin{equation*}
	r_n(\ell,\la):=\frac{a_{n+1}(\ell,\la)}{a_{n}(\ell,\la)}.
\end{equation*} 
Since $\lim_{n\ra\infty}A_n(\ell,\la)=\frac{3}{2}$ and $\lim_{n\ra\infty}B_n(\ell,\la)=-\frac{1}{2}$, the so-called characteristic equation of Eq.~\eqref{Eq:RecRel} is 
\begin{equation}\label{Eq:CharEq}
	t^2-\frac{3}{2}t+\frac{1}{2}=0.
\end{equation}
Furthermore, $t_1=\frac{1}{2}$ and $t_2=1$ are the solutions to Eq.~\eqref{Eq:CharEq} and by Poincar\'e's theorem (see Appendix~\ref{App:Results}) we conclude that given $\la\in\Hb$, either $a_n(\ell,\la)=0$ eventually in $n$ or 
\begin{equation}\label{Eq:Lim1}
	\lim\limits_{n\ra\infty}r_n(\ell,\la)=1
\end{equation}
or
\begin{equation}\label{Eq:Lim2}
	\lim\limits_{n\ra\infty}r_n(\ell,\la)=\tfrac{1}{2}.
\end{equation}
We now proceed by separately treating the cases $\ell\geq2$ and $\ell=0$.
\subsubsection{The case $\ell \geq 2$}

We claim that $a_n(\ell,\la)$ does not equal zero eventually in $n$. Indeed, assumptions that $a_n(\ell,\la)=0$ for large $n$ and $a_{-1}(\ell,\la)=0$ imply (by backwards substitution in Eq.~\eqref{Eq:RecRel}) that $a_0(\ell,\la)=0$, which is in contradiction with the assumption that $a_0(\ell,\la)=1$. Therefore, either \eqref{Eq:Lim1} or \eqref{Eq:Lim2} is true. We claim that Eq.~\eqref{Eq:Lim1} holds for all $\la\in\Hb$.  To establish this, we first derive
from Eq.~\eqref{Eq:RecRel} the recurrence relation for $r_n$. Namely, we have
\begin{equation}\label{Eq:RecRel2}
	r_{n+1}(\ell,\la)=A_n(\ell,\la)+\frac{B_n(\ell,\la)}{r_n(\ell,\la)},
\end{equation}
where 
\begin{equation}\label{Eq:InitCond}
	r_0(\ell,\la)=A_{-1}(\ell,\la)=\frac{\la^2+2(2\ell+5)\la+3(\ell^2+4\ell-13)}{4(2\ell+7)}.
\end{equation}
Note that $r_n(\ell,\la)$ can be brought to the form of the ratio of two polynomials in $\la$ of degrees which increase linearly with $n$. The explicit form of $r_n(\ell,\la)$ in fact gets very complicated even for small values of $n$. We therefore desire to construct a ``simple" approximate solution $\tilde{r}_n$ to the recurrence relation \eqref{Eq:RecRel2} which is provably close to $r_n$ uniformly in $\ell$ and $\la$. Furthermore, proving that a limiting value in $n$  of $\tilde{r}_n(\ell,\la)$ is 1 will then imply \eqref{Eq:Lim1}.  
The approximation we use is the following
\begin{equation}\label{Eq:QuasiSol}
	\tilde{r}_n(\ell,\la)=\frac{\la^2}{4(n+1)(2n+2\ell+7)}+\frac{(4n+2\ell+5)\la}{2(n+1)(2n+2\ell+7)}+\frac{n-1}{n+1}+\frac{3\ell}{8(n+1)}.
\end{equation}
This expression is carefully chosen in order to emulate the behavior of $r_n(\ell,\la)$ for different values of the participating parameters. Let us elaborate on this. Firstly, notice that $r_n(\ell,\la)$ is a ratio of two polynomials in $\la$ whose degrees differ by two. It is therefore expected that $r_n(\ell,\la)$ behaves like a second degree polynomial for large values of $\la$. In fact, this asymptotic behavior can be obtained precisely. Simple induction in $n$ based on \eqref{Eq:RecRel2} and \eqref{Eq:InitCond} yields
\[
r_n(\ell,\la)=\frac{\la^2}{4(n+1)(2n+2\ell+7)}+\frac{(4n+2\ell+5)\la}{2(n+1)(2n+2\ell+7)} + O_{n,\ell}(1)\quad \text{as} \quad\la \rightarrow \infty.	
\]
This justifies the first two terms in \eqref{Eq:QuasiSol}. On the other side, we similarly obtain the following asymptotic
\[
r_n(\ell,\la)=\frac{3\ell}{8(n+1)}+O_{n,\la}(1) \quad \text{as} \quad \ell \rightarrow +\infty.	
\]
This gives rise to the last term in \eqref{Eq:QuasiSol}. Finally, to find an approximation to $r_n(\ell,\la)$ for small values of both $\ell$ and $\la$, we generate the sequence\footnote{Of course, the upper limit of this range can be any number that is large enough to serve the purpose.
	According to our experience, 20 is just fine.} $\{r_n(0,0)\}_{1\leq n\leq 20}$ and fit to it an appropriate rational function in $n$. This  leads to the choice of the remaining term in the approximation \eqref{Eq:QuasiSol}.  We note that the procedure we used here to construct $\tilde{r}_n$ differs from the one in \cite{CosDonGlo17}. The new method is more transparent and much more efficient; in particular, it avoids polynomial approximation altogether in order to find multipliers of the powers of $\la$ in \eqref{Eq:QuasiSol}.

\noindent Now, to prove that $r_n$ and $\tilde{r}_n$ are indeed ``close", we define
\begin{equation}\label{Def:Delta} 
	\delta_n(\ell,\la):=\frac{r_n(\ell,\la)}{\tilde{r}_n(\ell,\la)}-1,
\end{equation}
and show that $\delta_n$ is small uniformly in $\ell$ and $\la$. To that end we substitute Eq.~\eqref{Def:Delta} into Eq.~\eqref{Eq:RecRel2} and derive the recurrence relation for $\delta_n$
\begin{equation}\label{Eq:DeltaRec} 
	\delta_{n+1}=\ve_n-C_n\frac{\delta_n}{1+\delta_n},
\end{equation}
where
\begin{equation}\label{Eq:Eps_and_C} 
	\ve_n=\frac{A_n\tilde{r}_n+B_n}{\tilde{r}_n\tilde{r}_{n+1}}-1 \quad \text{and} 
	\quad C_n=\frac{B_n}{\tilde{r}_n\tilde{r}_{n+1}}.
\end{equation}
Now, for every $\ell\geq2, \la\in\Hb$ and $n\geq 3$, we have the following estimates
\begin{align}\label{Eq:Est_l>1}
	\begin{split}
		|\delta_3(\ell,\la)|&\leq\tfrac{1}{3}, \\ |\ve_n(\ell,\la)|&\leq\tfrac{1}{12}+\tfrac{\ell}{6(\ell+n+5)}, \\ |C_n(\ell,\la)|&\leq\tfrac{1}{2}-\tfrac{\ell}{3(\ell+n+5)}. 
	\end{split}
\end{align}
We illustrate the proof of the third estimate above; the other two are established analogously. Namely, we first bring $C_n(\ell,\la)$ to the form of the ratio of two polynomials $P_1(n,\ell,\la)$ and $P_2(n,\ell,\la)$. Also, to remove any possible ambiguity, we provide explicit forms of these polynomials in the Appendix ~\ref{App:Expr_l>1}. Our aim is to establish the desired estimate on the imaginary line only, prove that $C_n(\ell,\la)$ is analytic and polynomially bounded for $\la\in\Hb$ and then use Phragm\'en-Lindel\"of principle to extend the estimate from the imaginary line to the whole of $\Hb$. We therefore proceed by computing the polynomials
$$Q_1(n,\ell,t):=|P_1(n,\ell,it)|^2 \quad \text{and} \quad Q_2(n,\ell,t):=|P_2(n,\ell,it)|^2,$$
for $t$ real. Now, a straightforward calculation shows that the polynomial
\[
[6(\ell+n+10)]^2\,Q_1(n+3,\ell+2,t)-[\ell+3n+26]^2\,Q_2(n+3,\ell+2,t)
\]
has manifestly negative coefficients (note the shift in $n$ and $\ell$). This in turn proves that for $\ell\geq2$ and $n\geq3$ the estimate for $C_n(\ell,\la)$ holds on the imaginary line. Since $C_n(\ell,\la)=P_1(n,\ell,\la)/P_2(n,\ell,\la)$ is obviously polynomially bounded for $\la\in\Hb$, so it remains to prove that it is analytic there. This follows from the fact that for fixed $n\geq3$ and $\ell\geq2$, all of the zeros of $P_2(n,\ell,\cdot)$ are contained in the (open) left complex half plane. This can be shown in various ways, and as a canonical method we use Wall's formulation of the Routh-Hurwitz criterion, see Appendix \ref{App:Wall}. 

Now, having established estimates \eqref{Eq:Est_l>1}, we employ a simple inductive argument to conclude from \eqref{Eq:DeltaRec} that
\begin{equation}\label{Eq:FinalEst}
	|\delta_n(\ell,\la)|\leq\tfrac{1}{3}
\end{equation}
for all $n\geq3, \la\in\Hb$ and $\ell\geq2$. Since $\lim_{n\rightarrow\infty}\tilde{r}_n(\ell,\la)=1$, Eqs.~\eqref{Eq:FinalEst} and \eqref{Def:Delta} rule out \eqref{Eq:Lim2} and we therefore finally conclude that \eqref{Eq:Lim1} holds. This proves the proposition for $\ell\geq2$.

\subsubsection{The case $\ell =0$}\label{Sec:Proof_l=0}
By direct inspection one can check that for $\ell=0$, $\{1,3\} \subset \Sigma_{0}$ with respective solutions  given by
\begin{equation}\label{Eq:Eigenfs_l=0}
	f_0(\rho;1) =\frac{1-\rho^2}{(1+\rho^2)^2} \quad \text{and} \quad f_0(\rho;3)=\frac{1}{(1+\rho^2)^2}.
\end{equation}	
These are subsequently transformed (up to a constant multiple) into solutions of Eq.~\eqref{Eq:ODEheun},
\begin{equation}\label{Eq:PolyEigenf}
	y_1(x)=1-x \quad \text{and} \quad y_3(x)=1.
\end{equation}
The fact that both functions in Eq.~\eqref{Eq:PolyEigenf} are polynomials will be crucial for the rest of the proof.
We now go on and prove that $\Sigma_{0}$ consists of $\la=1$ and $\la=3$ only.  Let $\la\in\Hb \setminus\{1,3\}$.  We first observe that that $a_n(0,\la)$ can not be equal to zero eventually in $n$. Indeed, similarly to the case $\ell\geq2$, let us assume that $a_{n}(0,\la)=0$ for some $\la\in\Hb$ and all $n\geq n_0\in\mathbb{N}$. This together with $a_{-1}(0,\la)=0$ (by backward substitution in Eq.~\eqref{Eq:RecRel}) yields $a_0(0,\la)=0$, unless\footnote{This is in fact a reflection of the fact that the eigenfunctions corresponding to $\la=1$ and $\la=3$ are polynomials, see Eq.~\eqref{Eq:PolyEigenf}.} $\la\in\{1,3\}$.
We therefore conclude, by the theorem of Poincar\'e, that either Eq.~\eqref{Eq:Lim1} or  Eq.~\eqref{Eq:Lim2} holds.	Note that $$a_2(0,\la)=\tfrac{1}{2016}(\la-1)(\la-3)(\la^2+32\la+235)$$ and $$a_3(0,\la)=\tfrac{1}{266112}(\la-1)(\la-3)(\la^4+58\la^3+1052\la^2+6350\la+4971).$$ 
The linear factors $\la-1$ and $\la-3$ in the expressions above reflect the fact that the analytic solutions corresponding to $\la \in \{1,3\}$ are polynomials. Now we ``remove" these values by canceling the common linear factors in the definition of $r_2(0,\la)$, i.e., we let
$$r_2(0,\la):=\frac{\la^4+58\la^3+1052\la^2+6350\la+4971}{132(\la^2+32\la+235)},$$
which is analytic in $\Hb$, and we define $r_n(0,\la)$ for $n\geq3$ by Eq.~\eqref{Eq:RecRel2}. 	
As an approximate solution, we use
\begin{equation*}
	\tilde{r}_n(0,\la)=\frac{\la^2}{4(n+1)(2n+7)}+\frac{(4n+3)\la}{2(n+1)(2n+7)}+\frac{n-1}{n+1},
\end{equation*}
instead of Eq.~\eqref{Eq:QuasiSol} (note the difference in the second term) and define $\delta_n, \varepsilon_n$ and $C_n$ as in Eqs.~\eqref{Def:Delta}-\eqref{Eq:Eps_and_C}. 	
\noindent Now, we have the following estimates
\begin{align*}
	|\delta_5(0,\la)|\leq\tfrac{1}{3}, \quad |\ve_n(0,\la)|\leq\tfrac{1}{12}, \quad \text{and} \quad |C_n(0,\la)|\leq\tfrac{1}{2}, 
\end{align*}
for all $n\geq5$ and $\la\in\Hb$. The proof is analogous to the one in the case $\ell\geq2$ and the explicit forms are given in the Appendix \ref{App:Expr_l>1}. Finally, by induction we show that $|\delta_n(0,\la)|\leq\frac{1}{3}$ for all $n\geq5$. This rules out Eq.~\eqref{Eq:Lim2} and therefore Eq.~\eqref{Eq:Lim1} holds throughout $\Hb$.

\subsection{Proof of Proposition \ref{prop:l=1}}
By direct inspection one can check that for $\ell=1$, $\{0,1\} \subset \Sigma_{1}$ with respective solutions given by
\begin{equation*}
	f_1(\rho;0) =\frac{3\rho-\rho^3}{(1+\rho^2)^2} \quad \text{and} \quad f_1(\rho;1)=\frac{\rho}{(1+\rho^2)^2}.
\end{equation*}	
To prove that there are no other elements of $\Sigma_{1}$ other than $\la=0$ and $\la=1$ we closely follow the method developed in \cite{CosDonGlo17}. Namely, we first perform what we call the supersymmetric removal of $\la=1$ and $\la=0$. This process relies on a well-known procedure in supersymmetric quantum mechanics (hence the name), and in case $\la=1$, this method is explained in detail in \cite{CosDonGlo17}, Section 3. Although the adjustment of this method for the removal of a general eigenvalue is straightforward, for reader's convenience, we explicitly do this in Appendix \ref{App:SUSY_removal}. Subsequently, we successively remove $\la=1$ and $\la=0$,  arriving thereby at a new equation which is ``isospectral" to Eq.~\eqref{Eq:ODE_L0} modulo $\la \in \{ 0,1\}$. We then follow the approach from above and prove that the new equation does not admit smooth solutions for $\la \in \Hb$.	
	To proceed we first let
$
\hat{f}(\rho):=\rho f(\rho)
$
and thereby transform Eq.~\eqref{Eq:ODEeigenv} into
\begin{align}\label{Eq:ODE_wm_form}
	(1-\rho^2)\hat{f}''(\rho)+\left( \frac{4}{\rho}-2\left( \la + 1 \right)\rho \right)\hat{f}'(\rho)-\la(\la+1)\hat{f}(\rho)-\frac{2(5\rho^4-14\rho^2+5)}{\rho^2(1+\rho^2)^2}\hat{f}(\rho)=0.
\end{align}
Now, since $ \hat{f}_1(\rho)=\frac{\rho^2}{(1+\rho^2)^2}$ solves Eq.~\eqref{Eq:ODE_wm_form} for $\la=1$ and does not vanish inside $(0,1)$, this allows for the supersymmetric removal of $\la=1$, see Appendix \ref{App:SUSY_removal}. The corresponding supersymmetric reformulation of Eq.~\eqref{Eq:ODE_wm_form} is
\begin{align}\label{Eq:ODE_SUSY_l=1}
	(1-\rho^2)\tilde{f}''(\rho)+\left( \frac{4}{\rho}-2\left( \la + 1 \right)\rho \right)\tilde{f}'(\rho)-\la(\la+1)\tilde{f}(\rho)+\frac{2(\rho^6+\rho^4-\rho^2-9)}{\rho^2(1+\rho^2)^2}\tilde{f}(\rho)=0.
\end{align}
Direct computation shows that 
$\tilde{f_0}(\rho)=\frac{\rho^3}{1+\rho^2}$
solves Eq.~\eqref{Eq:ODE_SUSY_l=1} for $\la=0$. Furthermore, since $\tilde{f}_0$ has no zeros within $(0,1)$, again by following the process in Appendix \ref{App:SUSY_removal} we perform the supersymmetric removal of $\la=0$ relative to Eq.~\eqref{Eq:ODE_SUSY_l=1}. The equation we thereby get is
\begin{align}\label{Eq:ODE_SUSY_l=0}
	(1-\rho^2)\tilde{f}''(\rho)+\left( \frac{4}{\rho}-2\left( \la + 1 \right)\rho \right)\tilde{f}'(\rho)-\la(\la+1)\tilde{f}(\rho)-\frac{4(\rho^2+7)}{\rho^2(1+\rho^2)}\tilde{f}(\rho)=0.
\end{align}
It remains to prove that Eq.~\eqref{Eq:ODE_SUSY_l=0} has no smooth solutions for $ \lambda \in \Hb $.  To that end we proceed as in the proof of Proposition \ref{prop:l >1_l=0}. Namely, we first define 
\begin{equation*}
	\rho=\sqrt{\frac{x}{2-x}} \quad \text{and} \quad \tilde f(\rho)=x^2(2-x)^{\frac{\la}{2}}y(x),
\end{equation*}
and by that transform \eqref{Eq:ODE_SUSY_l=0} into a Heun equation
\begin{multline}\label{Eq:ODEheun1} 
	y''(x)+\left(\frac{13}{2x}+\frac{\la-1}{x-1}+\frac{1}{2(x-2)}\right)y'(x)+\frac{(\la+6)(\la+4)x-\la^2-22\la-48}{4x(x-1)(x-2)}y(x)=0.
\end{multline}
Then we consider the normalized analytic solution of Eq.~\eqref{Eq:ODEheun1} at $x=0$	\begin{equation*}
	\sum_{n=0}^{\infty}a_n(\lambda)x^{n},\quad a_0(\la)=1. 
\end{equation*}
The recurrence relation that the sequence of coefficients $\{a_n(\la)\}_{n\geq0}$ obeys is
\begin{equation}\label{eq:recrel1}
	a_{n+2}(\la)=A_n(\la)\,a_{n+1}(\la)+B_n(\la)\,a_n(\la),
\end{equation}
\noindent where
\[
A_n(\la)=\frac{\la^2+2(4n+15)\la+12(n+5)(n+2)} {4(2n+15)(n+2)}
\]
and
\[
B_n(\la)=-\frac{(\la+2n+6)(\la+2n+4)}{4(2n+15)(n+2)}.
\]
and 	 $a_{-1}(\la)=0$ and $a_0(\la)=1$.
We now let $r_n(\la):=\frac{a_{n+1}(\la)}{a_n(\la)}$,
and thereby transform Eq.~\eqref{eq:recrel1} into
\begin{equation}\label{eq:recrel}
	r_{n+1}(\la)=A_n(\la)+\frac{B_n(\la)}{r_n(\la)},
\end{equation}
with the initial condition 
\begin{equation*}
	r_0(\la)=\frac{a_1(\la)}{a_0(\la)}=A_{-1}(\la)=\frac{\la^2}{52}+\frac{11\la}{26}+\frac{12}{13}.
\end{equation*}
By backwards substitution we see that if $\la\in\Hb$ then $a_n(\la)$ can not be eventually zero.  Also, by Poincar\'e's theorem we have again either $\lim_{n\rightarrow \infty} r_n(\la) = 1$ or $	\lim_{n\rightarrow \infty} r_n(\la) = \tfrac{1}{2}$. 
Then, by observing the behavior of $r_n$ for different values of $\la$ we construct the following approximate solution to Eq.~\eqref{eq:recrel}
\begin{equation*}
	\tilde{r}_n(\lambda):=\frac{\lambda^2}{4(2n+13)(n+1)}+\frac{(2n+5)\lambda}{(2n+13)(n+1)}+\frac{2n+9}{2n+13}.
\end{equation*}
Subsequently we define
$\delta_n(\la):=\frac{r_n(\la)}{\tilde{r}_n(\la)}-1$,
and derive the corresponding recurrence relation
\begin{equation*}
	\delta_{n+1}=\varepsilon_n-C_n\frac{\delta_n}{1+\delta_n},
\end{equation*}
where
$\varepsilon_n=\frac{A_n\tilde{r}_n+B_n}{\tilde{r}_n\tilde{r}_{n+1}}-1$ and $C_n=\frac{B_n}{\tilde{r}_n\tilde{r}_{n+1}}$. Now for every $n\geq1$ and $\la\in\Hb$ the following estimates hold
\begin{align*}
	|\delta_1(\la)|\leq\tfrac{1}{3}, \quad
	|\varepsilon_n(\la)|\leq\tfrac{1}{12}, \quad 
	|C_n(\la)|\leq\tfrac{1}{2}.
\end{align*}
They are proven in the same way as the corresponding ones in Proposition \ref{prop:l >1_l=0}; the relevant explicit expressions are given in Appendix \ref{App:Expr_l=1}. Finally, by induction we conclude that  $\delta_n(\la)\leq\tfrac{1}{3}$ for all $n\geq1$ and $\la\in \Hb$.
This and the fact that $\lim_{n\rightarrow\infty}\tilde{r}_n(\la)=1$ rule out $	\lim_{n\rightarrow \infty} r_n(\la) = \tfrac{1}{2}$. Therefore $\lim_{n\rightarrow \infty} r_n(\la) = 1$ holds throughout $\Hb$. This finishes the proof.

\begin{remark}\label{Rem:SUSY_fails}
	In the case $\ell = 0$, both solutions in \eqref{Eq:Eigenfs_l=0} are nonzero on the interval $(0,1)$. However, when trying to perform a supersymmetric reformulations of Eq.~\eqref{Eq:ODEeigenv} as in the case $\ell = 1$, the following scenario occurs: the supersymmetric reformulation relative to either of the two values $\lambda =1$ and $\lambda = 3$ is the same and it removes $\la=3$ only. What is more, any further attempt to remove $\la=1$ in this way, fails, see Appendix \ref{App:SUSY_removal} for an explanation of this.
	Therefore, a different approach was required. The new method we devised in Section \ref{Sec:Proof_l=0} crucially relies on the fact that the function in Eq.~\eqref{Eq:PolyEigenf} are polynomials. The process does not require any reformulation of Eq.~\eqref{Eq:ODEheun} and even allows for solutions that vanish inside $(0,1)$. For that reason, the method is, in a sense, more general than the one exhibited in \cite{CosDonGlo17} and would in addition yield even shorter solutions to similar spectral problems in \cite{CosDonGloHua16,CosDonGlo17,DonGlo19}.
\end{remark}

\section{The spectrum of $\mb L_a$ - Growth bounds for $\mb S_a(\tau)$}\label{Sec:Spectrum_growthbounds}

With the result of the previous section, we are now able to investigate the spectrum of the operator $\mb L_a:  \mc D(\mb L_a) \subset \mc H \to \mc H$. First, we make some general observations.

\begin{lemma}\label{Le:ResolventBounds_La}
Fix $\varepsilon >0$. There are constants $\kappa^* > 0$  and $c > 0$ such that for all 
 $a \in \overline{\B^7_{\delta}}$ with $\delta > 0$ sufficiently small 
\begin{equation*}
 \|\mb R_{\mb L_a}(\lambda) \| \leq c
 \end{equation*} 
for all $\lambda \in \C$ satisfying $\mathrm{Re} \la \geq -\frac12 + \varepsilon$ and  $|\la| \geq \kappa^*$.
\end{lemma}

\begin{proof}
In view of Proposition \ref{Prop:FreeEvol} and the identity
 \begin{align}\label{Eq:Operator_La_Id}
\lambda - \mb L_a = [1 - \mb L'_a \mb R_{\mb L}(\lambda)](\lambda - \mb L)
\end{align}
 it suffices to show that under suitable assumptions on $a$ and $\la$,
 $  \| \mb L'_a \mb R_{\mb L}(\lambda) \| \leq \frac12$
 such that the Neumann series $\sum_{k=0}^{\infty} [ \mb L'_a \mb R_{\mb L}(\lambda)]^k$  
converges. Then $\mb R_{\mb L_a}(\lambda) = \mb R_{\mb L}(\lambda) \sum_{k=0}^{\infty} [ \mb L'_a \mb R_{\mb L}(\lambda)]^k$ and the claimed bounds follow from Proposition \ref{Prop:FreeEvol}.  By definition of $\mb L'_a$ and the properties of the potential we have 
\[ \| \mb L'_a \mb R_{\mb L}(\lambda) \mb f\|  = \| V_a [\mb R_{\mb L}(\lambda) \mb f]_1 \|_{H^2(\B^7)} \lesssim   \| [\mb R_{\mb L}(\lambda) \mb f]_1 \|_{H^2(\B^7)},  \]
for all $a \in \overline{\B^7_{\delta}}$ with $0 < \delta < \delta^*$ suitably small. 
To estimate the last term, we use the structure of $\mb L$.  For $\mb f \in \mc H$, set $\mb u = \mb R_{\mb L}(\lambda) \mb f \in \mc D(\mb L)$ such that $(\la - \mb L) \mb u = \mb f$.  The first component of this equation reads 
\[ \xi^i \partial_i u_1(\xi) + (\la +1 ) u_1(\xi) - u_2(\xi) = f_1(\xi).\]
Hence,
\begin{align*}
\|u_1\|_{H^2(\B^7)} \lesssim \frac{1}{|\la + 1| } (\|u_1\|_{H^3 (\B^7)} + \|u_2\|_{H^2(\B^7)} + \|f_1\|_{H^2(\B^7)}).
\end{align*}
With Proposition \ref{Prop:FreeEvol} we infer that 
\[
\| [\mb R_{\mb L}(\lambda) \mb f]_1 \|_{H^2({\mathbb B^7)}} \lesssim |\lambda|^{-1} \| ( \|\mb R_{\mb L}(\lambda) \mb f  \| +   \| \mb f \|) \lesssim |\lambda|^{-1}   \| \mb f \|
\]
for all $\lambda \in \C$ with $\mathrm{Re} \lambda \geq -\frac12 + \varepsilon$ such that 
\[
 \| \mb L'_a \mb R_{\mb L}(\lambda) \mb f \| \lesssim |\lambda|^{-1} \|\mb f \| \]
for $a$ sufficiently small. The statement holds if $|\la| \geq \kappa^*$ for some suitably large $\kappa^* > 0$.
\end{proof}

\begin{lemma}\label{Le:SpecLa_Prelim}
Let  $\delta > 0$  be sufficiently small and $a \in \overline{\mathbb B^7_{\delta}}$.  If $\lambda \in \sigma(\mb L_a)$ and $\mathrm{Re} \la > -\frac12$, then $\la$ is an isolated eigenvalue.
\end{lemma}

\begin{proof}
The assumptions on $\la$ imply that $\la \notin \sigma(\mb L)$.  We use again Eq.~\eqref{Eq:Operator_La_Id}
to see that $\la \in \sigma(\mb L_a)$ if and only if the operator $ 1 - \mb L'_a \mb R_{\mb L}(\la)$ is not bounded invertible. This means that 
$1$ is an eigenvalue of the compact operator $ \mb L'_a \mb R_{\mb L}(\la)$. Hence, there is an eigenfunction $\mb f \in \mc H$ satisfying $(1 - \mb L'_a \mb R_{\mb L}(\la)) \mb f = 0$. We set $\mb u := \mb R_{\mb L}(\la) \mb f \in \mc D(\mb L) = \mc D(\mb L_a)$. Then $(\la - \mb L) \mb u = \mb f$ and infer that 
$(\la - \mb L_a)\mb u = ( 1 - \mb L'_a \mb R_{\mb L}(\la) ) (\la - \mb L) \mb u =0$. We conclude that $\la$ is an eigenvalue of $\mb L_a$. Next, we apply the Analytic Fredholm Theorem, see e.g.~\cite{Simon2015_4}, Theorem 3.14.3, p.~194, to the function $\la \mapsto \mb L'_a \mb R_{\mb L}(\la)$ defined on the open half plane $\mathbb H_{-\frac12} := \{ \la \in \C: \mathrm{Re} \la > -\frac12 \}$.  By Lemma \ref{Le:ResolventBounds_La}, there are points in $\mathbb H_{-\frac12}$ such that  $1 - \mb L'_a \mb R_{\mb L}(\la)$ is invertible. Thus, it is invertible on all of  $\mathbb H_{-\frac12}$ except for a discrete set $S \subset \mathbb H_{-\frac12}$ and we infer that $\la$ is isolated. 
\end{proof}

\subsection{The spectrum of $\mb L_a$, $a=0$} \label{Sec:Spetrum_L0}
 \begin{proposition}\label{Prop:Spectrum_L0}
There exists an $0 < \omega_0 \leq \frac12$, such that  
\[ \sigma(\mb L_0) \subset \{ \lambda \in \C: \mathrm{Re} \lambda  \leq - \omega_0 \} \cup \{\lambda_0,\lambda_1,\lambda_2\},\]
where $\lambda_0 = 0$, $\lambda_1 = 1$ and $\lambda_2 = 3$ are eigenvalues. The geometric eigenspace of the eigenvalue $\lambda_2$ is one-dimensional and spanned by $\mb h_0 = (h_{0,1},h_{0,2})$,
\begin{align}\label{Eq:Eigenfunct_la=3}
  h_{0,1}(\xi) =   \frac{1}{(1+ |\xi|^2)^{2}} , \quad h_{0,2}(\xi) = 4 h_{0,1}(\xi)+ \xi^j \partial_j h_{0,1}(\xi).
  \end{align}
Furthermore, the geometric eigenspaces of  $\lambda_1$ and $\lambda_0$ are spanned by functions $\{\mb g^{(k)}_{0} \}_{k = 0,\dots,7}$, and  $\{\mb q^{(j)}_{0}\}_{j=1,\dots,7}$, respectively. Explicitly, we have 
\begin{align}\label{Eq:Eigenfunct_la=1}
\begin{split}
 g^{(0)}_{0,1}(\xi)   =  \frac{1-|\xi|^2}{\left( 1+ |\xi|^2\right)^2 }, \quad g^{(0)}_{0,2}(\xi) = 2 g^{(0)}_{0,1}(\xi) + \xi^j \partial_j g^{(0)}_{0,1}(\xi), \\
 g^{(j)}_{0,1}(\xi)   =  \frac{ \xi_j }{\left( 1+ |\xi|^2\right)^2 }, \quad g^{(j)}_{0,2}(\xi)= 2 g^{(j)}_{0,1}(\xi) + \xi^j \partial_j g^{(j)}_{0,1}(\xi),
\end{split}
 \end{align}
 and
 \begin{align}\label{Eq:Eigenfunct_la=0}
q^{(j)}_{0,1}(\xi)  =  \frac{ \xi_j ( 3-|\xi|^2) }{\left( 1+ |\xi|^2\right)^2 }, \quad q^{(j)}_{0,2}(\xi)  = q^{(j)}_{0,1}(\xi) + \xi^j \partial_j q^{(j)}_{0,1}(\xi).
\end{align}
\end{proposition}

\begin{proof}
We prove that
\begin{align*}
 \{ \la \in \C: \mathrm{Re} \la \geq 0 \} \setminus \{0,1,3\}  \subset \rho(\mb L_0).
\end{align*}
Let $\mathrm{Re} \la \geq 0$ and assume that $\la \in \sigma(\mb L_0)$. By Lemma \ref{Le:SpecLa_Prelim}, $\la$ is an eigenvalue and there is a non-trivial $\mb u \in \mc D(\mb L_0)$ satisfying the eigenvalue equation $(\lambda - \mb L_0) \mb u =0$ (in a suitable weak sense). As in the proof of Lemma \ref{Le:LuPhi_DenseRange} this equation reduces to a degenerate elliptic problem for $u_1$
\begin{align}\label{Eq:Spectrum_L0}
- (\delta^{ij} - \xi^i \xi^j) \partial_i \partial_j u_1(\xi) + 2(\lambda +2) \xi^i \partial_i u_1(\xi) + (\lambda + 1)(\lambda + 2) u_1(\xi)  - V_0(\xi) u_1(\xi) = 0,
\end{align}
and $u_2$ is given by 
\begin{align}\label{Eq:Spectrum_L0_2}
 u_2(\xi) = \xi^i \partial_i u_1(\xi) + (\lambda + 1) u_1(\xi).
\end{align}
Now, $\mb u \in \mc H$ implies that $u_1 \in H^3(\mathbb B^7)$ and by elliptic regularity, we infer that  $u_1 \in C^{\infty}(\mathbb B^7) \cap H^3(\mathbb B^7)$. 
Since the potential $V_0$ is radial we can do a decomposition into spherical harmonics. We expand $u_1$ according to 
\[ u_1(\rho \omega) = \sum_{\ell = 0}^{{\infty}}  \sum_{m \in \Omega_{\ell}} (u_1 (\rho \cdot)|Y_{\ell, m} )_{L^2(\mathbb S^{d-1})} Y_{\ell, m} \left (\omega \right) =: \sum_{\ell = 0}^{{\infty}}  \sum_{m \in \Omega_{\ell}} u_{\ell, m}(\rho)  Y_{\ell, m} \left (\omega \right),\]
where the sum converges in $H^{k}(\B^7_{1-\varepsilon})$ for any $k \in \N_0$ and fixed $\varepsilon >0$, see Eq.~\eqref{Decomp:SpherHarm_Hk}. With this Eq.~\eqref{Eq:Spectrum_L0} decouples into infinitely many ODEs. The fact that $u_1$ is non-trivial implies that there are indices $\ell \in \N_0$, $m \in \Omega_{\ell}$, such that that $u_{\ell, m}$ is non-trivial and satisfies
\begin{align*}
\mc T_{\ell}(\lambda) u_{\ell,m}(\rho) = 0,
\end{align*}
where $\mc T_{\ell}(\lambda)$ is given in Eq.~\eqref{Eq:ODE_L0}. The properties of $u_1$ imply that $u_{\ell, m} \in C^{\infty}[0,1) \cap H^3(\frac12,1)$.
By inspection, the Frobenius indices at $\rho = 1$ are $\{0, 2 - \lambda \}$ . If $\la \notin \N_0$, then $u_{\ell, m}$ is either analytic or it behaves like $ (1-\rho)^{2 - \lambda }$. For $\la \in \N$, $\la \geq 3$, we come to the same conclusion.  If $\la =2,1$ or $0$ then the possible non-analytic behavior is described by $\log(1-\rho)$, $\log(1-\rho)(1-\rho)$ and $\log(1-\rho)(1-\rho)^2$, respectively. In all cases, non-analyticity can be excluded by the requirement  $u_{\ell, m} \in H^3(\frac12,1)$ and we infer that $u_{\ell, m} \in C^{\infty}[0,1]$. In view of the Propositions \ref{prop:l >1_l=0} - \ref{prop:l=1} we have $\ell \in \{0,1\}$ and we conclude that if $\la \notin \{0,1,3\}$ then $\la \in \rho(\mb L_0)$ by contradiction.

 If $\la = 3$, then $\ell = 0$ and 
$u_{0,m}(\rho) = (1+\rho^2)^{-2}$ where $m \in \Omega_m = \{1\}$, which implies Eq.~\eqref{Eq:Eigenfunct_la=3} since
$Y_{0,m}$ is a constant. If $\la = 1$, then either $\ell = 0$ and $u_{0,m} = f_0(\cdot;1)$ or $\ell = 1$ and $u_{1,m} = f_1(\cdot;1)$. Since $Y_{1,m}(\omega)$ is a constant multiple of $\omega^m$ for $m=1,\dots,7$, we obtain Eq.~\eqref{Eq:Eigenfunct_la=1}. Finally, if $\la = 0$, then $\ell = 1$ and $u_{1,m} = f_1(\cdot;0)$, which yields  Eq.~\eqref{Eq:Eigenfunct_la=0}. The claim now follows from the openness of $\rho(\mb L)$,  Lemma \ref{Le:ResolventBounds_La} and Lemma \ref{Le:SpecLa_Prelim}.
\end{proof}

\subsubsection{Spectral projections}
We define Riesz projections associated to the unstable eigenvalues of $\mb L_0$. We set
\begin{align*}
\mb H_0 := \frac{1}{2 \pi i} \int_{\gamma_2} \mb R_{\mb L_0}(\lambda) d\lambda
\quad \mb P_0 := \frac{1}{2 \pi i} \int_{\gamma_1} \mb R_{\mb L_0}(\lambda) d\lambda, \quad \mb Q_0 := \frac{1}{2 \pi i} \int_{\gamma_0} \mb R_{\mb L_0}(\lambda) d\lambda 
\end{align*}
where $\gamma_j (s) := \lambda_j +  \frac{\omega_0}{2} e^{2\pi i s}$ for $s \in [0,1]$.

\begin{lemma}\label{Le:Rank_P0}
	We have
	\[ \mathrm{dim}~\mathrm{rg }~\mb H_0  = 1, \quad  \mathrm{dim}~ \mathrm{rg }~\mb P_0 = 8, \quad \mathrm{dim}~\mathrm{rg }~\mb Q_0  = 7. \]
\end{lemma}

\begin{proof}
	First, we observe that the projection operators have finite rank. If this was not the case, they would belong to the the essential spectrum of $\mb L_a$
	see 	\cite{kato}, p.~239, Theorem 5.28 and (5.33) on p.~242. However, since the essential spectrum is invariant under compact perturbations, see \cite{kato}, Theorem 5.35, p.~244, this contradicts 						Proposition \ref{Prop:FreeEvol}. Next, we show that $\mathrm{dim}~ \mathrm{rg }~\mb P_0 = 8$ (the abstract part of the argument follows analogously for the other projections).
	 First, we have $\mc H = \mathrm{ker} \mb P_0 \oplus \mathrm{rg} \mb P_0$ and $	\mb L_0$ is decomposed into parts on the respective subspaces. Note that $\mathrm{rg} \mb P_0 \subset \mc D(\mb L_0)$ and that the spectrum of the part of $\mb L_0$ on $\mathrm{rg} \mb P_0$ is given by 
	\[ \sigma(\mb L_0|_{\mathrm{rg} \mb P_0}) = \{1\}. \]
	Since $\mb P_0$ has finite rank, $\mb L_0|_{\mathrm{rg} \mb P_0}$ is a bounded operator. We now consider the operator $\lambda - \mb L_0|_{\mathrm{rg} \mb P_0}$ for $\lambda = 1$,  which is nilpotent due to the spectral structure of $ \mb L_0|_{\mathrm{rg} \mb P_0}$. Hence, there exists a minimal $m \in \N$ such that 
	\begin{align}\label{Eq:AM}
	(\la -  \mb L_0|_{\mathrm{rg} \mb P_0})^m \mb u = 0
	\end{align}
	for $\la = 1$ and all $\mb u \in \mathrm{rg} \mb P_0$. Obviously, $\mathrm{ker}(1 - \mb L_0) = \mathrm{span} ( \mb g_0^{(0)},\mb g_0^{(1)},\dots,\mb g_0^{(7)} ) \subset \mathrm{rg} \mb P_0$. If $m =1$, then the reverse inclusion holds and the claim follows. We argue by contradiction and assume that $m \geq 2$. Then there is  $\mb u \in \mathrm{rg} \mb P_0$ such that 
	\begin{equation}\label{Eq:AlgebMult}
	(1 -  \mb L_0)\mb u =\mb v,
	\end{equation}
	for some nontrivial $\mb v \in \ker(1 - \mb L_0)$. Eq.~\eqref{Eq:AlgebMult} yields an elliptic equation for $u_1$ given by 
	\begin{align}\label{Eq:AlgebMult_Elliptic}
	- (\delta^{ij} - \xi^i \xi^j) \partial_i \partial_j u_1(\xi) + 2(\lambda +2) \xi^i \partial_i u_1(\xi) + (\lambda + 1)(\lambda + 2) u_1(\xi) - V_0(\xi) u_1(\xi) = F(\xi),
	\end{align}
	for $\la =1$, where
	\[ F(\xi) = \xi^i \partial_i v_{1}(\xi) + (\lambda +2) v_{1}(\xi) + v_{2}(\xi). \]
	Now, $\mb v \in \ker(1 - \mb L_0)$ implies that $\mb v =  \sum_{k=0}^{7 }\tilde {\alpha_k} \mb g_0^{(k)}$ for some constants $ \tilde {\alpha_k} \in \C$, see Proposition \ref{Prop:Spectrum_L0}. To abbreviate the notation we set $g_k := g^{(k)}_{0,1}$ to obtain
	\begin{align*}
	F(\xi) =\sum_{k=0 }^{7} \tilde {\alpha_k}   (2 \xi^i\partial_i  g_k+ 5 g_k).
	\end{align*}

	Note that $g_0(\xi)$ and $g_j(\xi)$ are constant multiples of $f_0(|\xi|) Y_{0,1}(\tfrac{\xi}{|\xi|})$ and $f_1(|\xi|) Y_{1,j}(\tfrac{\xi}{|\xi|})$ respectively,
	for $j=1,\dots, 7$, 
	with $f_0  = f_0(\cdot;1)$,  $f_1 =f_1(\cdot;1)$ defined in Proposition \ref{prop:l >1_l=0} and Proposition \ref{prop:l=1}. Hence, in polar coordinates the right hand side of Eq.~\eqref{Eq:AlgebMult_Elliptic} can be written as 
	\[ F(\rho \omega) = \alpha_0   [ 2 \rho f_0'(\rho) +  5 f_0(\rho)]  Y_{0,1}(\omega) + \sum_{j=1}^{7}   \alpha_j   [2 \rho f'_1(\rho) +  5 f_1(\rho)]  Y_{1,j}(\omega) \] 
	with $\alpha_k \neq 0$ for at least one $k \in \{0, \dots, 7\}$.
	Since $u_1$ is smooth on $\mathbb B^7$ we decompose it into spherical harmonics and set
	$u_{\ell,m}(\rho) := (u_1(\rho \cdot)|Y_{\ell,m})$. The properties of $u_1$ imply that $u_{\ell,m} \in C^{\infty}[0,1) \cap H^3(\frac12,1)$ such that $u_{\ell,m}  \in C^2[0,1]$ by Sobolev embedding. In view of Eq.~\eqref{Eq:AlgebMult_Elliptic} 
	the following ODEs are satisfied,
	\begin{align}\label{Eq:AlgebMult_ODEs}
	\mc T_{0}(1) u_{0,1} = -\alpha_0 G_0, \quad \mc T_{1}(1) u_{1,j} = -\alpha_j G_1,
	\end{align}
	for $j = 1, \dots, 7$ and $G_i(\rho) = 2 \rho f_i'(\rho) +  5 f_i(\rho)$, $i = 0,1$. Assume that  $\alpha_0 \neq 0$. Without loss of generality we set $\alpha_0  = -1$ and study the ODE problem
	\begin{equation}\label{Eq:NonHom1}
	(1-\rho^2)u''(\rho)+(\tfrac{6}{\rho}-6\rho )u'(\rho)-(6-\tfrac{48}{(1+\rho^2)^2}) u(\rho)=G_0(\rho)
	\end{equation}
	with
	\begin{equation*}
	G_0(\rho)=-\frac{\rho^4+12\rho^2-5}{(1+\rho^2)^3}.
	\end{equation*}
	We claim that there is no $C^2[0,1]$ solution to Eq.~\eqref{Eq:NonHom1}. For the homogeneous version of that equation a fundamental system is given by $\{\hat u_1, \hat u_2 \}$, where 
	\[ \hat{u}_1(\rho)=f_0(\rho)=\frac{1-\rho^2}{(1+\rho^2)^2} \]
	and
	\[ \hat{u}_2(\rho) = \hat{u}_1(\rho) \int_{\frac{1}{2}}^{\rho} \frac{ds}{s^6\,\hat{u}_1(s)^2} =\frac{1-\rho^2}{(1+\rho^2)^2}\int_{\frac{1}{2}}^{\rho}\frac{(1+s^2)^4}{s^6(1-s^2)^2}ds\]
	for $\rho\in (0,1).$ Note that
	\begin{equation}\label{Eq:Asympt0}
	\hat{u}_2(\rho)\simeq \rho^{-5} \quad \text{as} \quad \rho \ra 0^{+}
	\end{equation}
	and
	\begin{equation}\label{Eq:Asympt1}
	\hat{u}_2(\rho)=2-6(1-\rho)\ln(1-\rho)+O(1-\rho) \quad \text{as} \quad \rho \ra 1^{-}.
	\end{equation}
	Since $W(\hat{u}_1,\hat{u}_2)(\rho) = \rho^{-6}$, we solve Eq.~\eqref{Eq:NonHom1} by the method of variation of parameters. Namely, for $\rho\in(0,1)$ we have
	\begin{equation}\label{Eq:SolGen}
	u(\rho)=c_1\hat{u}_1(\rho)+c_2\hat{u}_2(\rho)+\hat{u}_2(\rho)\int_{0}^{\rho}\frac{\hat{u}_1(s)G_0(s)s^6}{1-s^2}dx-\hat{u}_1(\rho)\int_{0}^{\rho}\frac{\hat{u}_2(s)G_0(s)s^6}{1-s^2}ds
	\end{equation}
	for some $c_1,c_2\in \mathbb{C}$. If $u\in C^2[0,1]$ then $c_2=0$ in the above expression. Subsequently, by differentiating Eq.~\eqref{Eq:SolGen} we get
	\begin{equation}\label{Eq:SolGen2}
	u'(\rho)=c_1\hat{u}_1'(\rho)+\hat{u}_2'(\rho)\int_{0}^{\rho}\frac{\hat{u}_1(s)G_0(s)s^6}{1-s^2}dx-\hat{u}_1'(\rho)\int_{0}^{\rho}\frac{\hat{u}_2(s)G_0(s)s^6}{1-s^2}ds,
	\end{equation}
	for $\rho\in(0,1)$. We claim that 
	\begin{equation}\label{Eq:claim}
	u'(\rho)\simeq \ln(1-\rho) \quad \text{as} \quad \rho \ra 1^{-}.
	\end{equation}
	To establish this we study the asymptotics of all three terms on the right hand side of Eq.~\eqref{Eq:SolGen2}. First of all, the first term is bounded near $\rho=1$. Then, from
	\[ C:=\int_{0}^{1}\frac{\hat{u}_1(s)G_0(s)s^6}{1-s^2}ds=\int_{0}^{1}\frac{s^2}{1-s^2}\frac{d}{ds}\left[\frac{s^5(1-s^2)^2}{(1+s^2)^4}\right]ds=-2\int_{0}^{1}\frac{s^6}{(1+s^2)^4}ds<0 \] 
	and Eq.~\eqref{Eq:Asympt1} we have
	\[
	\hat{u}_2'(\rho)\int_{0}^{\rho}\frac{\hat{u}_1(s)G_0(s)s^6}{1-s^2}ds \sim 6C \ln(1-\rho) \quad \text{as} \quad \rho\ra 1^-. 
	\]
	Also,
	\[
	\hat{u}_1'(\rho)\int_{0}^{\rho}\frac{\hat{u}_2(s)G(s)s^6}{1-s^2}ds \sim -\frac{1}{2} \ln(1-\rho) \quad \text{as} \quad \rho \ra 1^-.
	\]
	Now we easily observe that  
	$6C+\frac{1}{2}>0$ and then infer from these asymptotics and Eq.~\eqref{Eq:SolGen2} that \eqref{Eq:claim} holds. We therefore conclude that there is no $C^2[0,1]$ solution to Eq.~\eqref{Eq:NonHom1} and hence $\alpha_0=0$.
	
	Then, $\alpha_j \neq 0$ for at least one $j \in \{1,\dots, 7\}$. Without loss of generality, we set $\alpha_1 = -1$. Eq.~\eqref{Eq:AlgebMult_ODEs} implies that the ODE
	\begin{align}\label{Eq:NonHom3}
	(1-\rho^2) u''(\rho) + (\tfrac{6}{\rho}   - 6 \rho ) u'(\rho)   - ( 6 + \tfrac{6}{\rho^2} - \tfrac{48}{(1+\rho^2)^2} ) u(\rho) = G_1(\rho),
	\end{align}
	with $G_1(\rho) = \frac{\rho (7-\rho^2)}{(1+\rho^2)^2}$
	has a solution which is in $C^2[0,1]$. We will again show that this is impossible. More precisely, we prove that any solution to Eq.~\eqref{Eq:NonHom3} that is bounded near $\rho=0$ has an unbounded derivative near $\rho=1.$ Note that	$\hat{u}_1(\rho)=f_1(\rho)=\frac{\rho}{(1+\rho^2)^2}$,
	solves the homogeneous version of Eq.~\eqref{Eq:NonHom3}. We can therefore compute another solution, namely
	\[ \hat{u}_2(\rho) = \hat{u}_1(\rho) \int_{1}^{\rho} \frac{ds}{s^6\,\hat{u}_1(s)^2},\]
	such that $\hat{u}_1$ and $\hat{u}_2$ are linearly independent.
	In particular, we have
	$\hat{u}_2(\rho)\simeq \rho^{-6}$ as $\rho \ra 0^{+}$ and
	$\hat{u}_2(\rho)\simeq 1-\rho$ as $\rho \ra 1^{-}$.
	Now, the general solution to Eq.~\eqref{Eq:NonHom3} for $\rho\in(0,1)$ is
	\begin{equation*}
	u(\rho)=c_1 \hat{u}_1(\rho) + c_2\hat{u}_2(\rho)+\hat{u}_2(\rho)\int_{0}^{\rho}\frac{\hat{u}_1(s)G_1(s)s^6}{1-s^2}\,ds-\hat{u}_1(\rho)\int_{0}^{\rho}\frac{u_2(s)G_1(s)s^6}{1-s^2}\,ds.
	\end{equation*}
	Boundedness of $u$ at the origin implies $c_2=0$. Then we simply observe that
	$u'(\rho) \simeq \ln(1-\rho)$ as $\rho \ra 1^{-}$.  We finally infer that $m = 1$ in Eq.~\eqref{Eq:AM}, which in turn implies the claim for $\mb P_0$. 
	\\
	
	For $\mb H_0$, the right hand side of the analogue of Eq.~\eqref{Eq:AlgebMult_Elliptic} is given by $F(\xi) =  2 \xi^{j} \partial_j h_{0,1}(\xi) + 9 h_{0,1}(\xi)$ which leads to the claim that  
	the ODE  
	\begin{align}\label{Eq:NonHom2}
	\begin{split}
	(1-\rho^2) u''(\rho) + (\tfrac{6}{\rho}   - 10 \rho ) u'(\rho)  - (20  - \tfrac{48}{(1+\rho^2)^2}) u(\rho)  = H(\rho)
	\end{split}
	\end{align}
	with $H(\rho) = \frac{9+\rho^2}{(1+\rho^2)^3}$
	has a solution $u \in C^2[0,1]$. However, we exclude this in a similar way as above. Namely, we show that every solution to Eq.~\eqref{Eq:NonHom2} that is bounded near the origin is necessarily unbounded near $\rho=1$. Since $\hat{u}_1(\rho)=\frac{1}{(1+\rho^2)^2}$ solves the homogeneous version of Eq.~\eqref{Eq:NonHom2}, another (linearly independent) solution is
	\[ \hat{u}_2(\rho) = \hat{u}_1(\rho) \int_{\frac{1}{2}}^{\rho} \frac{ds}{s^6(1-s^2)^2\,\hat{u}_1(s)^2} =\frac{1}{(1+\rho^2)^2}\int_{\frac{1}{2}}^{\rho}\frac{(1+s^2)^4}{s^6(1-s^2)^2}ds\]
	for $\rho\in(0,1).$ Note that $\hat{u}_2$ is singular at both endpoints of the interval $(0,1)$. More precisely, we have
	$\hat{u}_2(\rho)\simeq \rho^{-5}$ as $\rho \ra 0^{+}$	and
	$\hat{u}_2(\rho)\simeq (1-\rho)^{-1}$ as  $\rho \ra 1^{-}$.
	Now, for $\rho\in(0,1)$ we have	\begin{align}\label{Eq:SolGen3}
	\begin{split}
	u(\rho)=c_1\hat{u}_1(\rho)+c_2\hat{u}_2(\rho)&+ \hat{u}_2(\rho)\int_{0}^{\rho}\hat{u}_1(s)H(s)(1-s^2)s^6ds \\&-\hat{u}_1(\rho)\int_{0}^{\rho}\hat{u}_2(s)H(s)(1-s^2)s^6ds,
	\end{split}
	\end{align}
	for some $c_1,c_2\in \mathbb{C}$. Boundedness of $u$ at the origin  forces $c_2=0$ in the above expression. 
	Note that the first and the last term on the right hand side of Eq.~\eqref{Eq:SolGen3} are bounded near $\rho=1$. However, the remaining term is unbounded, unless the integral multiplying $\hat{u}_2(\rho)$ is equal to zero for $\rho=1$. This is impossible since the integrand is strictly positive on $(0,1).$
	Therefore, $u(\rho)\simeq \hat{u}_2(\rho)\simeq(1-\rho)^{-1}$ near $\rho=1$.
	\\
	
	Finally, for $\mb Q_0$, 
	\[ F(\xi) = \sum_{j=1}^7 \alpha_j   (2 \xi^{j} \partial_j q_{0,1}^{(j)}(\xi) + 3 q_{0,1}^{(j)}(\xi))\] 
	and we have to exclude the existence of $C^2[0,1]$ solutions of the equation
	\begin{align}\label{Eq:NonHom4}
	\begin{split}
	(1-\rho^2) u''(\rho) + (\tfrac{6}{\rho}   - 4 \rho ) u'(\rho)  - (2 + \tfrac{6}{\rho^2} - \tfrac{48}{(1+\rho^2)^2}) u(\rho)  = Q(\rho)
	\end{split}
	\end{align}
	with $Q(\rho) = \frac{- \rho(5 \rho^4+ 6\rho^2-15)}{(1+\rho^2)^2}$.
	We start with the observation that
	$\hat{u}_1(\rho)=\frac{3\rho-\rho^3}{(1+\rho^2)^2}$
	solves the homogeneous version of Eq.~\eqref{Eq:NonHom4}. Then another (linearly independent) solution is
	\[ \hat{u}_2(\rho) = \hat{u}_1(\rho) \int_{1}^{\rho} \frac{1-s^2}{s^6\,\hat{u}_1(s)^2}ds.\]
	Furthermore, we have the following asymptotics,
	$\hat{u}_2(\rho)\simeq \rho^{-6}$ as $\rho \ra 0^{+}$
	and
	$\hat{u}_2(\rho)\simeq (1-\rho)^{2}$ as $\rho \ra 1^{-}$.
	Now, every solution to Eq.~\eqref{Eq:NonHom4} is of the following form
	\begin{equation*}
	u(\rho)=c_1\hat{u}_1(\rho)+c_2\hat{u}_2(\rho)+\hat{u}_2(\rho)\int_{0}^{\rho}\frac{\hat{u}_1(s)Q(s)s^6}{(1-s^2)^2}ds-\hat{u}_1(\rho)\int_{0}^{\rho}\frac{\hat{u}_2(s)Q(s)s^6}{(1-s^2)^2}ds,
	\end{equation*}
	for some complex constants $c_1, c_2$ and $\rho\in(0,1)$.
	Singular behavior of $\hat{u}_2(\rho)$ at $\rho=0$ forces $c_2=0$. Then simple analysis yields
	$u''(\rho) \simeq \ln(1-\rho)$ as $\rho \ra 1^{-}$ and the claim follows.
\end{proof}

\subsection{The spectrum of $\mb L_a$ for $a \neq 0$}

We turn to the investigation of the spectrum of $\mb L_a$ for small $a$. 

\begin{lemma}\label{Le:Spectrum_La}
Let  $\delta > 0$  be sufficiently small. Then for all $a \in \overline{\mathbb B^7_{\delta}}$ the following holds: \[ \sigma(\mb L_a) \subset \{ \lambda \in \C: \mathrm{Re} \lambda  < - \tfrac{\omega_0}{2} \} \cup \{\lambda_0,\lambda_1,\lambda_2\},\]
where $\omega_0 > 0$ is the constant in Propositions \ref{Prop:Spectrum_L0} and $\lambda_0 =0$, $\la_1=1$, $\la_2=3$ are eigenvalues.  The eigenspace of $\lambda_2$ is one-dimensional and spanned by $\mb h_a = (h_{a,1}, h_{a,2})$, where    
\[   h_{a,1}(\xi) =  \frac{1}{(2\gamma(\xi,a)^2 + |\xi|^2 - 1)^2}, \quad h_{a,2}(\xi) = 4 h_{a,1}(\xi) + \xi^j \partial_j h_{a,1}(\xi).  \]
Furthermore, the eigenspaces of $\lambda_0$ and $\lambda_1$  are spanned by $\{\mb g^{(k)}_{a} \}_{k = 0,\dots,7}$ and  $\{\mb q^{(j)}_{a}\}_{j=1,\dots,7}$, respectively. Explicitly, we have 
\begin{align*}
g_{a,1}^{(0)}(\xi) & =\frac{1}{2\gamma(\xi,a)^2+|\xi|^2-1}\left(A_0(a)\gamma(\xi,a)^2-2\frac{\gamma(\xi,a)+A_0(a)(|\xi|^2-1)}{2\gamma(\xi,a)^2+|\xi|^2-1}\right), \\
 g^{(0)}_{a,2}(\xi)  & = 2 g^{(0)}_{a,1}(\xi) + \xi^j \partial_j g^{(0)}_{a,1}\xi), \\
g_{a,1}^{(j)}(\xi) & =\frac{1}{2\gamma(\xi,a)^2+|\xi|^2-1}\left(\frac{A_j(a)}{\gamma(\xi,a)^2}+2\frac{\xi^j \gamma(\xi,a)^2+A_j(a)(|\xi|^2-1)}{2\gamma(\xi,a)^2+|\xi|^2-1}\right),   \\
g^{(j)}_{a,2}(\xi) & = 2 g^{(j)}_{a,1}(\xi) + \xi^j \partial_j g^{(j)}_{a,1}(\xi),\\
q_{a,1}^{(j)}(\xi) & =4 \partial_{a_j}\gamma(\xi,a)\cdot\frac{-2\gamma(\xi,a)^2+|\xi|^2-1}{(2\gamma(\xi,a)^2+|\xi|^2-1)^2}, \\
q^{(j)}_{a,2}(\xi)  & = q^{(j)}_{a,1}(\xi) + \xi^j \partial_j q^{(j)}_{a,1}(\xi),\\
\end{align*}
and the eigenfunctions depend Lipschitz continuously on the parameter $a$, i.e., 
\begin{align*}
 \| \mb h_a - \mb h_b \| +  \|\mb g_{a}^{(k)} - \mb g_{b}^{(k)}\| + \|\mb q_{a}^{(j)}-\mb q_{b}^{(j)} \| \lesssim | a- b |,
\end{align*}
for all $a,b \in \overline{\mathbb B^7_{\delta}}$.
\end{lemma}

\begin{proof}
Let $\varepsilon =- \frac{\omega_0}{2} + \frac{1}{2}$, see Proposition \ref{Prop:Spectrum_L0}, and let $\kappa^*$ be the constant associated via Lemma \ref{Le:ResolventBounds_La}. We define
\begin{align}\label{Def:Omega}
 \Omega := \{ \lambda \in \C: \mathrm{Re} \lambda \geq  -\tfrac{\omega_0}{2}, |\la| \leq \kappa^* \}, 
 \end{align}
and
 $\Omega'  :=\{ \lambda \in \C: \mathrm{Re} \lambda \geq -\tfrac{\omega_0}{2}\} \setminus \Omega$.
By Lemma \ref{Le:ResolventBounds_La},  we have $\Omega' \subset \rho(\mb L_a)$ and it is left to investigate the spectrum in the compact region $\Omega$. By Lemma \ref{Le:SpecLa_Prelim}, we know that there are only isolated eigenvalues.  First, one can check by a direct calculation that $\mb h_a, \mb g^{(k)}_a, \mb q^{(j)}_{a}$ are eigenfunctions corresponding to the eigenvalues $3$, $1$ and $0$. The Lipschitz estimates for the eigenfunctions follow from the fact the they are smooth provided that $a$ is small enough. It is left to show that there are no other eigenvalues. First, we claim that $\partial \Omega \in \rho(\mb L_a)$. For this, we need information on the line segment $\Gamma:= \{\la \in \C:\mathrm{Re} \la =   -\tfrac{\omega_0}{2}, |\la| \leq \kappa^*\}$, which  is contained in the resolvent set of $\mb L_0$ by Proposition \ref{Prop:Spectrum_L0}. In view of the identity 
\begin{align}\label{Eq:IdentityLa}
 \lambda - \mb L_a = [1 - (\mb L'_a - \mb L'_0) \mb R_{\mb L_0}(\lambda)](\lambda  - \mb L_0), 
 \end{align}
it suffices to show that 
\[ \|\mb L'_0 - \mb L'_a \| \| \mb R_{\mb L_0}(\lambda)\| < 1 ,\]
for all $\la \in \Gamma$ and for all $a \in \overline{\mathbb B^7_{\delta}}$ with $\delta >0$ small enough. 
This follows from the Lipschitz continuity of $ \mb L'_a$, see Lemma \ref{Le:Perturbation}, if we require that 
$\delta < \frac{1}{2KC}$, where $C := \max_{\la \in \Gamma} \| \mb R_{\mb L_0}(\lambda) \|$.
Having this, we define a projection
\[ \mb {\tilde  T_a}=  \frac{1}{2 \pi i } \int_{\partial \Omega} \mb R_{\mb L_a}(\lambda) d \lambda,\]
which depends continuously on $a$ for small enough values of the parameter. This follows from the continuity of $a \mapsto \mb R_{\mb L_a}(\lambda)$ for small enough $a$, which can be seen for example from Eq.~\eqref{Eq:IdentityLa}.
 For $a = 0$,   $ \mb  {\tilde T_0}$ has rank $16$, see Lemma  \ref{Le:Rank_P0}. By \cite{kato}, p.~34, Lemma $4.10$, we infer that $\mathrm{dim}~\mathrm{rg }~  \mb  {\tilde T_a}= 16$ for $a$ small enough, which excludes the existence of other eigenvalues. 
\end{proof}

\begin{remark}\label{Remark:Lorentz_instability}
The eigenfunctions corresponding to the eigenvalues $\la = 0$ and $\la = 1$ originate from the fact that we are perturbing around a family of solutions depending on 
several symmetry parameters. For Lorentz  boosts, this can be seen most easily. Since $ \Psi_a^*$ satisfies the equation $\mb L \Psi_a^* + \mb N( \Psi_a^*) = 0$, the chain rule implies that $(\mb L  + \mb N'(  \Psi_a^*)) \partial_{a^j} \Psi_a^* = \mb L_a  \partial_{a^j} \Psi_a^*  =0$, i.e., $\partial_{a^j} \Psi_a^* $ solves the eigenvalue problem for $\la = 0$. In fact, it can easily be checked that $ \partial_{a^j} \Psi_a^* =  \mb q_a^{(j)}$. 
\end{remark}

\subsection{Growth bounds for the semigroup}

Define 
\begin{align*}
\mb H_a := \frac{1}{2 \pi i} \int_{\gamma_2} \mb R_{\mb L_a}(\lambda) d\lambda,
 \quad \mb P_a := \frac{1}{2 \pi i} \int_{\gamma_1} \mb R_{\mb L_a}(\lambda) d\lambda, \quad \mb Q_a := \frac{1}{2 \pi i} \int_{\gamma_0} \mb R_{\mb L_a}(\lambda) d\lambda,
\end{align*}
where $\gamma_j (s) := \lambda_j +  \frac{\omega_0}{4} e^{2\pi i s}$ for $s \in [0,1]$.

\begin{lemma}\label{Le:Projections_Pa}
Let  $ \delta > 0$  be sufficiently small, then,
\[ \mathrm{rg} \mb H_a =  \mathrm{span}( \mb h_a ), \quad  \mathrm{rg}\mb P_a  = \mathrm{span}(\mb g_{a}^{(0)}, \dots, \mb g_{a}^{(7)}  ),  \quad \mathrm{rg} \mb Q_a = \mathrm{span}(   \mb q_{a}^{(1)}, \dots, \mb q_{a}^{(7)}) \]
for all $a \in \overline{\mathbb B^7_{\delta}}$.
Furthermore, the projections are mutually transversal, 
\[ \mb H_a \mb P_a = \mb P_a \mb H_a  = \mb H_a \mb Q_a =  \mb Q_a  \mb H_a =  \mb Q_a \mb P_a =  \mb P_a  \mb Q_a  = 0 \]
and  depend Lipschitz continuously on the Lorentz parameter, i.e.,  
\[ \|\mb H_a - \mb H_b  \| + \| \mb P_a  -\mb P_b   \|   + \| \mb Q_a  - \mb Q_b  \| \lesssim | a - b | \]
for all $a,b \in  \overline{\mathbb B^7_{\delta}}$. 
\end{lemma}

\begin{proof}
As in the proof of Lemma \ref{Le:Spectrum_La}, the dimension of the ranges of the projections is a consequence of the continuity of the projections with respect to the parameter $a$. Transversality follows from the definition. For the  Lipschitz bounds, we use the second resolvent identity and Corollary \ref{Cor:TimeEvol_La} which imply that 
\[ \|\mb R_{\mb L_a}(\lambda)  -\mb R_{\mb L_b}(\lambda)  \| \leq \|\mb R_{\mb L_a}(\lambda)\|\|\mb L_a - \mb L_b \| \| \mb R_{\mb L_b}(\lambda)  \| \lesssim |a-b| \]
for all $\la \in \mathrm{rg~}\gamma_j$ and $a,b \in  \overline{\mathbb B^7_{\delta}}$ for $\delta > 0$ small enough.
\end{proof}

Since $\mb P_a$ and $\mb Q_a$ are operators of finite rank, every $\mb f \in \mc H$ has the unique expansion
\[ \mb P_a \mb f = \sum_{k = 0}^{7} \alpha_k \mb g_{a}^{(k)}, \quad  \mb Q_a \mb f = \sum_{j = 1}^{7} \beta_j \mb q_{a}^{(j)}, \]
for $\alpha_k, \beta_j \in \mathbb C$.  We define 
\[  \mb P_a^{(k)}\mb f := \alpha_k \mb g_{a}^{(k)},  \quad \mb Q^{(j)}_a \mb f := \beta_j  \mb q_{a}^{(j)}, \]
such that $\mb P_a = \sum_{k=0}^{7}   \mb P_a^{(k)}$, 
 $\mb Q_a = \sum_{j=1}^{7}   \mb Q_a^{(j)}$ and
\[ \mb P_a^{(k)} \mb P_a^{(l)} = \delta_{kl} \mb P_a^{(k)}, \quad \mb Q_a^{(i)} \mb Q_a^{(j)} = \delta_{ij } \mb Q_a^{(j)}. \]
Finally, we define
\[ \mb T_a : = \mb I - \mb H_a -   \mb P_a  -   \mb Q_a,\]
which is Lipschitz continuous with respect to $a$ by Lemma \ref{Le:Projections_Pa}.
Note that the projections $\mb T_a, \mb H_a$, $\mb P_a^{(k)}$ and $\mb Q_a^{(j)}$ are mutually transversal. Furthermore, is easy to see that 
the Lipschitz continuity of $\mb Q_a$ and the eigenfunctions $\mb q_a^{(j)}$, $j= 1, \dots,7$ imply that 
\begin{align}\label{Eq:Lipschitz_Q_a,j}
\| \mb Q_a^{(j)}  - \mb Q_b^{(j)}  \| \lesssim |a-b|,
\end{align}
for all $a,b \in \overline{\mathbb B^7_{\delta}}$.
Similarly,
\begin{align*}
\| \mb P_a^{(k)}  - \mb P_b^{(k)}  \| \lesssim |a-b|,
\end{align*}
for $k =0,\dots,7$. 

\begin{theorem}\label{Th:Decay_linearized}
The projections commute with the semigroup, 
\[ [ \mb S_a(\tau) ,  \mb H_a ] = [ \mb S_a(\tau) ,  \mb P_a^{(k)} ] = [ \mb S_a(\tau) ,\mb Q_a^{(j)}  ]   = 0,\]
and there are constants $\delta >0$ and $\omega > 0$ such that
\begin{align*}
\mb S_a(\tau) \mb H_a  = e^{3 \tau} \mb H_a, \quad \mb S_a(\tau) \mb P^{(k)}_{a}  = e^{\tau}  \mb P^{(k)}_{a} , \quad \mb S_a(\tau) \mb Q^{(j)}_{a}  = \mb Q^{(j)}_{a},
\end{align*}
and
\begin{align}\label{Eq:Growthbound_stable}
 \| \mb S_a(\tau) \mb T_a \mb u\| \lesssim e^{-\omega \tau} \| \mb T_a \mb u\|  
\end{align}
for all $\tau \geq 0$,  $\mb u \in \mc H$,  $a \in \overline{\mathbb B^7_{\delta}}$. Furthermore, we have 
\begin{align}\label{Eq:Evol_Lipschitz_stablesubspace}
 \| \mb S_a(\tau) \mb T_a - \mb S_b(\tau) \mb T_b  \|  \lesssim e^{-\omega \tau} |a-b|, 
 \end{align}
for all  $a,b \in \overline{\mathbb B^7_{\delta}}$ and all $\tau \geq 0$.
\end{theorem}

\begin{proof}

For the semigroup acting on the ranges of the projections $\mb H_a$, $\mb P_{a}$ and $\mb Q_{a}$ we have
\begin{align*}
\mb S_a(\tau) \mb H_a \mb u  = e^{3 \tau} \mb H_a \mb u , \quad \mb S_a(\tau) \mb P_{a}\mb u  = e^{\tau}  \mb P_{a}\mb u , \quad \mb S_a(\tau) \mb Q_{a} \mb u = \mb Q_{a}\mb u,
\end{align*}
for all $\mb u \in \mc H$, $\tau \geq 0$ and sufficiently small $a$. Furthermore, the semigroup commutes with the resolvent and therefore with the projection operators,
\[ [ \mb S_a(\tau) ,  \mb H_a ] = [ \mb S_a(\tau) ,  \mb P_a] = [ \mb S_a(\tau) ,\mb Q_a  ]   = 0.\]
This implies that for $k=0, \dots,7$, we have
\begin{align*}
\mb P^{(k)}_a  \mb S_a(\tau) \mb u =  \mb P_a \mb P_a^{(k)}    \mb S_a(\tau) \mb u  =  \mb P_a^{(k)}    \mb S_a(\tau) \mb P_a\mb u  = e^{\tau}  \mb P_a^{(k)}  \mb P_a   \mb u = \mb S_a(\tau) \mb P_a^{(k)} \mb u.
\end{align*}
The argument for $\mb Q_a^{(j)}$ is analogous. Now, $\mb R_{\mb L_a} (\la) \mb T_a$ is holomorphic in $\Omega$, see Eq.~\eqref{Def:Omega}, and uniformly bounded with respect to $a \in \overline{\mathbb B^7_{\delta}}$. In view of Lemma \ref{Le:ResolventBounds_La}, we infer that there is a constant $c  > 0$ such that 
\[ \| \mb R_{\mb L_a} (\la) \mb T_a \| \leq c \]
for all $\la \in \C$ with $\mathrm{Re} \la \geq -\frac{\omega_0}{2} $ and all $a \in \overline{\mathbb B^7_{\delta}}$.
An  application of the Gearhart-Pr\"uss Theorem, see \cite{Pru84}, Proposition 2, shows that for every $\varepsilon >0$, there exists a constant $C_{\varepsilon}$ such that 
\begin{align}\label{Eq:Growthbound_stable1}
 \| \mb S_a(\tau) \mb T_a \mb u\| \leq C_{\varepsilon} e^{-(\frac{\omega_0}{2} - \varepsilon)\tau } \| \mb T_a \mb u\|  
\end{align}
for all $\mb u \in \mc H$ and all $a \in  \overline{\mathbb B^7_{\delta}}$. Eq.~\eqref{Eq:Evol_Lipschitz_stablesubspace} is obtained analogously to \cite{DonSch14b}, Lemma $3.9$. One can easily check that for $\mb u \in \mc D(\mb L_a)$, the function 
\[ \Phi_{a,b}(\tau) := \frac{\mb S_a(\tau) \mb T_a \mb u - \mb S_b(\tau) \mb T_b \mb u}{|a-b|} \]
satisfies the inhomogeneous equation
\begin{align}\label{Eq:Aux_abstractODE}
\partial_{\tau} \Phi_{a,b}(\tau)  = \mb L_a \mb T_a \Phi_{a,b}(\tau) + \frac{\mb L_a \mb T_a - \mb L_b \mb T_b}{|a-b|} \mb S_b(\tau) \mb T_b \mb u
\end{align}
with initial data $\Phi_{a,b}(0) =  \frac{ \mb T_a \mb u - \mb T_b \mb u}{|a-b|}  $ for all $\tau \geq 0$.
We have $\mb L_a \mb T_a = \mb L_a ( 1 - \mb H_a - \mb P_a - \mb Q_a) = \mb L_a - 3 \mb H_a - \mb P_a$, such that
\[ \mb L_a \mb T_a - \mb L_b \mb T_b =  \mb L'_a - \mb L_b' - 3 (\mb H_a - \mb H_b) - (\mb P_a -\mb P_b), \]
which implies that 
\[ \| \mb L_a \mb T_a - \mb L_b \mb T_b\| \lesssim |a-b| \]
by Lemma \ref{Le:Perturbation} and Lemma \ref{Le:Projections_Pa}. The integral equation associated to Eq.~\eqref{Eq:Aux_abstractODE} by the Duhamel principle is given by 
\begin{align*}
\Phi_{a,b}(\tau) = \mb S_a(\tau)  \mb T_a \frac{ \mb T_a \mb u - \mb T_b \mb u}{|a-b|} + \int_0^{\tau} \mb S_a(\tau - \tau')\mb T_a  
\frac{\mb L_a \mb T_a - \mb L_b \mb T_b}{|a-b|} \mb S_b(\tau') \mb T_b \mb u ~ d\tau'. 
\end{align*}
 Eq.~\eqref{Eq:Growthbound_stable1} yields the bound
\[ \| \Phi_{a,b}(\tau) \| \lesssim  e^{-(\frac{\omega_0}{2} - \varepsilon)\tau } (1 + \tau)  \|\mb u \| \lesssim e^{-(\frac{\omega_0}{2} - 2 \varepsilon)\tau }   \|\mb u \|,\]
which extends to all of $\mc H$ by density.  We fix $\varepsilon > 0$ such that  $\omega := \frac{\omega_0}{2} - 2 \varepsilon > 0$. This yields Eq.~\eqref{Eq:Evol_Lipschitz_stablesubspace} and Eq.~\eqref{Eq:Growthbound_stable}.
\end{proof}

\section{Nonlinear perturbation theory}\label{Sec:Nonlin_Pert}

\subsection{Function spaces and basic estimates}

Let $\omega > 0$ be fixed by Theorem \ref{Th:Decay_linearized}. We define spaces
\[ \mc X := \{ \Phi \in C([0,\infty), \mc H) : \| \Phi \|_{\mc X} < \infty \}, \quad  \| \Phi \|_{\mc X} :=   \sup_{\tau > 0} e^{ \omega \tau} \| \Phi(\tau) \|, \]
\[ X := \{ a \in C^1([0,\infty), \R^7) : a(0) = 0, \|a \|_X < \infty \}, \quad \|a \|_X := \sup_{\tau > 0} [ e^{\omega \tau} |\dot a(\tau)| + |a(\tau)|]. \]
For $a \in X$, we can write $a(\tau) = \int_{0}^{\tau} \dot a (\sigma) d\sigma$
and the integral converges in the limit $\tau \to \infty$. Hence, we define
\[ a_{\infty}  = \lim_{\tau \to \infty} a(\tau).\]
In the following, we assume that $\tilde \delta > 0$ is sufficiently small, such that the results of the  preceding sections hold for all $a \in \overline{\mathbb B^7_{\tilde \delta}}$.  
For $\delta  > 0$ satisfying $\frac{\delta}{\omega} < \tilde \delta$, we set
\[ \mc X_{\delta} := \{ \Phi \in \mc X: \| \Phi \| \leq \delta \}, \quad   X_{\delta} := \{ a \in X: \sup_{\tau > 0} [e^{\omega \tau} |\dot a(\tau)|]  \leq \delta  \}. \]

For $a \in X_{\delta}$, $|a(\tau)| \leq \delta/\omega < \tilde \delta$ for all $\tau \geq 0$. In the following, we will frequently use that 
\begin{align}\label{Eq:LorentzPar1}
 |a_{\infty} - a(\tau)| \leq \int_{\tau}^{\infty} | \dot a (\sigma)| d\sigma  \leq \tfrac{\delta}{\omega} e^{-\omega \tau},
\end{align}
and  
$ |a(\tau)- b(\tau) | \leq  \| a - b \|_X $
for  all $a,b \in X_{\delta}$ and all $\tau \geq 0$. In particular, $|a_{\infty}- b_{\infty} | \leq  \| a - b \|_X$.

In the following, we provide some bounds for the function $\mb G_a$ defined by 
 \[ \mb G_{a(\tau)}(\Phi(\tau)) = [\mb L'_{a(\tau)} - \mb L'_{a_{\infty}}]\Phi(\tau) + \mb F_{a(\tau)}(\Phi(\tau)),\]
with $\mb F_a$ given in Eq.~\eqref{Def:Potential}. We start with some basic estimates for the nonlinear part and denote by $\mc B \subset \mc H$ the unit ball in $\mc H$. 

\begin{lemma}\label{Est:Nonlinearity}
Let $\delta > 0$ be sufficiently small. Then
\[ \| \mb F_a(\mb u) - \mb F_b(\mb v) \| \lesssim (\| \mb u \| + \| \mb v \|) \| \mb u - \mb v\| + (\| \mb u \|^2 + \| \mb  v\|^2) |a - b| \]
for all $\mb u, \mb v \in \mc B \subset \mc H$ and all $a, b \in \overline{\mathbb B^7_{\delta}}$.
\end{lemma}

\begin{proof}
First, we show that 
\begin{align*}
\| \mb F_a(\mb u) - \mb F_a(\mb v) \| \lesssim (\| \mb u \| + \| \mb v \|) \| \mb u - \mb v\| 
\end{align*}
for all $a \in \overline{\mathbb B^7_{\delta}}$ and all $\mb u,\mb v  \in \mc B$, more precisely, we prove that
\begin{align*}
\| u^3 - v^3 \|_{H^2(\B^7)} + \| \psi^*_a (u^2 - v^2) \|_{H^2(\B^7)} \lesssim  (\|u\|_{H^3(\B^7)} + \|v\|_{H^3(\B^7)} ) \|u - v\|_{H^3(\B^7)}
\end{align*}
for all $u,v \in H^3(\mathbb B^7)$. The Sobolev embedding $W^{j+m,2}(\B^d) \hookrightarrow W^{j,q}(\B^d) $ for $2 \leq q \leq \frac{2d}{d-2m}$, $j \in \N_0$, implies that 
\[ \| \partial^{\alpha} u \|_{L^q(\B^7)} \lesssim \|u\|_{H^3(\B^7)} \]
for multi-indices $\alpha \in \N^7$ with $0 \leq |\alpha| \leq 2$ and $2 < q \leq \frac{14}{1+2|\alpha|}$. Using this, we show that 
\begin{align}\label{Eq:Trlin_Est}
\|u v w\|_{H^2(\mathbb B^7)} \lesssim \|u \|_{H^3(\mathbb B^7)} \|v \|_{H^3(\mathbb B^7)} \|w \|_{H^3(\mathbb B^7)}
\end{align} 
for all $u,v,w \in H^3(\mathbb B^7)$. First, observe that H\"older's inequality with $q_1 = \frac{14}{5}$, $q_2 = 14$, $\frac{1}{q_1} + \frac{2}{q_2} = \frac12$ implies 
\begin{align*}
\| u v w \|_{L^2(\mathbb B^7)}  \lesssim 
\|u\|_{L^{q_1}(\mathbb B^7)} \|v\|_{L^{q_2}(\mathbb B^7)} \|w\|_{L^{q_2}(\mathbb B^7)}  \lesssim \|u \|_{H^3(\mathbb B^7)} \|v \|_{H^3(\mathbb B^7)} \|w \|_{H^3(\mathbb B^7)} .
 \end{align*}
For $|\alpha| = 1$ we have
\[ \|v  w  \partial^{\alpha} u  \|_{L^2(\mathbb B^7)}  \lesssim \| \partial^{\alpha} u\|_{L^{q_1}(\mathbb B^7)}   \|v\|_{L^{q_2}(\mathbb B^7)} \|w\|_{L^{q_2}(\mathbb B^7)}  \lesssim \|u \|_{H^3(\mathbb B^7)} \|v \|_{H^3(\mathbb B^7)} \|w \|_{H^3(\mathbb B^7)},  \]
where $q_1 = \frac{14}{3}, q_2 = 7$.  For $|\alpha| = 2$, $\alpha_j \in \N^{7}$, $j=1,2,3$, with $\sum_{j} \alpha_j = \alpha$ we set $q_j =\frac{14}{1+2|\alpha_j|}$ to obtain 
\begin{align*}
\| \partial^{\alpha_1} u   \partial^{\alpha_2} v   \partial^{\alpha_3} w \|_{L^2(\mathbb B^7)}  \lesssim 
\|\partial^{\alpha_1} u \|_{L^{q_1}(\mathbb B^7)} \|\partial^{\alpha_2}  v\|_{L^{q_2}(\mathbb B^7)} \|  \partial^{\alpha_3}  w\|_{L^{q_3}(\mathbb B^7)}  \lesssim \|u \|_{H^3} \|v \|_{H^3(\mathbb B^7)} \|w \|_{H^3(\mathbb B^7)}.
 \end{align*} 
Eq.~\eqref{Eq:Trlin_Est} now follows by applying the Leibnitz rule.  Consequently, 
\[ \| u^3 - v^3 \|_{H^2(\mathbb B^7)} \lesssim (\|u\|_{H^3(\mathbb B^7)}^2 + \|v\|_{H^3(\mathbb B^7)}^2) \|u-v\|_{H^3(\mathbb B^7)}. \]
Since  $\psi^*_a$ is smooth and uniformly bounded for sufficiently small $a$, we obtain
\[ \| \psi^*_a (u^2 - v^2) \|_{H^2(\B^7)}  \lesssim  (\|u\|_{H^3(\mathbb B^7)} + \|v\|_{H^3(\mathbb B^7)}) \|u-v\|_{H^3(\mathbb B^7)}. \] 
Finally, in view of Eqns. \eqref{Eq:Trlin_Est} and \eqref{Eq:Selfsim_Sol_Lipschitz} we have 
\begin{align*}
 \| \mb F_a(\mb u) - \mb F_b(\mb u) \| \lesssim \|u^2 (\psi^*_a - \psi^*_b) \|_{H^2(\mathbb B^7)} \lesssim \|u \|^2_{H^3(\mathbb B^7)} \|\psi^*_a - \psi^*_b \|_{H^3(\mathbb B^7)}  \lesssim  \|u \|^2_{H^3(\mathbb B^7)}  |a - b|,
\end{align*}
for all $a, b \in \overline{\mathbb B^7_{\delta}}$ which implies the claim.
\end{proof}

\begin{lemma}\label{Le:Bound_Ga}
Let $\delta > 0$ be sufficiently small. Then
\begin{align}\label{Eq:Est_RHS}
\begin{split}
\|  \mb G_{a(\tau)}(\Phi(\tau)) \| &  \lesssim \delta^2 e^{-2 \omega \tau},\\
\|  \mb G_{a(\tau)}(\Phi(\tau)) -   \mb G_{b(\tau)}(\Psi(\tau))\|  &  \lesssim \delta e^{-2 \omega \tau} ( \| \Phi - \Psi \|_{\mc X} + \| a - b \|_X ),
\end{split}
\end{align}
for all $\Phi, \Psi \in \mc X_{\delta}, a,b \in X_{\delta}$, and $\tau \geq 0$. 
\end{lemma}

\begin{proof}
 Since $\mb F_a(0) = 0$ for all $a$, Lemma \ref{Est:Nonlinearity}, Lemma \ref{Le:Perturbation} and Eq.~\eqref{Eq:LorentzPar1} immediately imply the first estimate. For the Lipschitz bounds, we obtain
 \[ \|  \mb F_{a(\tau)}(\Phi(\tau)) -   \mb F_{b(\tau)}(\Psi(\tau))\|   \lesssim \delta e^{-2 \omega \tau} ( \| \Phi - \Psi \|_{\mc X} + \| a - b \|_X ) \]
from Lemma \ref{Est:Nonlinearity}. Furthermore, 
\[ \|[\mb L'_{b(\tau)} - \mb L'_{b_{\infty}} ](\Phi(\tau) -  \Psi(\tau)) \| \lesssim \delta e^{-2 \omega \tau} \| \Phi - \Psi \|_{\mc X} \]
by Lemma \ref{Le:Perturbation} and Eq.~\eqref{Eq:LorentzPar1}. Finally, we use that 
\[  \psi^*_{a_{\infty}}(\xi)^2 - \psi^*_{a(\tau)}(\xi)^2 =  \int_{\tau}^{\infty} \dot a^{k}(s) \varphi_{a(s),k}(\xi) ds \]
with $\varphi_{a,k}(\xi) = \partial_{a^k} \psi^*_a(\xi)^2$,
to obtain the estimate
\[  \| [\mb L'_{a(\tau)} - \mb L'_{a_{\infty}}] - [\mb L'_{b(\tau)} - \mb L'_{b_{\infty}}] \|  \lesssim e^{- \omega \tau} \|a - b \|_X, \]
see the proof of Lemma 5.5 in \cite{DonSch14b} for the details. These bounds imply the second line in Eq.~\eqref{Eq:Est_RHS}.
\end{proof}

\subsection{Integral equation for the peturbation}\label{Sec:Modulation_CorrectedData}

To solve Eq.~\eqref{Eq:SelfSim_Perturb}, we first study the general  initial value problem,
\begin{align}\label{Eq:SelfSimEvol_CP}
\begin{split}
\partial_{\tau} \Phi(\tau) & = \mb L_{a_{\infty}}\Phi(\tau) + \mb G_{a(\tau)}(\Phi(\tau))  - \partial_{\tau} \Psi^*_{a(\tau)} , \quad \tau > 0   \\
\Phi(0) &  = \mb u
\end{split}
\end{align}
for $\mb u \in \mc H$. In fact, we are interested in solutions of the corresponding integral equation
\begin{align}\label{Eq:Integral_Equation_Perturbation}
\Phi(\tau) = \mb S_{a_{\infty}}(\tau) \mb u + \int_0^{\tau} 
 \mb S_{a_{\infty}}(\tau - \sigma)[ \mb G_{a(\sigma)}(\Phi(\sigma)) - \partial_{\sigma} \Psi^*_{a(\sigma)} ] d \sigma.
\end{align}

In the following, we use modulation theory to prove the existence of solutions in $\mc X_{\delta}$ under a co-dimension 9 condition on the initial data $\mb u 
\in \mc H$ (corresponding
to the unstable directions defined by $\mb g_a^{(k)}$, $k = 0,\dots,7$ and the genuine unstable mode $\mb h_a$). This is along the lines of \cite{DonSch14b}, Section 5.2 - Section 5.4. 
In Section \ref{Sec:Codim1} below we use the specific form of the initial data in Eq.~\eqref{Eq:SelfSim_Perturb} to remove the translation instabilities by fixing suitable parameters $(T, x_0)$ and to prove co-dimension one stability. 

\subsubsection{The modulation equation}
 
First, we take care of the Lorentz instability by deriving a suitable modulation equation for the parameter $a$. Here, it is crucial that $\partial_{\tau} \Psi^*_{a(\tau)} = \dot a_{j}(\tau) \mb q^{(j)}_{a(\tau)} = \sum_{j=1}^{7} \dot a^{j}(\tau) \mb q^{(j)}_{a(\tau)} $, see Remark \ref{Remark:Lorentz_instability}.

We introduce a  cut-off function $\chi: [0,\infty) \to [0,1]$ satisfying $\chi(\tau) = 1$ for $\tau \in [0,1]$,  $| \chi'(\tau) | \leq 1$ for all $\tau \geq 0$ and $\chi(\tau) = 0$ for $\tau \geq 4$, and require that
\begin{align} \label{Eq:Modulation_Lorentzinstability}
\mb Q^{(j)}_{a_{\infty}} \Phi(\tau)= \chi(\tau) \mb Q^{(j)}_{a_{\infty}}  \mb u
\end{align}
for all $\tau \geq 0$. 
Applying $\mb Q^{(j)}_{a_{\infty}}$ to Eq.~\eqref{Eq:Integral_Equation_Perturbation} and using Theorem \ref{Th:Decay_linearized} yields 
\begin{align*}
[ 1 - \chi(\tau) ]\mb Q^{(j)}_{a_{\infty}}   \mb u  +  \int_0^{\tau} 
  [\mb Q^{(j)}_{a_{\infty}}  \mb G_{a(\sigma)}(\Phi(\sigma)) -  \mb Q^{(j)}_{a_{\infty}} \dot a_{i}(\sigma) \mb q^{(i)}_{a(\sigma)}
 ] d \sigma = 0.
\end{align*}
In view of  $\mb Q^{(j)}_{a_{\infty}} \mb q^{(i)}_{a_{\infty}} = \delta^{ij}  \mb q^{(j)}_{a_{\infty}}$ and  $a(0) =0$ this can be written as
\begin{align*}
a^j(\tau) \mb q^{(j)}_{a_{\infty}} & =  - \int_0^{\tau} \chi'(\sigma)\mb Q^{(j)}_{a_{\infty}}   \mb u~ d \sigma + \int_0^{\tau}\left[\mb Q^{(j)}_{a_{\infty}}   \mb G_{a(\sigma)}(\Phi(\sigma))  -
\mb Q^{(j)}_{a_{\infty}} \dot a_{i}(\sigma)[\mb q^{(i)}_{a(\sigma)}-\mb q^{(i)}_{a_{\infty}}] \right]  d \sigma \\
 & =: \int_0^{\tau} \mb A_j(a, \Phi, \mb u) (\sigma) d\sigma.
\end{align*}
Thus, we obtain the equation
\begin{align}\label{Eq:Modulation}
a(\tau) = A(a, \Phi, \mb u)(\tau)
\end{align} 
for $a \in X_{\delta}$, where $A = (A_1, \dots, A_7)$, and 
\[ A_j(a, \Phi, \mb u)(\tau) :=   \| \mb q^{(j)}_{a_{\infty}} \|^{-2}  \int_0^{\tau}(\mb A_j(a, \Phi, \mb u)(\sigma)| \mb q^{(j)}_{a_{\infty}}  ) d\sigma .\]

\begin{lemma}\label{Le:Modulation}
Let $\delta > 0$ be sufficiently small and $c > 0$ be sufficiently large. For every
$\mb u \in \mc H$ satisfying $\|\mb u \| \leq \frac{\delta}{c}$ and every $\Phi \in \mc X_{\delta}$, there is a unique $a = a(\Phi, \mb u)  \in X_{\delta}$ such that Eq.~\eqref{Eq:Modulation} holds. Furthermore, 
\[ \|a(\Phi, \mb u)  - a(\Psi, \mb v)  \|_{X} \lesssim  \|\Phi - \Psi \|_{\mc X}  + \|\mb u - \mb v\| \]
for all $\Psi, \Phi \in \mc X_{\delta} $ and $\mb u,\mb v \in \mc B_{\delta/c}$.
\end{lemma}

\begin{proof}
We use a fixed point argument and show that under the above assumptions $A( \cdot, \Phi, \mb u) : X_{\delta} \to X_{\delta}$ defines a contraction.  
We have
\[   \| \mb A_j(a, \Phi, \mb u)(\tau)\| \lesssim (\tfrac{\delta}{c} + \delta^2) e^{-2 \omega \tau} \]
provided $\delta >0$ is sufficiently small and $c > 0$ is sufficiently large. This can be seen by using the bounds of Lemma \ref{Le:Bound_Ga}, the fact that
$\| \chi'(\tau)\mb Q^{(j)}_{a_{\infty}}   \mb u \| \lesssim \tfrac{\delta}{c} e^{-2 \omega \tau}$
and
\[ \|\dot a^{j}(\tau)[\mb q^{(j)}_{a(\tau)}-\mb q^{(j)}_{a_{\infty}}] \| \lesssim \delta e^{-\omega \tau} |a(\tau) -a_{\infty}| \lesssim \delta^2 e^{-2 \omega \tau}. \]
This implies that $| \dot A(a, \Phi, \mb u)(\tau)| \leq \delta e^{-2 \omega \tau}$ and hence $A( \cdot, \Phi, \mb u) : X_{\delta} \to X_{\delta}$. Next, we show that 
\begin{align}\label{Eq:Lipschitz_Aj}
\| \mb A_j(a, \Phi, \mb u)(\tau) -  \mb A_j(b, \Phi, \mb u)(\tau)   \| \lesssim \delta e^{-2 \omega \tau} \|a -b \|_X,
\end{align}
 for all $a,b, \in X_{\delta}$, which implies that 
\begin{align}\label{Eq:Lipschitz_A}
  \|A(a, \Phi, \mb u) -  A(b, \Phi, \mb u) \|_X \lesssim \delta \|a - b\|_X,
  \end{align}
by using the Lipschitz continuity of $a\mapsto\mb q_{a}^{(j)}$.
To prove Eq.~\eqref{Eq:Lipschitz_Aj} we use the Lipschitz bounds of Lemma \ref{Le:Bound_Ga} and Eq.~\eqref{Eq:Lipschitz_Q_a,j} to obtain
\[   \| \mb Q^{(j)}_{a_{\infty}}    \mb G_{a(\tau)}(\Phi(\tau)) - \mb Q^{(j)}_{b_{\infty}}    \mb G_{b(\tau)}(\Psi(\tau)) \| \lesssim \delta e^{-2 \omega \tau} \left ( \|a -b \|_X  +   \|\Phi - \Psi \|_{\mc X}  \right),\]
and 
$\| \chi'(\tau)[\mb Q^{(j)}_{a_{\infty}}    -\mb Q^{(j)}_{b_{\infty}}   ]  \mb u \| \lesssim \delta e^{-2 \omega \tau} \|a -b \|_X$. 
Furthermore, we have
\begin{align}\label{Eq:Bounds_hatq}
\|\mb q^{(j)}_{a(\tau)} - \mb q^{(j)}_{a_{\infty}} \| \lesssim \delta e^{-\omega \tau} 
\end{align}
and 
\begin{align}\label{Eq:Lipschitzbounds_hatq}
\|(\mb q^{(j)}_{a(\tau)} - \mb q^{(j)}_{a_{\infty}}) - (\mb q^{(j)}_{b(\tau)} - \mb q^{(j)}_{b_{\infty}})   \| \lesssim  e^{-\omega \tau} \|a - b\|_X.
\end{align}
The last estimate is a consequence of the fact that 
\[ \mb q^{(j)}_{a(\tau)}(\xi) - \mb q^{(j)}_{a_{\infty}}(\xi)  = - \int_{\tau}^{\infty} \dot a^{i} (\sigma) \varphi^{(j)}_{i}(\xi, a(\sigma))  d\sigma, \]
where $\varphi^{(j)}_{i}(\xi,a)  = \partial_{a^{i}} \mb q^{(j)}_{a}(\xi) $ is smooth with respect to both variables for small enough $a$. 
Thus,
\[ \| \mb Q^{(j)}_{a_{\infty}} \dot a^{i}(\tau)(\mb q^{(i)}_{a(\tau)} - \mb q^{(i)}_{a_{\infty}})   -    \mb Q^{(j)}_{b_{\infty}} \dot b^{i}(\tau)(\mb q^{(i)}_{b(\tau)} - \mb q^{(i)}_{b_{\infty}})   \| \lesssim
 \delta e^{-2 \omega \tau}   \|a -b \|_X. \]
By combining these estimates, we obtain Eq.~\eqref{Eq:Lipschitz_Aj}. We conclude that Eq.~\eqref{Eq:Modulation} has a unique fixed point in $X_{\delta}$ for $\delta >0$ sufficiently small.  It is left to show the Lipschitz continuity of the solution map. Let
$a = A(a, \Phi, \mb u)$, $b = A(b, \Psi, \mb v)$. It is easy to see that 
\[  \| A(b,\Phi, \mb u)- \ A(b,\Phi, \mb v)\|_X \lesssim \| \mb u - \mb v\|. \]
and $ \| A(b,\Phi, \mb v) - \ A(b,\Psi, \mb v) \|_X \lesssim \delta \|\Phi - \Psi \|_{\mc X} $.
By combining these bounds with Eq.~\eqref{Eq:Lipschitz_A} we infer that 
\begin{align*}
\| a - b\|_X  & \leq \| A(a,\Phi, \mb u)-  A(b,\Phi, \mb u)\|_X +   \| A(b,\Phi, \mb u)-  A(b,\Phi, \mb v)\|_X   \\
& +  \| A(b,\Phi, \mb v)-  A(b,\Psi, \mb v)\|_X   \lesssim  \delta  \|a-b\|_X + \|\mb u -\mb v\| +  \delta \|\Phi - \Psi \|_{\mc X}.
\end{align*}
The claim follows by choosing $\delta >0$ sufficiently small.
\end{proof}

\subsubsection{Global existence for modified initial data} 
To control the remaining instabilities we  define correction terms 
\begin{align*}
\mb C_1(\Phi,a,\mb u) &  : = \mb P_{a_{\infty}} \mb u +\mb P_{a_{\infty}} \int_0^{\infty} e^{-\sigma} [\mb G_{a(\sigma)}(\Phi(\sigma)) - \partial_{\sigma} \Psi^*_{a(\sigma)} ] d \sigma, \\
\mb C_2(\Phi,a,\mb u) &  : = \mb H_{a_{\infty}} \mb u +\mb H_{a_{\infty}} \int_0^{\infty} e^{-3 \sigma} [\mb G_{a(\sigma)}(\Phi(\sigma)) - \partial_{\sigma} \Psi^*_{a(\sigma)} ] d \sigma,
\end{align*}
and set $\mb C := \mb C_1 + \mb C_2$.  In the following, we prove that Eq.~\eqref{Eq:SelfSimEvol_CP} with modified initial data $\Phi(0) = \mb u -  \mb C(\Phi, a, \mb u)$ has a solution in $\mc X$. For this we study the integral equation
\begin{align}\label{Eq:Mod_Integral}
\begin{split}
\Phi(\tau) & = \mb S_{a_{\infty}}(\tau) [ \mb u - \mb C(\Phi, a, \mb u)] + \int_0^{\tau} 
 \mb S_{a_{\infty}}(\tau - \sigma)[ \mb G_{a(\sigma)}(\Phi(\sigma)) - \partial_{\sigma} \Psi^*_{a(\sigma)} ] d \sigma \\
&  =: \mb K(\Phi, a, \mb u)(\tau)  
\end{split}
\end{align}
\begin{proposition}\label{Prop:NonlEvolution_Corr}
Let $c > 0$ be sufficiently large and let $\delta > 0$ be sufficiently small. If $\|\mb u \| \leq \frac{\delta}{c}$ then there exist functions $\Phi \in \mc X_{\delta}$ and  $a \in X_{\delta}$ such that Eq.~\eqref{Eq:Mod_Integral} holds for all $\tau \geq 0$. Furthermore, the map $\mb u \mapsto (\Phi(\mb u), a(\mb u))$ is Lipschitz continuous, i.e., 
\[  \|\Phi(\mb u) - \Phi(\mb v)  \|_{\mc X} + \|a(\mb u ) - a(\mb v) \|_X \lesssim \| \mb u - \mb v\| \]
for all $\mb u, \mb v \in \mc B_{\delta/c}$ .
\end{proposition}

\begin{proof}
First, let $c> 0, \delta > 0$ be such that Lemma \ref{Le:Modulation} holds.  For fixed $\mb u \in \mc B_{\delta/c}$ there is a unique $a = a_{\mb u}(\Phi) \in X_{\delta}$ associated to every $\Phi \in  \mc X_{\delta}$, such that Eq.~\eqref{Eq:Modulation} and hence Eq.~\eqref{Eq:Modulation_Lorentzinstability} are satisfied. We define $ \mb K_{\mb u}( \Phi) := \mb K(\Phi, a, \mb u)$ and show that 
$\mb K_{\mb u}$ maps $\mc X_{\delta}$ into itself.  By transversality of the projections we have 
\begin{align*}
\mb P_{a_\infty} \mb K_{\mb u}(\Phi)(\tau) = - \int_{\tau}^{\infty} e^{\tau - \sigma} \mb P_{a_\infty}  [ \mb G_{a(\sigma)}(\Phi(\sigma)) - \partial_{\sigma} \Psi^*_{a(\sigma)} ] d \sigma,\\
\mb H_{a_\infty} \mb K_{\mb u}(\Phi)(\tau) = - \int_{\tau}^{\infty} e^{3(\tau - \sigma)} \mb H_{a_\infty}  [ \mb G_{a(\sigma)}(\Phi(\sigma)) - \partial_{\sigma} \Psi^*_{a(\sigma)} ] d \sigma.
\end{align*}
To estimate these expressions we write  $\partial_{\tau} \Psi^*_{a(\tau)} = \dot a_j(\tau) \mb q^{(j)}_{a_{\infty} }+ \dot a_j(\tau)  [\mb q^{(j)}_{a(\tau)}   -\mb q^{(j)}_{a_{\infty}}  ]$ and use $ \mb H_{a_\infty}  \mb q^{(j)}_{a_{\infty} } = \mb P^{(k)}_{a_\infty}   \mb q^{(j)}_{a_{\infty} }    = 0$ for $k = 0,\dots,7$ together with
Eq.~\eqref{Eq:Bounds_hatq} to infer that 
\begin{align} \label{Bounds:Proj_dtauPsi^*_1}
\| \mb H_{a_\infty} \partial_{\tau} \Psi^*_{a(\tau)} \| + \| \mb P^{(k)}_{a_\infty} \partial_{\tau} \Psi^*_{a(\tau)} \|  \lesssim \delta^2 e^{-2 \omega \tau} 
\end{align}
for all $a \in X_{\delta}$.  Note that by the same reasoning, we get
\begin{align}\label{Bounds:Proj_dtauPsi^*_2}
 \|(1 -  \mb Q_{a_\infty})\partial_{\tau} \Psi^*_{a(\tau)} \| \lesssim \delta^2 e^{-2 \omega \tau}. 
 \end{align}
With Lemma \ref{Le:Bound_Ga} we obtain
\[ \|\mb H_{a_\infty} \mb K_{\mb u}(\Phi)(\tau) \| + \|\mb P_{\infty} \mb K_{\mb u}(\Phi)(\tau) \| \lesssim \delta^2 e^{-2 \omega \tau}. \]
By Eq.~\eqref{Eq:Modulation_Lorentzinstability}, $\mb Q_{a_\infty} \mb K_{\mb u}(\Phi)(\tau) = \chi(\tau) \mb Q_{a_\infty} \mb u$
which implies that 
\[ \| \mb Q_{a_\infty} \mb K_{\mb u}(\Phi)(\tau) \| \lesssim \tfrac{\delta}{c} e^{-2 \omega \tau}. \]
On the stable subspace we have
\begin{align*}
 \mb T_{a_{\infty}}  \mb K_{\mb u}( \Phi)(\tau)  & =  [1 - \mb H_{a_\infty} - \mb P_{a_\infty} - \mb Q_{a_\infty} ] \mb K_{\mb u}( \Phi)(\tau)  \\
 & =  \mb S_{a_{\infty}}(\tau)  \mb T_{a_{\infty}}   \mb u    + \int_0^{\tau} 
 \mb S_{a_{\infty}}(\tau - \sigma)  \mb T_{a_{\infty}}  [ \mb G_{a(\sigma)}(\Phi(\sigma)) - \partial_{\sigma} \Psi^*_{a(\sigma)} ] d \sigma.
\end{align*}
With Lemma \ref{Le:Bound_Ga} and Theorem \ref{Th:Decay_linearized},
\begin{align*}
\| \mb T_{a_{\infty}}  \mb K_{\mb u}( \Phi)(\tau) \| \lesssim (\tfrac{\delta}{c} + \delta^2 ) e^{-\omega \tau}. 
\end{align*}
These bounds imply that $\mb K_{\mb u}: \mc X_{\delta} \to \mc X_{\delta}$ for $\delta>0$ sufficiently small and $c > 0$ sufficiently large. We claim that 
\begin{align}\label{Eq:K_Lipschitz}
\|\mb K_{\mb u}(\Phi) - \mb K_{\mb u}(\Psi) \|_{\mc X} \lesssim \delta \| \Phi - \Psi \|_{\mc X} 
\end{align}
for all $\Phi, \Psi \in  \mc X_{\delta}$ and $\delta$ sufficiently small. For this, we use that for all $a,b \in X_{\delta}$ with $\delta$ sufficiently small,
\begin{align}\label{Bounds:Proj_dtauPsi^*_3}
\begin{split}
\|  \mb P^{(k)}_{a_\infty} \partial_{\tau} \Psi^*_{a(\tau)}  & - \mb P^{(k)}_{b_\infty} \partial_{\tau} \Psi^*_{b(\tau)}  \|  + \|\mb H_{a_\infty} \partial_{\tau} \Psi^*_{a(\tau)} - \mb H_{b_\infty} \partial_{\tau} \Psi^*_{b(\tau)}  \|  \lesssim \delta e^{-2 \omega \tau}  \| a - b\|_X,
\end{split}
\end{align}
for $k = 0,\dots,7$ and
\begin{align*}
\begin{split}
\| (1 -  \mb Q_{a_\infty})\partial_{\tau} \Psi^*_{a(\tau)} - (1 -  \mb Q_{b_\infty})\partial_{\tau} \Psi^*_{b(\tau)}  \|  \lesssim \delta e^{-2 \omega \tau}  \| a - b\|_X.
\end{split}
\end{align*}
These estimates are proved similarly to Eqns.~\eqref{Bounds:Proj_dtauPsi^*_1}-\eqref{Bounds:Proj_dtauPsi^*_2} by using Eq.~\eqref{Eq:Lipschitzbounds_hatq} and Lemma \ref{Le:Projections_Pa} in addition. Now, let $a = a_{\mb u}(\Phi)$, $b = b_{\mb u}(\Psi) \in X_{\delta}$ be assigned to $\Phi, \Psi \in \mc X_{\delta}$ via Lemma  \ref{Le:Modulation}. With Eq.~\eqref{Bounds:Proj_dtauPsi^*_3}, Lemma \ref{Le:Bound_Ga}, Lemma \ref{Le:Projections_Pa}, 
\begin{align} \label{Eq:Lipschitz_ProjK}
\begin{split}
 \|  \mb P_{a_\infty} \mb K_{\mb u}(\Phi)(\tau)  & - \mb P_{b_\infty} \mb K_{\mb u}(\Psi)(\tau) \| + \|  \mb H_{a_\infty} \mb K_{\mb u}(\Phi)(\tau) - \mb H_{b_\infty} \mb K_{\mb u}(\Psi)(\tau) \| \\
&  \lesssim  \delta  e^{-2 \omega \tau} ( \| \Phi - \Psi\|_{\mc X} - \|a - b \|_{X} ) \lesssim   \delta  e^{-2 \omega \tau}   \|\Phi - \Psi\|_{\mc X},
\end{split}
\end{align}
where we have used the Lipschitz bounds of Lemma \ref{Le:Modulation} in the last step. Eq.~\eqref{Eq:Modulation_Lorentzinstability} implies that
\begin{align*}
\|  \mb Q_{a_\infty} \mb K_{\mb u}(\Phi)(\tau) - \mb Q_{b_\infty} \mb K_{\mb u}(\Psi)(\tau) \|&  \lesssim \delta  e^{-2 \omega \tau} \|a - b\|_{X} \lesssim  \delta  e^{-2 \omega \tau} \| \Phi - \Psi\|_{\mc X}.
\end{align*}
Finally, it is easy to check that  
\begin{align}\label{Eq:Lipschitz_ProjK_stable}
\|  \mb T_{a_{\infty}} \mb K_{\mb u}(\Phi)(\tau)  -   \mb T_{b_{\infty}} \mb K_{\mb u}(\Psi)(\tau)  \| \lesssim \delta  e^{-\omega \tau}  \| \Phi - \Psi\|_{\mc X}.
\end{align}
By combining these estimates we infer that Eq.~\eqref{Eq:K_Lipschitz} holds and the existence of a solution follows from the contraction mapping principle, which applies provided $\delta >0$ is chosen sufficiently small.  

Finally, we verify the Lipschitz continuity of the solution map. For $\mb u, \mb v \in \mc B_{\delta/c}$, there are unique $(\Phi, a), (\Psi, b) \in \mc X_{\delta} \times X_{\delta}$ such that $\mb Q_{a_{\infty}} \Phi= \chi \mb Q_{a_{\infty}}  \mb u$,  $\mb Q_{b_{\infty}} \Psi= \chi \mb Q_{b_{\infty}}  \mb v$  and 
$\Phi = \mb K(\Phi, a,\mb u)$ as well as $\Psi = \mb K(\Psi,b,\mb v)$. We show that 
\begin{align}\label{Eq:Lipschitz_Sol}
\| \Phi - \Psi \|_{\mc X}  \lesssim \|\mb u - \mb v\| 
\end{align}
for all $\mb u, \mb v \in \mc B_{\delta/c}$. First, we have
\begin{align*}
\|\mb Q_{a_\infty}  \mb K(\Phi, a,\mb u)(\tau)  & - \mb Q_{b_\infty}   \mb K(\Psi,b,\mb v)(\tau) \|  \lesssim \|\chi(\tau)[  \mb Q_{a_\infty}\mb u -  \mb Q_{b_\infty} \mb v]\ \| \\
& \lesssim  e^{-\omega \tau} (\delta \| a - b \|_X + \|\mb u - \mb v\|) \lesssim  e^{-\omega \tau} (\delta \| \Phi - \Psi\|_{\mc X} +  \|\mb u - \mb v\|)
\end{align*}
by Lemma \ref{Le:Modulation}. The expressions $ \mb P_{a_\infty}  \mb K(\Phi, a,\mb u)$ and $ \mb H_{a_\infty}  \mb K(\Phi, a,\mb u)$ do not depend on $\mb u$ explicitly and analogously to Eq.~\eqref{Eq:Lipschitz_ProjK} we obtain 
\begin{align*}
\|\mb P_{a_\infty}  \mb K(\Phi, a,\mb u)(\tau) & - \mb P_{b_\infty}   \mb K(\Psi,b,\mb v)(\tau) \| \\
& + \|\mb H_{a_\infty}  \mb K(\Phi, a,\mb u)(\tau) -\mb H_{b_\infty}   \mb K(\Psi,b,\mb v)(\tau) \| \lesssim \delta e^{-2 \omega \tau}  \| \Phi - \Psi\|_{\mc X}.
\end{align*}
Finally, 
\begin{align*}
\|  \mb S_{a_{\infty}}(\tau)  \mb T_{a_{\infty}} \mb u & - \mb S_{b_{\infty}}(\tau)  
\mb T_{b_{\infty} }\mb v\| \lesssim e^{-\omega \tau} \| a- b \|_{X}  \| \mb u\|   +  e^{-\omega \tau} \| \mb u - \mb v \|   \\
& \lesssim e^{-\omega \tau} (\delta \| \Phi - \Psi\|_{\mc X} +  \|\mb u - \mb v\|)  
\end{align*}
by Theorem \ref{Th:Decay_linearized}. With this, the bound
\begin{align*}
\| \mb T_{a_{\infty}}   \mb K(\Phi, a,\mb u)(\tau) -\mb T_{b_{\infty} }  \mb K(\Psi,b,\mb v)(\tau) \|  \lesssim  e^{-\omega \tau} (\delta \| \Phi - \Psi\|_{\mc X} +  \|\mb u - \mb v\|)  
\end{align*}
can be proved analogous to Eq.~\eqref{Eq:Lipschitz_ProjK_stable}.
A combination of these estimates shows that 
\begin{align*}
\|\Phi - \Psi \|_{\mc X} = \|  \mb K(\Phi, a,\mb u) -  \mb K(\Psi,b,\mb v) \|_{\mc X}  \lesssim  \delta \| \Phi - \Psi\|_{\mc X} +  \| \mb u - \mb v \|
\end{align*}
and Eq.~\eqref{Eq:Lipschitz_Sol} follows for $\delta >0$ chosen sufficiently small. Hence, by Lemma \ref{Le:Modulation}, 
\[ \|a - b\|_X \lesssim \|\Phi - \Psi \|_{\mc X} +  \|\mb u - \mb v\| \lesssim  \|\mb u - \mb v\|, \]
which implies the claim.
\end{proof}
 
\subsection{ Co-dimension 9 stability in similarity variables}

In the following we construct a co-dimension 9 manifold determined by the vanishing of the correction term. Let 
\[B_{\delta/c}:= \{\mb u \in \mc H: \| u \| \leq \tfrac{\delta}{c}  \} \]
 and 
\[ \mb C(\mb u) := \mb C(\Phi_{\mb u}, \mb a_{\mb u}, \mb u ) \]
where $(\Phi_{\mb u},  a_{\mb u}) \in  \mc X_{\delta} \times X_{\delta}$ are associated to $\mb u \in B_{\delta/c}$ via 
Proposition \ref{Prop:NonlEvolution_Corr}. Furthermore, we define a projection operator combining all unstable directions 
\begin{align*}
\mb J_a := \mb P_a + \mb H_a =\sum_{k=0}^7 \mb P_a^{(k)} + \mb H_a.
\end{align*}
By definition, $\mb J_{a_{\infty}}   \mb C(\mb u) =  \mb C(\mb u)$
and we have the Lipschitz property
\begin{align*}
\|\mb J_a - \mb J_b\| \lesssim |a - b |
\end{align*}
for all $a,b \in X_{\delta}$ suitably small.

\begin{proposition}\label{Prop:Stable_manifold}
Let $\delta >0$ be sufficiently small and $\tilde c > c >0$ be sufficiently large. There exists a co-dimension $9$ Lipschitz manifold $\mc M \subset \mc H$ with $\mb 0 \in \mc M$ defined as the graph of a Lipschitz function $\mb M: \ker \mb J_0 \mapsto \rg \mb J_0$,

\begin{align*}
\mc M := \{\mb v + \mb M(\mb v)  \mid \mb v \in \ker \mb J_0 , \|\mb v\| \leq \tfrac{\delta}{\tilde c} \} \subseteq 
\{ \mb u \in B_{\delta/c} \mid \mb C(\mb u) = \mb 0 \} \subseteq \ker \mb J_0 \oplus \rg \mb J_0,
\end{align*}

such that for any $\mb u \in \mc M$ there exist unique $\Phi_{\mb u} \in \mc X_{\delta}$  and $a_{\mb u} \in X_{\delta}$ such that the equation
\begin{align}\label{Evol_Pert_Duhamel}
\Phi(\tau) = \mb S_{a_{\infty}}(\tau) \mb u + \int_0^{\tau} 
 \mb S_{a_{\infty}}(\tau - \sigma)[ \mb G_{a(\sigma)}(\Phi(\sigma)) - \partial_{\sigma} \Psi^*_{a(\sigma)} ] d \sigma
\end{align}
 is satisfied for all $\tau \geq 0$. Furthermore, there is a constant $K >\tilde c$ such that if $\mb u \in \mc H$, 
$ \| \mb u \|\leq \frac{\delta}{K} $ and $\mb C(\mb u) = 0$, then $\mb u \in \mc M$.
\end{proposition}

\begin{proof}
First, we convince ourselves that $\mb C(\mb u) = 0$ if and only if $\mb J_0 \mb C(\mb u) =0$. In fact, if $\mb J_0 \mb C(\mb u) =0$, then the bound
\begin{align*}
\| \mb C(\mb u)  \| \leq \|\mb J_0 \mb C(\mb u) + (\mb J_{a_{\mb u,\infty}} -\mb J_0 ) \mb C(\mb u) \| \lesssim   | a_{\mb u,\infty}|\|\mb C(\mb u) \|
\end{align*}
and the fact that $a_{\mb u,\infty} = \mc O(\delta)$ imply that $\mb C(\mb u)  = 0$. The other direction is obvious. \\

Next, we write $\mc H = \ker \mb J_0 \oplus \rg \mb J_0$, $\mb u = \mb v +\mb w \in \ker \mb J_0 \oplus \rg \mb J_0$ and require that $\mb C( \mb u) = 0$. For this, we fix $\mb v \in \ker \mb J_0$ and define 
\begin{align*}
\tilde {\mb  C}_{\mb v} : \rg \mb J_0 \to \rg \mb J_0, \quad \mb w \mapsto \tilde {\mb C}_{\mb v} :=  \mb J_0 \mb C(\mb v +\mb w ).
\end{align*}
We will show that this map is invertible at zero as long as $\mb v$ is sufficiently small to obtain $\mb w = \tilde {\mb  C}^{-1}_{\mb v}(\mb 0)$.  This determines the map
\[ \mb M:  \ker \mb J_0  \to \rg \mb J_0, \quad \mb v \mapsto \mb M(\mb v)  :=  \tilde {\mb  C}^{-1}_{\mb v}( \mb 0 ).\]
We prove this by a fixed point argument. Recall that  $\mb C = \mb C_1 + \mb C_2$, $\mb C_1 = \sum_{k=0}^{7} \mb C_1^{(k)}$ with
\begin{align}\label{Eq:Corr1}
 \mb C_1^{(k)}(\Phi,a,\mb u) &   = \mb P_{a_{\infty}}^{(k)}\mb u +\mb P_{a_{\infty}}^{(k)}  \mb I_1(\Phi,a),
 \end{align}
 and
 \begin{align}\label{Eq:Corr2}
\mb C_2(\Phi,a,\mb u)   = \mb H_{a_{\infty}} \mb u   + \mb H_{a_{\infty}} \mb I_2(\Phi,a),
\end{align}
where
\begin{align}\label{Eq:Corr3}
\begin{split}
\mb I_1(\Phi,a) & :=  \int_0^{\infty} e^{-\sigma} [\mb G_{a(\sigma)}(\Phi(\sigma)) - \partial_{\sigma} \Psi^*_{a(\sigma)} ] d \sigma, \\
\mb I_2(\Phi,a) &  : = \int_0^{\infty} e^{-3 \sigma} [\mb G_{a(\sigma)}(\Phi(\sigma)) - \partial_{\sigma} \Psi^*_{a(\sigma)} ] d \sigma.
\end{split}
\end{align}
To abbreviate the notation, we define $\mb F_1(\mb u)  := \sum_{k=0}^{7}  \mb F_1^{(k)}(\mb u)$, where
\begin{align*}
 \mb F_1^{(k)}(\mb u) := \mb P_{a_{\infty}}^{(k)}  \mb I_1(\Phi_{\mb u},a_{\mb u})
\end{align*}
as well as 
\[ \mb F_2(\mb u) := \mb H_{a_{\infty}} \mb I_2(\Phi_{\mb u},a_{\mb u}). \]
In view of Lemma \ref{Le:Bound_Ga} and Eq.~\eqref{Bounds:Proj_dtauPsi^*_1} we have 
\begin{align}\label{Eq:Integral_terms_delta}
\| \mb F_1^{(k)}(\mb u)\| \lesssim \delta^2, \quad  \|\mb F_2(\mb u)\| \lesssim \delta^2.   
\end{align}
In order to rewrite the equation we note that since $\mb v \in \ker \mb J_0$, 
\begin{align*}
\tilde {\mb  C}_{\mb v}  & = \mb J_0 \mb C(\mb v +\mb w )  = \mb J_0 \mb J_{a_{\infty}} (\mb v +\mb w ) + \mb J_0 ( \mb F_1(\mb v +\mb w) +\mb F_2(\mb v +\mb w)) \\
& = \mb J_0^2 \mb w + \mb J_0 (\mb J_{a_{\infty}} - \mb J_0) \mb w + \mb J_0 \mb J_{a_{\infty}}  \mb v+ \mb J_0 ( \mb F_1(\mb v +\mb w) +\mb F_2(\mb v +\mb w)) \\
& = \mb w +  \mb J_0 (\mb J_{a_{\infty}} - \mb J_0)( \mb v+ \mb w) + \mb J_0 ( \mb F_1(\mb v +\mb w) +\mb F_2(\mb v +\mb w)) 
\end{align*}
By setting
\begin{align*}
\Omega_{\mb v}(\mb w) :=  \mb J_0 (\mb J_0 -\mb J_{a_{\infty}})( \mb v+ \mb w) - \mb J_0 ( \mb F_1(\mb v +\mb w) +\mb F_2(\mb v +\mb w)) ,
\end{align*}
the equation $\tilde {\mb  C}_{\mb v}  = 0$ is equivalent to 
\begin{align}\label{Eq:FP_Omega} 
\mb w = \Omega_{\mb v}(\mb w).
\end{align}
We define
\begin{align*}
\tilde B_{\delta/c}(\mb v) := \{ \mb w \in \rg \mb J_0 \mid \|\mb v+ \mb w\| \leq \tfrac{\delta}{c} \}
\end{align*}
and prove that $\Omega_{\mb v}: \tilde B_{\delta/c}(\mb v) \to \tilde B_{\delta/c}(\mb v)$ is a contraction for small enough $\mb v$. In fact, let $\mb v \in \mc H$ satisfy 
 $\|\mb v\| \leq \frac{\delta}{2c}$ and let $\mb v \in \tilde B_{\delta/c}(\mb v)$. Using the bound given in Eq.~\eqref{Eq:Integral_terms_delta} we obtain  
 \begin{align*}
\| \Omega_{\mb v}(\mb w) \| \leq \| \mb J_0 -\mb J_{a_{\infty}}\| \| \mb v+ \mb w \| + \|\mb F_1(\mb v +\mb w)\| +\|\mb F_2(\mb v +\mb w)\| 
\lesssim \tfrac{\delta^2}{c}  + \delta^2.
 \end{align*}
 Thus, $\| \mb v +  \Omega_{\mb v}(\mb w)  \| \leq \frac{\delta}{c}$ for $\delta$ sufficiently small.  Now for $\mb w, \tilde {\mb w} \in \tilde B_{\delta/c}(\mb v)$ we associate to $\mb v + \mb w$ and $\mb v + \tilde {\mb w}$ by Proposition \ref{Prop:NonlEvolution_Corr} functions $(\Phi,a)$ and $(\Psi,b) \in  \mc X_{\delta} \times X_{\delta}$ and estimate
\begin{align*}
\| \Omega_{\mb v}(\mb w) -  \Omega_{\mb v}(\tilde {\mb w}) \| \leq & \|\mb J_0 (\mb J_0 -\mb J_{a_{\infty}})( \mb v+ \mb w) - \mb J_0 (\mb J_0 -\mb J_{b_{\infty}})( \mb v+ \tilde {\mb w}) \|\\
& + \|\mb F_1(\mb v +\mb w) -\mb F_1(\mb v +\tilde {\mb w})\| + \|\mb F_2(\mb v +\mb w) -\mb F_2(\mb v +\tilde {\mb w})\|.
\end{align*}
We write 
\begin{align*}
\mb J_0 &  (\mb J_0 -\mb J_{a_{\infty}})( \mb v+ \mb w) - \mb J_0 (\mb J_0 -\mb J_{b_{\infty}})( \mb v+ \tilde {\mb w}) \\
& = \mb J_0  (\mb J_0 -\mb J_{a_{\infty}})( \mb w- \tilde {\mb w}) - \mb J_0 (\mb J_{a_{\infty}} -\mb J_{b_{\infty}}) ( \mb v+ \tilde {\mb w})
\end{align*}
and estimate 
\begin{align*}
\|\mb J_0   (\mb J_0 -\mb J_{a_{\infty}})( \mb w- \tilde {\mb w})\| & + \| \mb J_0 (\mb J_{a_{\infty}} -\mb J_{b_{\infty}}) ( \mb v+ \tilde {\mb w}) \| \lesssim |a_{\infty}| \| \mb w-\tilde {\mb w}\| + |a_{\infty} - b_{\infty}| \|\mb v+ \tilde {\mb w}\|  \\
& \lesssim \delta \| \mb w- \tilde {\mb w}\|  + \tfrac{\delta}{c} \|a -b \|_X  \lesssim \delta  \| \mb w- \tilde {\mb w}\| 
\end{align*}
by Proposition \ref{Prop:NonlEvolution_Corr}.

Note that by Lemma \ref{Le:Bound_Ga} and Eq.~\eqref{Bounds:Proj_dtauPsi^*_3}, for $k = 0,\dots, 7$, 
\begin{align*}
\|\mb P_{a_{\infty}}^{(k)}  \mb I_1(\Phi,a) - \mb P_{b_{\infty}}^{(k)}  \mb I_1(\Psi,b )   & \| \lesssim \delta ( \|a-b\|_X + \|\Phi - \Psi\|_{\mc X} ) 
\end{align*}
and 
\begin{align*}
  \|\mb H_{a_{\infty}}  \mb I_2(\Phi,a) - \mb H_{b_{\infty}}  \mb I_2(\Psi,b)  \|  &   \lesssim \delta ( \|a-b\|_X + \|\Phi - \Psi\|_{\mc X} ). 
\end{align*}
Hence,
\begin{align*}
 \|\mb F_1(\mb v +\mb w)  &  -\mb F_1(\mb v +\tilde {\mb w})\| + \|\mb F_2(\mb v +\mb w) -\mb F_2(\mb v +\tilde {\mb w})\| \lesssim \delta \| \mb w- \tilde {\mb w}\|,
\end{align*} 
which implies the claimed Lipschitz continuity. By the Banach fixed point theorem we infer that for every $\mb v \in \ker \mb J_0$ with $\| \mb v \| \leq \frac{\delta}{2c}$ there exists a unique $\mb w \in \tilde B_{\delta/c}(\mb v)$ that solves Eq.~\eqref{Eq:FP_Omega} and thus $\mb C ( \mb v+ \mb w ) =\tilde {\mb  C}_{\mb v}(\mb w) = 0$. 
Finally, we prove that the map $\mb v \mapsto \mb M(\mb v)$ is Lipschitz continuous. Let $\mb v, \tilde {\mb v}  \in \ker \mb J_0$, $\|\mb v\|, \|\tilde {\mb v}\| \leq \frac{\delta}{2c}$ and $\mb w \in \tilde B_{\delta/c}(\mb v)$, $\tilde {\mb w}  \in \tilde B_{\delta/c}(\tilde {\mb v})$ the solutions to Eq.~\eqref{Eq:FP_Omega}. Then,
\begin{align*}
\| \mb M(\mb v)  - \mb M(\tilde {\mb v}) \| & = \|\mb w - \tilde {\mb w} \| \leq  
\| \Omega_{\mb v}(\mb w)  - \Omega_{\mb v}(\tilde {\mb w})  \| + \|\Omega_{\mb v}(\tilde {\mb w}) - \Omega_{\tilde {\mb v}}(\tilde {\mb w}) \|  \\
& \leq \tfrac{1}{2}  \|\mb w - \tilde {\mb w} \| + \|\Omega_{\mb v}(\tilde {\mb w}) - \Omega_{\tilde {\mb v}}(\tilde {\mb w})  \|.
\end{align*}
With
\begin{align*}
\|\Omega_{\mb v}(\tilde {\mb w}) &  - \Omega_{\tilde {\mb v}}(\tilde {\mb w})\|   =
 \|\mb J_0 (\mb J_0 -\mb J_{a_{\mb v+\tilde {\mb w},\infty}})( \mb v+ \tilde {\mb w}) - \mb J_0 ( \mb F_1(\mb v +\tilde {\mb w}) +\mb F_2(\mb v +\tilde {\mb w})) \\
 & -   \mb J_0 (\mb J_0 -\mb J_{a_{\tilde {\mb v}+\tilde {\mb w},\infty}})( \tilde {\mb v}+ \tilde {\mb w}) - \mb J_0 ( \mb F_1(\tilde {\mb v}+\tilde {\mb w}) +\mb F_2(\tilde {\mb v} +\tilde {\mb w})) \| \\
 & \lesssim \| \mb J_0 (  \mb J_{a_{\tilde {\mb v}+\tilde {\mb w},\infty}} - \mb J_{a_{\mb v+\tilde {\mb w},\infty}}) \tilde {\mb w} \|
 + \| \mb J_0( \mb J_{a_{\mb v+\tilde {\mb w},\infty}} \mb v - \mb J_{a_{\tilde {\mb v}+\tilde {\mb w},\infty}}  \tilde {\mb v} )  \|
 +\| \mb J_0 ( \mb F_1(\tilde {\mb v}+\tilde {\mb w}) +\mb F_2(\tilde {\mb v} +\tilde {\mb w}))  \| \\
 & \lesssim |a_{\tilde {\mb v}+\tilde {\mb w},\infty} - a_{\mb v+\tilde {\mb w},\infty}| \|\tilde {\mb w} \|  
+ \|\tilde {\mb v}  - \mb v \|   +  |a_{\tilde {\mb v}+\tilde {\mb w},\infty} - a_{\mb v+\tilde {\mb w},\infty}| \|\tilde {\mb v} \| +   \delta   \| \tilde {\mb v}  - \mb v  \| \\
 & \lesssim \tfrac{\delta}{c} \|\tilde {\mb v}  - \mb v \| + \|\tilde {\mb v}  - \mb v \| + \tfrac{\delta}{2 c}  \| \tilde {\mb v}  - \mb v  \| + \delta  \| \tilde {\mb v}  - \mb v  \|  \lesssim \|\tilde {\mb v}  - \mb v \|.
\end{align*}
This implies the claimed Lipschitz bound
\begin{align*}
\| \mb M(\mb v)  - \mb M(\tilde {\mb v}) \| \leq 2 \|\Omega_{\mb v}(\tilde {\mb w}) - \Omega_{\tilde {\mb v}}(\tilde {\mb w})  \| \lesssim \|\tilde {\mb v}  - \mb v \|.
\end{align*}

We note that for $\mb u =\mb 0$, the associated $(\Phi, a)$ are trivial, $\Phi = \mb 0$ and $a = 0$. Hence, $\mb C( \mb 0) =\mb F_1( \mb 0) + \mb F_2( \mb 0) = 0$.
On the other hand  $\mb u = \mb v + \mb w = \mb 0$ if and only if $\mb  v = \mb w = \mb 0$. However, since in that case $\mb v$ obviously satisfies the smallness condition, the associated $ \mb w$ solving $\mb C(\mb 0 +\mb w ) = 0$ is unique, hence $\mb M(\mb 0) = \mb 0$ which shows that $\mb 0 \in \mc M$. In particular, we have 
\begin{align*}
\| \mb v + \mb M(\mb v) \| \leq \| \mb v \| + \| \mb M(\mb v) - \mb M(\mb 0)  \| \lesssim \| \mb v \| 
\end{align*}
and thus, $\| \mb v + \mb M(\mb v) \| \leq \frac{\delta}{c}$ whenever $\| \mb v \| \leq \frac{\delta}{\tilde c}$ for sufficiently large $\tilde c > 2c$.\\

We prove the final statement. Let $\mb u \in \mc H$ satisfy $\mb C(\mb u) = \mb 0$. Then
\[ \|(1-\mb J_0) \mb u\| \lesssim \| \mb u \| \]
and we have $\mb v_{\mb u} := (1-\mb J_0)  \mb u \in \ker \mb J_0$ and $\|\mb v_{\mb u}\|  \leq \frac{\delta}{\tilde c}$ as long as $\|\mb u \| \leq \frac{\delta}{K}$ for sufficiently large $K > \tilde c$. Again, by uniqueness, $\mb w_{\mb u} := \mb J_0  \mb u = \mb M(\mb v_{\mb u})$ and thus $\mb u \in \mc M$. 
\end{proof}

Due to the structure of the correction term $\mb C = \mb C_1 + \mb C_2$ we can repeat the same arguments to derive the following result.

\begin{lemma}
Let $\delta >0$ be sufficiently small and $c, \tilde c_1, \tilde c_2 >0$ be sufficiently large. There exist manifolds $\mc M_1, \mc M_2 \subset \mc H$ with $\mb 0 \in \mc M_1, \mc M_2$ defined as the graph of Lipschitz functions, 
\[ \mb M_1: \ker \mb P_0 \mapsto \rg \mb P_0  \quad \mb M_2: \ker \mb H_0 \mapsto \rg \mb H_0  ,\]
\begin{align*}
\mc M_1 := \{\mb v + \mb M_1(\mb v)  \mid \mb v \in \mathrm{ker} \mb P_0 , \|\mb v\| \leq \tfrac{\delta}{\tilde c_1} \} \subseteq 
\{ \mb u \in B_{\delta/c} \mid \mb C_1(\mb u) = \mb 0 \} \subseteq \ker \mb P_0 \oplus \rg \mb P_0,
\end{align*}
and
\begin{align*}
\mc M_2 := \{\mb v + \mb M_2(\mb v)  \mid \mb v \in \mathrm{ker} \mb H_0 , \|\mb v\| \leq \tfrac{\delta}{\tilde c_2} \} \subseteq 
\{ \mb u \in B_{\delta/c} \mid \mb C_2(\mb u) = \mb 0 \} \subseteq \ker \mb H_0 \oplus \rg \mb H_0.
\end{align*}
Furthermore, there are constants $K_j>\tilde c_j$, $j =1,2$ such that if $\mb u \in \mc H$, 
$ \| \mb u \|\leq \frac{\delta}{K_j} $ and $\mb C_j(\mb u) = 0$.
\end{lemma}  

From the above results and the fact that  $\mb C = 0$ if and only if both $\mb C_1 =\mb 0$ and $C_2 = \mb 0$ we immediately obtain the following characterization. 

\begin{corollary}
Let $\delta >0$ be sufficiently small. There is a $\tilde K >0$ such that  
\[\mc M' := \{ \mb u \in \mc M : \|\mb u \| \leq \tfrac{\delta}{\tilde K} \},\]
 a small ball within $\mc M$, is contained in $\mc M_1 \cap \mc M_2$. 
\end{corollary}

The the statement of Proposition \ref{Prop:Main} is now a direct consequence of Proposition \ref{Prop:Stable_manifold}.

\subsubsection{Proof of Proposition \ref{Prop:Main}}

\begin{proof}
Let $\delta >0$ be sufficiently small and $\tilde c > c >0$ be sufficiently large such that the manifold $\mc M$ is defined according to  Proposition \ref{Prop:Stable_manifold}. Let $\Phi_0 \in \mc M \cap \mc D(\mb L)$. In particular, $\| \Phi_0 \| \leq \frac{\delta}{c}$  and $\mb C(\Phi_0) = 0$. According to Proposition \ref{Prop:Stable_manifold} there exist unique $\Phi \in \mc X_{\delta}$ and $a \in X_{\delta}$ such that Eq.~\ref{Evol_Pert_Duhamel} with $\mb u = \Phi_0$ is satisfied. It is easy to check that the function $ \mb {\tilde G}(\sigma, \Phi)  = \mb G_{a(\sigma)}(\Phi) - \partial_{\sigma} \Psi^*_{a(\sigma)} $ is continuously differentiable with respect to both variables ($\Psi^*_a$ depends smoothly on $a \in X$, $\mb F_a(\mb u)$ is algebraic in $\mb u$ and the remaining parts are bounded linear operators). Furthermore, since  $\Phi_0 \in  \mc D(\mb L)$ we can apply standard semigroup theory,  cf. for example Theorem 6.1.5 in \cite{pazy}, to conclude that $\Phi$ solves the initial value problem
\begin{align}
\begin{split}
\partial_{\tau} \Phi(\tau)   & =  [\tilde{\mb L}   + \mb L'_{a(\tau)} ] \Phi(\tau)  + \mb F_{a(\tau)}(\Phi(\tau))  - \partial_{\tau} \Psi^*_{a(\tau)} , \quad \tau > 0,    \\
  \Phi(0) & = \Phi_0.
\end{split}
\end{align}
By defining $\Psi(\tau) := \Psi^*_{a(\tau)} + \Phi(\tau)$
we obtain a solution $\Psi \in C^1([0,\infty), \mc H)$ to 
\begin{align}\label{Eq:NonlinearWave_Sim}
\begin{split}
\partial_{\tau} \Psi(\tau) & =   \mb L \Psi(\tau) + \mb N(\Psi(\tau)), \quad \tau > 0, \\
\Psi(0) & =  \Psi^*_0 + \Phi_0.
\end{split}
\end{align}
We set $\tilde \Phi(\tau) :=  \Phi(\tau) + \Psi^*_{a(\tau)} - \Psi^*_{a_{\infty}}$
and write 
\[ \Psi(\tau) =  \Psi^*_{a_{\infty}} + \tilde \Phi(\tau). \]
Using Eqns.~\eqref{Eq:Selfsim_Sol_Lipschitz} and \eqref{Eq:LorentzPar1} we infer that 
\[ \|  \tilde \Phi(\tau)\| \leq \|\Phi(\tau) \| + \| \Psi^*_{a(\tau)} - \Psi^*_{a_{\infty}} \| \lesssim e^{- \omega \tau} \]
as claimed. 
\end{proof}

This proof completes the discussion of the problem in self-similar variables. Now we connect this to the original equation in physical variables, i.e., we consider Eq.~\eqref{Eq:Mod_Integral} with $\mb u = \mb U(\mb v, T, x_0)$ for suitable $\mb v$, see Eq.~\eqref{Eq:InitialData_Op}.

\subsection{Variation of blowup parameters}\label{Sec:Codim1}

\subsubsection{Initial data operator}

First, we properly define the initial data operator $\mb U$, see Eq.~\eqref{Eq:InitialData_Op},  which can be written as 
 \begin{align*}
\mb U(\mb v, T, x_0) = \mc R(\mb v, T, x_0) + \mc R(\Psi^*_0, T, x_0) -  \mc  R(\Psi^*_0, 1, 0),
\end{align*}
for $\mb v = (v_1,v_2)$. Since we expect $(T,x_0)$ to be close to $(1,0)$ we a priori assume that $x_0 \in \B^7_{1/2}$ and that $T \in I := [\frac12, \frac32]$. Furthermore, we define the Hilbert space 
\[  \mc Y := H^4 \times H^3(\mathbb B^7_2)\] 
and denote by $ \mc B_{\mc Y}$ the unit ball in $\mc Y$.

\begin{lemma}\label{Le:InitialData_Operator}
Let $\delta >0$  be sufficiently small. The initial data operator $\mb U: \mc B_{\mc Y}  \times I \times \mathbb B^7_{1/2}  \to \mc H$ is Lipschtiz continuous, i.e.,  
\begin{align*}
\|\mb U(\mb v, T_1, x_0) - \mb U(\mb w, T_2,  y_0) \| \lesssim \| \mb v - \mb w\|_{\mc Y} + |T_1- T_2| - |x_0 - y_0|
\end{align*}
for all $\mb v, \mb w \in \mc B_{\mc Y}$, all $T_1,T_2  \in   I$ and all $x_0, y_0 \in \mathbb  B^7_{1/2} $. Furthermore, if $\|\mb v \|_{\mc Y} \leq \delta$, then 
for all $T \in  [1-\delta, 1+ \delta]  \subset I$ and all $x_0 \in  \B^7_{\delta}$,
\begin{align*}
\|\mb U(\mb v, T, x_0) \|  \lesssim \delta.
\end{align*}
\end{lemma}

\begin{proof}
	Let $v \in C^{\infty}(\B^7_2)$. For all $T \in I$ and all $x_0, y_0 \in \B^7_{1/2}$, we have 
	\begin{align*}
	v(T\xi + x_0) - v(T\xi + y_0) = (x_0^j - y_0^j) \int_0^1 \partial_j v(T\xi + y_0 +s(x_0 - y_0)) ds,
	\end{align*}
	which implies that 
	$ \| v(T\cdot + x_0) - v(T\cdot + y_0) \|_{L^2(\B^7)} \lesssim \|v\|_{H^1(\B_{2}^7)} |x_0 - y_0|.$
	The same argument shows that for fixed $k \in \N$,  
	\begin{equation}\label{Eq:Shift_Lipschitz}
	\| v(T \cdot + x_0) - v(T \cdot + y_0) \|_{H^k(\B^7)} \lesssim \|v\|_{H^{k+1}(\B_{2}^7)} |x_0 - y_0|.
	\end{equation}
	Similarly, for all $T_1, T_2 \in I$ and all $x_0 \in  \B^7_{1/2}$
	\[ v(T_1 \xi +x_0) - v(T_2 \xi +x_0) = (T_1 - T_2)\int_0^1 \xi^{j} \partial_j v((T_2 + s (T_1 - T_2))\xi +x_0) ds, \]
	and thus
	\[ \| v(T_1 \cdot+x_0) - v(T_2 \cdot +x_0) \|_{L^2(\B^7)} \lesssim \|(\cdot) \nabla v \|_{L^2(\B^7_2)}   |T_1 - T_2| \lesssim\|v \|_{H^1(\B^7_{2})}  |T_1 - T_2|  , \]
	where we have used that $| \xi^j \partial_j v(\xi)|^2 \leq |\xi|^2 \partial_j v(\xi) \overline{\partial^j v(\xi)}$. Analogously, for fixed $k \in \N$, we have
	\begin{equation}\label{Eq:Scaling_Lipschitz}
	\| v(T_1 \cdot+x_0) - v(T_2 \cdot+x_0) \|_{H^k(\B^7)} \lesssim |T_1 - T_2|  \|v \|_{H^{k+1}(\B^7_{2})}. 
	\end{equation}
	By a density argument Eqns.~\eqref{Eq:Shift_Lipschitz} - \eqref{Eq:Scaling_Lipschitz} can be extended to all $v \in H^{k+1}(\B^7_{2})$. 
	Now let $\mb v \in \mc Y$. Then, for all $T_1,T_2 \in I$, all $x_0, y_0 \in \B^7_{1/2}$, and all $\mb v,\mb w \in \mc Y$,  Eqns.~\eqref{Eq:Shift_Lipschitz} - \eqref{Eq:Scaling_Lipschitz} imply that 
	\begin{align}\label{Eq:Lipschitz_R}
	\| \mc R(\mb v,T_1,x_0) - \mc R(\mb w,T_2,y_0) \| \lesssim \| \mb v \|_{\mc Y} ( | T_1 - T_2 | + |x_0 - y_0| ) + \|\mb v - \mb w \|_{\mc Y}.
	\end{align}
	
	Since $\Psi^*_0$ is defined and smooth on all of $\R^7$ we immediately obtain 
	\begin{align}\label{Lipschitz_RescalingPsi}
	\| \mc R(\Psi^*_0, T_1, x_0) -  \mc  R(\Psi^*_0, T_2, y_0)\| \lesssim|T_1-T_2| + |x_0 - y_0| 
	\end{align}
	for all $T_1, T_2 \in I$ and all $x_0, y_0 \in \B^7_{1/2}$.
	Finally, the bound
	\[ \| \mb U(\mb v, T, x_0)  \lesssim \|\mb v\|_{\mc Y}  + |T-1|  + |x_0|\]
	implies the claimed smallness of $\mb U$. 
\end{proof}

This yields the following corollary to Proposition \ref{Prop:NonlEvolution_Corr}.
\begin{corollary}\label{Cor:NonlEvolution_Corr}
Let $M >0$ be sufficiently large and $\delta >0$  be sufficiently small. Then, for every $\|\mb v \|_{\mc Y} \leq \frac{\delta}{M}$, every $T \in  [1-\frac{\delta}{M}, 1+ \frac{\delta}{M}]  \subset I$ and every $x_0 \in \overline{\mathbb B^7_{\delta/M}}$ there exist functions $\Phi \in \mc X_{\delta}$ and $a \in X_{\delta}$ such that the integral equation
\[ \Phi(\tau)  = \mb S_{a_{\infty}}(\tau) [ \mb U(\mb v, T, x_0)- \mb C(\Phi, a, \mb U(\mb v, T, x_0))] + \int_0^{\tau} 
 \mb S_{a_{\infty}}(\tau - \sigma)[ \mb G_{a(\sigma)}(\Phi(\sigma)) - \partial_{\sigma} \Psi^*_{a(\sigma)} ] d \sigma \]
is satisfied  for all $\tau \geq 0$. In particular, $\Phi$ decays exponentially, 
\[ \|\Phi(\tau) \| \lesssim \delta e^{-\omega \tau}, \quad \forall \tau \geq 0.\] 
Furthermore, the solution map $(\mb v, T, x_0)  \mapsto (\Phi, a)$ satisfies
\begin{align*}
\begin{split}
 \|\Phi(\mb v,T,x_0) - \Phi(\mb w, \tilde T, y_0)\|_{\mc X} & + \|a(\mb v,T,x_0) - a(\mb w, \tilde T, y_0)\|_X \\
 &  \lesssim \|\mb v - \mb w\|_{\mc Y} + |T-\tilde T| + |x_0 - \tilde x_0| 
 \end{split}
 \end{align*}
 for all $\mb v, \mb w \in \mc Y$ satisfying the smallness condition, all $T, \tilde T \in [1-\frac{\delta}{M}, 1+ \frac{\delta}{M}] $ and all $x_0, \tilde x_0  \in\overline{\mathbb B^7_{1/2}}$
\end{corollary}

\subsubsection{Variation of blowup parameters - Proof of Proposition  \ref{Prop:Main_Manifold}} 

To shorten the notation, we set $\mb h := \mb h_0$, where $\mb h_0$ is eigenfunction of $\mb L_0$ corresponding to the eigenvalue $\lambda = 3$, see Proposition \ref{Prop:Spectrum_L0}.
Proposition \ref{Prop:Main_Manifold} follows directly from the next result. 

\begin{proposition}\label{Prop:VanishingCorrection}
Let $c > 0$ be sufficiently large and $\delta >0$ be sufficiently small. For every $\mb v \in \mc Y$ satisfying $\|\mb v \|_{\mc Y} \leq \frac{\delta}{c^2}$, there exist functions $\Phi \in \mc X_{\delta}$, $a \in X_{\delta}$ and parameters 
$\alpha \in [-\frac{\delta}{c},\frac{\delta}{c}]$, $T \in I_{\delta/c} = [1- \frac{\delta}{c}, 1+\frac{\delta}{c}] \subset I$, $x_0 \in \overline{\mathbb B^7_{\delta/c} }\subset  \mathbb B^7_{1/2} $ such that 
\begin{align}\label{Eq:Vanishing_Correction}
\mb C(\Phi, a, \mb U(\mb v + \alpha  \mb h, T, x_0)) = 0.
\end{align}
Moreover, the parameters depend Lipschitz continuously on the data, i.e., 
\begin{align*}
|\alpha(\mb v) -\alpha(\mb w)| +  | T(\mb v) - T(\mb w)| + |x_0(\mb v) - x_0(\mb w)|  \lesssim \| \mb v - \mb w\|_{\mc Y}
\end{align*}
for all $\mb v, \mb w \in \mc Y$ satisfying the above smallness assumption. In particular, $\mb U(\mb v + \alpha  \mb h, T, x_0) \in \mc M$.
\end{proposition}

\begin{proof}
First, we observe that 
$\|\mb v + \alpha  \mb h \| \leq \frac{\delta}{M}$
for $c >M$ sufficiently large, where $M >0$ is the constant in Corollary \ref{Cor:NonlEvolution_Corr}. Consequently, for fixed $\mb v$ and all $\alpha \in  [-\frac{\delta}{c},\frac{\delta}{c}]$, $T \in  I_{\delta/c}$ and $x_0 \in \overline{\mathbb B^7_{\delta/c} }$ there are functions $(\Phi, a) \in \mc X_{\delta} \times X_{\delta}$,  $\Phi = \Phi(\mb v + \alpha \mb h,T, x_0)$, $a = a(\mb v + \alpha \mb h,T, x_0)$  associated  via Corollary \ref{Cor:NonlEvolution_Corr}. Furthermore, by choosing $c$ possibly larger, we can always guarantee that $\| \mb U(\mb v + \alpha  \mb h, T, x_0)\| \leq \frac{\delta}{K}$, where $K$ is the constant in Proposition \ref{Prop:Stable_manifold}. 
For the correction term, recall Eq.~\eqref{Eq:Corr1} - \eqref{Eq:Corr3}.
We show that $(\alpha,T, x_0)$ can be chosen such that for $k = 0,\dots,7$,
\begin{align}\label{Eq:Correction_FIniteDim}
\begin{split}
\left ( \mb C_1^{(k)}(\Phi, a, \mb U(\mb v + \alpha \mb h, T, x_0)) \big | \mb g_{a_{\infty}}^{(k)} \right ) & = 0, \\
\left ( \mb C_2(\Phi, a, \mb U(\mb v + \alpha \mb h, T, x_0)) \big | \mb h_{a_{\infty}} \right ) & = 0, \\
\end{split}
\end{align}
which implies Eq.~\eqref{Eq:Vanishing_Correction}.
For this, we use the fact that 
\begin{align*}
\partial_T [ T \psi^*_{0,1}(T\xi + x_0)]\big|_{(T,x_0)=(1,0)} = c_0 g_{0,1}(\xi), \quad \partial_{x_0^j} [ T \psi^*_{0,1}(T\xi + x_0)]\big|_{(T,x_0)=(1,0)} = c_{j} g^{(j)}_{0,1}(\xi)
\end{align*}
for $j=1,\dots,7$ and some constants $c_0, c_{j} \in \R$. The analogous statement holds for the second component $\psi^*_{0,2}$. By Taylor expansion we get that 
for all $T \in I_{\delta/c}$ and $x_0 \in \overline{\mathbb B^7_{\delta/c} }$,
\begin{align*}
\mc R(\Psi^*_0, T, x_0) -  \mc  R(\Psi^*_0, 1, 0)&  =  c_0 (T-1) \mb g_{0}^{(0)}  + \sum_{j=1}^{7}  c_j x_0^j  \mb g^{(j)}_{0} +
R( T, x_0). 
\end{align*}
For all $T, \tilde T \in I_{\delta/c}$, and all $x_0, \tilde x_0 \in  \overline{\mathbb B^7_{\delta/c} }$, we have 
\begin{equation*}
\|R( T, x_0) - R( \tilde  T, \tilde x_0) \| \lesssim \delta ( |T-\tilde T| + |x_0 - \tilde x_0| ).
\end{equation*}
In the following, we write
\begin{align*}
 \mb U(\mb v + \alpha \mb h, T, x_0)  &  =   \mc R(\mb v + \alpha \mb h, T, x_0) + 
  c_0 (T-1) \mb g_{a_{\infty}}^{(0)} + \sum_{j=1}^{7}  c_j x_0^j  \mb g^{(j)}_{a_\infty}  \\
  & + c_0 (T-1)  [\mb g_{0}^{(0)} -\mb g_{a_\infty}^{(0)}] +\sum_{j=1}^{7}  c_j x_0^j  [\mb g^{(j)}_{0}  - \mb g^{(j)}_{a_\infty} ] + R(T, x_0)\\
  &  =  \mc R(\mb v + \alpha \mb h, T, x_0) + 
  c_0 (T-1) \mb g_{a_{\infty}}^{(0)} + \sum_{j=1}^{7}  c_j x_0^j  \mb g^{(j)}_{a_\infty}  + R_{a_{\infty}}(T, x_0),
\end{align*}
with
\[ R_{a_{\infty}}(T, x_0) = c_0 (T-1)  [\mb g_{0}^{(0)} -\mb g_{a_\infty}^{(0)}] +\sum_{j=1}^{7}  c_j x_0^j  [\mb g^{(j)}_{0}  - \mb g^{(j)}_{a_\infty} ] + R(T, x_0).\]
It is easy to check that the remainder term satisfies
\begin{equation}\label{Eq:Remainder_1}
\|R_{a}(T, x_0) - R_{b}(\tilde T, \tilde x_0)  \| \lesssim \delta  (|a-b | + |T - \tilde T| + |x_0 - \tilde x_0|)
\end{equation}
for all $a,b \in \B^7$ with $|a| \lesssim \delta$, $|b| \lesssim \delta$,  all $T, \tilde T \in I_{\delta/c}$, and all $x_0, \tilde x_0 \in  \overline{\mathbb B^7_{\delta/c} }$.
To further simplify the above expression, we write
\[   \mc R(\mb v + \alpha \mb h ,T , x_0) = 
 \mc R(\mb v, T, x_0) + 
\alpha  \mc R(\mb h_{a_{\infty}}, T, x_0)   + \alpha \mc R(\mb h - \mb h_{a_{\infty}} , T, x_0).\]
The last term can be estimated by
\begin{align*}
\| \mc R(\mb h - \mb h_{a_{\infty}} , T, x_0) \| \lesssim \|\mb h - \mb h_{a_{\infty}} \|_{H^3 \times H^2(\mathbb B^7_{2})} \lesssim |a_{\infty}| ,
\end{align*}
where the Lipschitz bound for $\mb h_a$ follows from its smoothness on $\overline{\B^7_2}$ for small $a$. 
By Taylor expansion of  $\mc R(\mb h_{a_{\infty}}, T, x_0) $ at $(T,x_0) = (1,0)$ we obtain 
\begin{align*}
 \mc R(\mb v + \alpha \mb h, T, x_0) =   \mc R(\mb v ,T , x_0) +  \alpha  \mb h_{a_{\infty}} +  \alpha  \tilde R_{a_\infty}(T,x_0) 
\end{align*}
where 
\begin{equation}\label{Eq:Remainder_2}
 \|  \tilde R_{a}(T,x_0) - \tilde R_{b}(\tilde T,\tilde x_0) \| \lesssim |T-\tilde T| + |x_0 -\tilde x_0| +|a-b|.
 \end{equation}
In summary, we have 
\begin{align*}
 \mb U(\mb v + \alpha \mb h, T, x_0)=  \mc R(\mb v ,T , x_0)  &+  \alpha  \mb h_{a_{\infty}}  + c_0 (T-1) \mb g_{a_{\infty}}^{(0)} + \sum_{j=1}^{7}  c_j x_0^j  \mb g^{(j)}_{a_\infty}  \\
 &  + R_{a_\infty}(T, x_0) +  \alpha  \tilde R_{a_\infty}(T,x_0). 
\end{align*}
By transversality of the projections, we infer that for $j = 1, \dots, 7$, 
\begin{align*}
\mb P_{a_\infty}^{(0)} \mb U(\mb v + \alpha \mb h, T, x_0) & = 
\mb P_{a_\infty}^{(0)} \mc R(\mb v, T, x_0) + c_0 (T-1) \mb g_{a_{\infty}}^{(0)} + \mb P_{a_\infty}^{(0)} R_{a_\infty}(T, x_0)  + \alpha  \mb P_{a_\infty}^{(0)}  \tilde R_{a_\infty}(T,x_0),
\\
\mb P_{a_\infty}^{(j)} \mb U(\mb v + \alpha \mb h, T, x_0) & = 
\mb P_{a_\infty}^{(j)} \mc R(\mb v, T, x_0) + c_j x_0^j  \mb g^{(j)}_{{a_{\infty}} } +\mb P_{a_\infty}^{(j)}   R_{a_\infty}( T, x_0)  +\alpha  \mb P_{a_\infty}^{(j)}  \tilde R_{a_\infty}(T,x_0), 
 \\
 \mb H_{a_\infty}  \mb U(\mb v + \alpha \mb h, T, x_0)  & =   \mb H_{a_\infty}  \mc R(\mb v ,T , x_0) +  \alpha  \mb h_{a_{\infty}} +  \mb H_{a_\infty}  R_{a_\infty}(T, x_0)+  \alpha \mb H_{a_\infty}  \tilde R_{a_\infty}(T,x_0). 
\end{align*}
In the following, we write $T = 1+ \beta$ for $\beta \in  \overline{\mathbb B_{\delta/c}}$ and define for $k= 0,\dots,7$,
\begin{align*}
 \mb \Gamma^{(k)}_{\mb v}(\alpha,\beta,x_0) & := \mb P_{a_\infty}^{(k)} \mc R(\mb v, 1+ \beta, x_0) + \mb P_{a_\infty}^{(k)}   R_{a_\infty}( \beta, x_0)  +\alpha  \mb P_{a_\infty}^{(k)}  \tilde R_{a_\infty}( \beta,x_0)  +\mb P_{a_{\infty}}^{(k)}  \mb I_1(\alpha,\beta,x_0),  \\
  \mb \Gamma^{(8)}_{\mb v}(\alpha,\beta,x_0) & :=  \mb H_{a_\infty}  \mc R(\mb v ,1+ \beta , x_0) +   \mb H_{a_\infty}  R_{a_\infty}(\beta, x_0)+  \alpha \mb H_{a_\infty}  \tilde R_{a_\infty}(\beta,x_0)  + \mb H_{a_\infty}   \mb I_2(\alpha,\beta,x_0)
\end{align*}
 by slight abuse of notation. Then Eq.~\eqref{Eq:Correction_FIniteDim} can be written as 
\begin{align}\label{Eq:FP_Parameters}
\begin{split}
\beta =  \Gamma^{(0)}_{\mb v}( \alpha, \beta, x_0) =  \tilde c_0  (\mb \Gamma^{(0)}_{\mb v}(\alpha,\beta,x_0)  | \mb g_{a_{\infty}}^{(0)} ) \\
  x_0^j =  \Gamma^{(j)}_{\mb v}( \alpha, \beta, x_0) =  \tilde c_j   (\mb \Gamma^{(j)}_{\mb v}(\alpha,\beta,x_0)  |  \mb g_{a_{\infty}}^{(j)} )\\
\alpha  =  \Gamma^{(8)}_{\mb v}( \alpha, \beta, x_0) =  \tilde c_8
(\mb \Gamma^{(8)}_{\mb v}(\alpha,\beta,x_0)  |  \mb h_{a_{\infty}})
\end{split}
\end{align}
for $j = 1, \dots, 7$  some constants $\tilde c_0, \tilde  c_j, \tilde  c_8 \in \R$. First, we show that  $\Gamma_{\mb v} = (\Gamma_{\mb v}^{(0)}, \dots, \Gamma_{\mb v}^{(8)})$ maps $\overline{\mathbb B^{9}_{\delta/c}}$ to itself for suitably large  $c > 0$ and suitably small 
$\delta >0$. Since $\|\mc R(\mb v, 1+ \beta, x_0)\| \lesssim  \|\mb v\|_{\mc Y}$, Eqns.~\eqref{Eq:Remainder_1} - \eqref{Eq:Remainder_2} and Eq.~\eqref{Eq:Integral_terms_delta} imply that 
\[ \Gamma^{(i)}_{\mb v}( \alpha, \beta, x_0)  = O(\tfrac{\delta}{c^2}) + O(\delta^2) \]
for $i = 0,\dots, 8$.  In particular, for $c > 0$ sufficiently large and $\delta = \delta(c)$ sufficiently small, we obtain 
\[ |\Gamma_{\mb v}( \alpha, \beta, x_0)| \leq \tfrac{\delta}{c}. \]
Next, we show that $\Gamma_{\mb v}: \overline{\mathbb B^{9}_{\delta/c}}  \to \overline{\mathbb B^{9}_{\delta/c}}$ is contracting. Let again $(\Phi, a) \in \mc X_{\delta} \times X_{\delta}$ be the functions associated to $\mb v_{\alpha} = \mb v + \alpha  \mb h$, $T = 1+ \beta$ and $x_0$ via Corollary \ref{Cor:NonlEvolution_Corr}. Furthermore, let 
$(\Psi, b) \in \mc X_{\delta} \times X_{\delta}$ be  associated to $\tilde{ \mb v}_{\alpha} = \mb v + \tilde \alpha  \mb h$, $\tilde T = 1+ \tilde \beta $ and $\tilde x_0$. By Proposition \ref{Prop:NonlEvolution_Corr} and Lemma \ref{Le:InitialData_Operator}, we have
\begin{align*}
\|\Phi - \Psi \|_{\mc X} + \|a-b\|_{X} & \lesssim 
\| \mb U(\mb v +  \alpha  \mb h, T, x_0) -  \mb U(\mb v+ \tilde \alpha  \mb h, \tilde T, \tilde x_0) \| \\
&  \lesssim |\alpha - \tilde \alpha| +  |\beta-  \tilde \beta| + |x_0 - \tilde x_0 |.
\end{align*}
Thus, by Lemma \ref{Le:Bound_Ga} and Eq.~\eqref{Bounds:Proj_dtauPsi^*_3},
\begin{align*}
\|\mb P_{a_{\infty}}^{(k)}  \mb I_1(\Phi,a) - \mb P_{b_{\infty}}^{(k)}  \mb I_1(\Psi,b )  \| \lesssim \delta ( \|a-b\|_X + \|\Phi - \Psi\|_{\mc X} ) \lesssim  \delta (|\alpha - \tilde \alpha| +  |\beta-  \tilde \beta| + |x_0 - \tilde x_0 |), 
\end{align*}
 for $k = 0,\dots, 7$, and 
\[ \|\mb H_{a_{\infty}}  \mb I_2(\Phi,a) - \mb H_{b_{\infty}}  \mb I_2(\Psi,b)  \|  \lesssim  \delta (|\alpha - \tilde \alpha| +   |\beta-  \tilde \beta| + |x_0 - \tilde x_0 |).  \]
Furthermore, using Eq.~\eqref{Eq:Lipschitz_R}  and the Lipschitz continuity of the operators $\mb P_{a}^{(k)}$, $\mb H_{a}$ we get
\begin{align*}
\|\mb P_{a_{\infty}}^{(k)} & \mc R(\mb v, T, x_0)  -  \mb P_{b_{\infty}}^{(k)}\mc R(\mb v, \tilde T, \tilde x_0)  \| +  \|\mb H_{a_{\infty}} \mc R(\mb v, T, x_0)  - \mb H_{b_{\infty}}  \mc R(\mb v, \tilde T, \tilde x_0)  \|   \\
& \lesssim \|\mb v\|_{\mc Y} ( \|a- b\|_{X} + |T-\tilde T| +  |x_0 - \tilde x_0 |) \lesssim  \delta (|\alpha - \tilde \alpha| +   |\beta-  \tilde \beta| + |x_0 - \tilde x_0 |).
\end{align*}
Furthermore,  
\begin{align*}
 \|\mb P_{a_\infty}^{(k)}   R_{a_\infty}( T, x_0) - \mb P_{b_\infty}^{(k)}   R_{b_\infty}( \tilde T,  \tilde x_0)\|  & +    \| \alpha  \mb P_{a_\infty}^{(k)}  \tilde R_{a_\infty}( T ,x_0)   -  \tilde \alpha  \mb P_{b_\infty}^{(k)}  \tilde R_{b_\infty}(  \tilde T , \tilde x_0)\|
 \\ &  \lesssim  \delta (|\alpha - \tilde \alpha| + |\beta-  \tilde \beta|  + |x_0 - \tilde x_0 |) 
 \end{align*}
 for $k = 0, \dots, 7$ and analogous estimates hold for $\mb H_{a_{\infty}}$. By combining these bounds we obtain that for $i = 0, \dots,8$, 
\begin{align}\label{Eq:Parmeters_Contr}
| \Gamma^{(i)}_{\mb v}( \alpha, \beta, x_0) -  \Gamma^{(i)}_{\mb v}(\tilde   \alpha, \tilde  \beta, \tilde  x_0)  | \lesssim \delta (|\alpha - \tilde \alpha| +  |\beta-  \tilde \beta| + |x_0 - \tilde x_0 |).
\end{align}
In particular, $ \Gamma_{\mb v}$ is contracting for $\delta$ chosen sufficiently small. An application of the Banach fixed point theorem implies the first part of the statement.

It remains to show that the parameters depend Lipschitz continuously on $\mb v$. To keep track of the dependencies we denote by $(\Phi(\mb v_{\alpha},\beta,x_0) , a(\mb v_{\alpha},\beta,x_0))$ the functions $(\Phi, a)$ associated to $\mb v + \alpha \mb h$, $T = 1+\beta$, $x_0$ via Corollary \ref{Cor:NonlEvolution_Corr}. 
For $\mb v, \mb w \in \mc Y$ satisfying the required smallness condition, let $( \alpha, \beta, x_0)$, $( \tilde \alpha, \tilde  \beta, \tilde  x_0) $ be the corresponding unique set of parameters, such that Eq.~\eqref{Eq:FP_Parameters} is satisfied. The first component of Eq.~\eqref{Eq:FP_Parameters} implies that 
\begin{align*}
 |\beta -  \tilde  \beta| & =| \Gamma^{(0)}_{\mb v}( \alpha, \beta, x_0)  - \Gamma^{(0)}_{\mb w}( \tilde \alpha, \tilde  \beta, \tilde  x_0)  |  \leq | \Gamma^{(0)}_{\mb v}( \alpha, \beta, x_0)  - \Gamma^{(0)}_{\mb w}( \alpha, \beta, x_0) |  \\
 & + |\Gamma^{(0)}_{\mb w}( \alpha, \beta, x_0) - \Gamma^{(0)}_{\mb w}( \tilde \alpha, \tilde  \beta, \tilde  x_0)| . 
 \end{align*}
For the last term, we use Eq.~\eqref{Eq:Parmeters_Contr}. To estimate the first term, we use the Lipschitz continuity of the projections together with Corollary \ref{Cor:NonlEvolution_Corr} to get
\begin{align*}
\|  \mb P_{a_{\infty}(\mb v_{\alpha},\beta,x_0)}^{(0)} & \mc R(\mb v, 1+\beta , x_0) - 
\mb P_{a_{\infty}(\mb w_{\alpha}, \beta,x_0)}^{(0)} \mc R(\mb w, 1+\beta , x_0) \|  \\
& \lesssim \|\mb v \|_{\mc Y} \|a_{\infty}(\mb v_{\alpha},\beta,x_0) - a_{\infty}(\mb w_{\alpha} ,\beta,x_0)\|_X  + \| \mb v - \mb w\|_{\mc Y}  \lesssim \|\mb v - \mb w\|_{\mc Y}.
\end{align*}
Similar estimates exploiting Eqns.~\eqref{Eq:Remainder_1}-\eqref{Eq:Remainder_2}, Lemma \ref{Le:Bound_Ga} and Eq.~\eqref{Bounds:Proj_dtauPsi^*_3} yield
\[  | \Gamma^{(0)}_{\mb v}( \alpha, \beta, x_0)  - \Gamma^{(0)}_{\mb w}( \alpha, \beta, x_0) |  \lesssim \|\mb v - \mb w \|_{\mc Y}.\]
In summary, we obtain 
\[ |\beta -  \tilde  \beta|  \lesssim  \delta (|\alpha - \tilde \alpha| +  |\beta-  \tilde \beta| + |x_0 - \tilde x_0 |) +  \|\mb v - \mb w \|_{\mc Y}.  \]
Analogous estimates for the remaining components show that 
\[ |\alpha - \tilde \alpha|  + |\beta -  \tilde  \beta| + |x_0 - \tilde x_0| \lesssim  \delta (|\alpha - \tilde \alpha| +  |\beta-  \tilde \beta| + |x_0 - \tilde x_0 |) +  \|\mb v - \mb w \|_{\mc Y}, \]
which implies the claim provided that $\delta$ is chosen sufficiently small.
\end{proof}

\begin{remark}\label{Rem:Manifold}
By an argument similar to the one above (but simpler) we can always achieve $\mb U(\mb v, T, x_0) \in \mc M_1$ for suitable $\mb v$ by choosing properly the blowup time and the blowup point (depending on $\mb v$). In this sense, the requirement $\mb U(\mb v, T, x_0) \in \mc M_{2}$ defines a co-dimension one condition on the rescaled and shifted initial data. 
\end{remark}

\subsection{Proof of Theorem \ref{Th:Main}}

\begin{proof}
Let $\delta > 0$ and $c > 0$ be such that Proposition \ref{Prop:VanishingCorrection} holds. Define $\delta' := \frac{\delta}{c}$. Let 
$\mb v = (v_1,v_2) \in C^{\infty}(\overline{\B^7_2}) \times C^{\infty}(\overline{\B^7_2})$ satisfy
\[ \| \mb v \|_{H^4 \times H^3(\B^7_2)} \leq \frac{\delta'}{c} = \frac{\delta}{c^2}.\]
For smooth  $\mb v$, $\mb U(\mb v + \alpha \mb h, T, x_0) \in \mc D(\mb{\tilde L}) \subset   \mc D(\mb L)  = \mc D(\mb L_{a_\infty})$ for small enough $\alpha \in \R$.
By application of Proposition \ref{Prop:VanishingCorrection} and Proposition \ref{Prop:Main} we find that there exist  $a_{\infty} \in \overline{\mathbb B^{7}_{c \delta/\omega}}$, $x_0 \in \overline{\mathbb B^{7}_{\delta}}$, $T \in [1 - \delta, 1+\delta]$ and $\alpha \in [-\delta,\delta]$ depending Lipschitz continuously on $\mb v$ with respect to $\mc Y$ as well as a function  $\Psi  \in C^1([0,\infty), \mc H)$, $\Psi = (\psi_1,\psi_2)$,
\[ \Psi(\tau) = \Psi_{a_{\infty}}^* + \tilde \Phi(\tau) \]
solving the initial value problem
\begin{align}
\begin{split}
\partial_{\tau} \Psi(\tau) & =   \mb L \Psi(\tau) + \mb N(\Psi(\tau)), \quad \tau > 0 \\
\Psi(0) & =  \Psi^*_0 + \mb U(\mb v + \alpha \mb h, T, x_0)
\end{split}
\end{align}
and $\|  \tilde \Phi(\tau) \| \lesssim e^{- \omega \tau}$. Setting
\[ u(t,x) := \tfrac{1}{T-t} \psi_1(-\log(T-t) + \log T)(\tfrac{x-x_0}{T-t} ), \]
we obtain a solution to Eq.~\eqref{Eq:InitialValue_NLW_Codim1} in the backward lightcone $\bigcup_{t\in[0,T)}\{t\} \times \overline{\mathbb B_{T-t}^{7}(x_0)}$.
By definition, 
\[ u^*_{T,x_0,a}(t,x) = \tfrac{1}{T-t} \psi^*_{a}(\tfrac{x-x_0}{T-t}). \]
Eq.~\eqref{Eq:Blowup_behavior} implies that $u^*_{T,x_0,a}$ blows up in $\dot H^k(\mathbb B^7_{T-t}(x_0))$ for $k \geq 3$ with a rate given by $(T-t)^{\frac52-k}$, which yields the suitable normalization factor in the respective norms. In particular, we obtain
\begin{align*}
(T-t)^{\frac12} & \| u(t,\cdot) - u^*_{T,x_0, a_{\infty}}(t,\cdot) \|_{\dot H^3(\mathbb B_{T-t}^7(x_0))}   = 
 \| \psi_1(-\log(T-t) + \log T) - \psi^*_{a_{\infty}}\|_{\dot H^3(\mathbb B^7)}  \\
 &  \leq  \|\Psi(-\log(T-t) + \log T) - \Psi^*_{a_{\infty}} \| \lesssim  
 \| \tilde \Phi(-\log(T-t) + \log T) \|  \lesssim (T-t)^{\omega}.
\end{align*}
The other bounds follow analogously.
\end{proof}

\appendix

\section{Relevant results}

\subsection{Poincar\'e's theorem on difference equations}\label{App:Results}
Let
\begin{equation}\label{eq:Poincare}
	x(n+k) + p_1(n)x(n+k-1) + \cdots + p_k(n)x(n)=0
\end{equation}
be a difference equation with variable coefficients $p_k(n)$ such that there are real numbers $p_i$ for which
\begin{equation}\label{eq:limitpi}
	\lim_{n\rightarrow\infty}p_i(n)=p_i, \quad i\in\{1,2,\dots,k\}.
\end{equation}
\noindent Consequently, we define the \emph{characteristic equation} associated with~\eqref{eq:Poincare}
\begin{equation}\label{eq:char}
	t^k + p_1 t^{k-1}+ \cdots +p_{k-1}t+p_k =0.
\end{equation}
\begin{theorem}[Poincar\'e, 1885]
	Suppose \eqref{eq:limitpi} holds and the roots $t_1, t_2,\cdots,t_k$ of the characteristic equation~\eqref{eq:char} have distinct moduli. If $x(n)$ is a solution of~\eqref{eq:Poincare} then either $x(n)=0$ for all large $n$ or 
	\begin{equation}
		\lim_{n\rightarrow\infty}\frac{x(n+1)}{x(n)}=t_i,
	\end{equation}
	for some $i\in\{1,2,\dots,k\}$.
\end{theorem}
For a proof and further details, we direct the reader to a comprehensive treatment of difference equations by Elaydi \cite{Ela05}.

\subsection{Wall's formulation of the Routh-Hurwitz stability criterion}
\label{App:Wall}
The Routh-Hurwitz stability criterion is a famous control theory result which characterizes real-coefficient polynomials whose zero sets are contained in the open left half-plane. This includes calculating all principal minors of the so-called Hurwitz matrix, which consists of the polynomial coefficients arranged in a specific way. For exposition purposes we found it convenient to use the following reformulation by Wall.

\begin{theorem}[Wall, \cite{Wal45}]
	Let $P(z)=z^n+a_1z^{n-1}+\cdots+a_n$ be a polynomial with real coefficients, and let 
	$Q(z)=a_1z^{n-1}+a_3z^{n-3}+\cdots$ be the polynomial that is comprised of the odd-indexed terms of $P(z)$. Then all the zeros of $P(z)$ have negative real parts if and only if
	\[
	\frac{Q(z)}{P(z)}=\cfrac{1}{1+c_1z+\cfrac{1}{c_2z + \cfrac{1}{ c_3z +\lastcfrac{1}{c_nz}}}}.
	\]
	where the coefficients $c_1$, $c_2$, $\dots$, $c_n$ are all positive.
\end{theorem}

\section{Supersymmetric removal}\label{App:SUSY_removal}

\subsection{Motivation from Supersymmetric Quantum Mechanics}\label{Sec:Motivation_SUSY}
In this section we, rather informally, outline a trick which is well-known in quantum mechanics. Let $$H:= -\partial^2_x + V$$ be a Schr\"odinger operator on $L^2(\R)$, for some convenient smooth potential $V$. Furthermore, assume that $\la_0$ is an eigenvalue of $H$ with the corresponding eigenfunction $f_0 \in C^\infty(\R)$. In addition, we assume that $f_0$ has no zeros. Then $V=f_0''/f_0+\la_0$, and the operator $H$ has the following decomposition
\begin{equation}\label{Def:H}
	H = \left(-\partial_x-\frac{f'_0}{f_0}\right)\left(\partial_x-\frac{f'_0}{f_0}\right) + \la_0 =: Q^*Q + \la_0.
\end{equation}
Now, by reversing the order of the operators $Q^*$ and $Q$, we define the so-called supersymmetric partner of $H$ (relative to $\la_0$)
\begin{equation*}\label{Def:H_tilde}
	\tilde{H} := QQ^* + \la_0 = \left(\partial_x-\frac{f'_0}{f_0}\right)\left(-\partial_x-\frac{f'_0}{f_0}\right) + \la_0.
\end{equation*}
What is more, $\tilde{H}=-\partial_x^2+\tilde{V}$ where
$\tilde{V}= V + 2f_0'^2f_0^{-2},$
and $V$ and $\tilde{V}$ are called supersymmetric partner potentials. Furthermore, the following holds $$\sigma_p(\tilde{H})=\sigma_p(H) \setminus \{\la_0\}.$$ Indeed, let $\la \neq \la_0$ be an eigenvalue of $H$, with the corresponding eigenfunction $f$. Then by acting with $Q$ on Eq.~\eqref{Def:H} we get $\tilde{H}(Qf)=\la Qf$. Furthermore $Qf \neq 0$, since otherwise $f=f_0$, which is not possible because $\la \neq \la_0$. Consequently, $\la$ is an eigenvalue of $\tilde{H}$, and therefore $\sigma_p(H) \setminus \{\la_0\} \subseteq  \sigma_p(\tilde{H})$. To show the reversed inclusion we only prove that $\la_0 \notin \sigma_p(\tilde{H})$, as the rest can be done by the same procedure as above but relative to $\tilde{H}$. Assume contrary, that $\tilde{H}g=\la_0 g$ for some $g \in L^2(\R)$. Then $QQ^*g=0$, and since $\text{rg}\, Q^* \perp \text{ker}\, Q$ we have that $Q^*g=0$. But this in turn implies $g=f_0^{-1}$, which is in contradiction with $g \in L^2(\R)$. To sum up, the operator $\tilde{H}$ has the same set of eigenvalues as $H$ apart from $\la_0$.

\subsection{The Supersymmetric Problem}
In this section we implement the reasoning from above for the problems of type \eqref{Eq:ODEeigenv}. The canonical form we use is the following
\begin{align}\label{Eq:ODE_SUSY}
	(1-\rho^2) f''(\rho) +  \left(\frac{2m}{\rho}   - 2(\lambda +1) \rho \right) f'(\rho)
	-  \la (\la + 1)f(\rho) + V(\rho)  f(\rho)  = 0.
\end{align} 
Assume $\la_0$ is an eigenvalue,\footnote{For terminological consistency with the previous section we somewhat abuse the conventional notion of eigenvalue here.} i.e., there is $f_0 \in C^\infty([0,1])$ which solves Eq.~\eqref{Eq:ODE_SUSY} for $\la=\la_0$. Furthermore, we assume that $f_0$ is nonzero in $(0,1)$. The first step is expressing Eq.~\eqref{Eq:ODE_SUSY} in its normal form. To this end, we introduce a transformation
\begin{equation*}
	f(\rho)=\rho^{-m}(1-\rho^2)^{-\frac{\la+1-m}{2}}g(\rho),
\end{equation*}
which then leads to
\begin{equation*}
	-g''(\rho)+ \left( \frac{(m-1)(\rho^2+m)}{\rho^2(1-\rho^2)^2} + \frac{V(\rho)}{1-\rho^2} +\frac{\la(\la-2m)}{(1-\rho^2)^2} \right)g(\rho) = 0.
\end{equation*}
Now, we define
\begin{equation*}
	V_1(\rho):=  \frac{(m-1)(\rho^2+m)}{\rho^2(1-\rho^2)^2} + \frac{V(\rho)}{1-\rho^2} +\frac{\la_0(\la_0-2m)}{(1-\rho^2)^2},
\end{equation*}
according to which we have that
\begin{equation}\label{Eq:Norm_form}
	-g''(\rho)+V_1(\rho)g(\rho)=\frac{(\la-\la_0)(2m-\la-\la_0)}{(1-\rho^2)^2}g(\rho).
\end{equation}
Obviously, for $g_0(\rho):=\rho^m(1-\rho^2)^{\frac{\la_0+1-m}{2}}f_0(\rho)$ we have $-g_0''+V_1g_0=0$ and we therefore perform the factorization
\begin{equation*}
	-\partial_\rho^2+V_1=(-\partial_\rho-h)(\partial_\rho-h)
\end{equation*}
for $h=g_0'/g_0$. Now, by setting $\tilde{g}:=(\partial_\rho - h)g$, from Eq.~\eqref{Eq:Norm_form} we get
\begin{equation*}
	-(\partial_\rho-h(\rho))\Big((1-\rho^2)^2(\partial_\rho+h(\rho))\Big)\tilde{g}(\rho)=(\la-\la_0)(2m-\la-\la_0)\tilde{g}(\rho).
\end{equation*}
Furthermore, by expanding this expression we have
\begin{equation*}
	-(1-\rho^2)^2\tilde{g}''(\rho)+4(1-\rho^2)\rho \tilde{g}'(\rho) + (1-\rho^2) W(\rho) \tilde{g}=(\la-\la_0)(2m-\la-\la_0)\tilde{g}(\rho),
\end{equation*}
for
\begin{equation*}
	W(\rho)=(1-\rho^2)(h(\rho)^2-h'(\rho))+4\rho h(\rho).
\end{equation*}
Now, to obtain the form \eqref{Eq:ODE_SUSY}, we undo the initial change of variables
\begin{equation*}
	\tilde{g}(\rho)=\rho^m(1-\rho^2)^{\frac{\la-1-m}{2}}\tilde{f}(\rho),
\end{equation*}
by which we arrive at the supersymetric problem
\begin{equation}\label{Eq:SUSY_partner}
	(1-\rho^2) \tilde{f}''(\rho) +  \left(\frac{2m}{\rho}   - 2(\lambda +1) \rho \right) \tilde{f}'(\rho)
	- \la (\la + 1)\tilde{f}(\rho)  + \tilde{V}(\rho) \tilde{f}(\rho)  = 0,
\end{equation}
with the supersymmetric potential 
\begin{equation*}
	\tilde{V}(\rho)=W(\rho)+\frac{(2\la_0-1)m-\la_0^2+1}{1-\rho^2}-\frac{m(m-1)}{\rho^2(1-\rho^2)}-2.
\end{equation*}
\subsection{Correspondence of eigenvalues}
The process above defines a map $f_\la \mapsto \tilde{f}_\la$
\begin{equation*}
	\tilde{f}_\la(\rho)=\rho^{-m}(1-\rho^2)^{-\frac{\la-1-m}{2}}	\Big(\partial_\rho - h(\rho)\Big)\rho^{m}(1-\rho^2)^{\frac{\la+1-m}{2}}f_\la(\rho)
\end{equation*}
between solutions to the original problem \eqref{Eq:ODE_SUSY} and its supersymmetric partner \eqref{Eq:SUSY_partner} (this corresponds to the map $f \mapsto Qf$ from Section \ref{Sec:Motivation_SUSY} above). However, there is no a priori guarantee that for a given $f_\la \in C^\infty([0,1])$ with $\la \neq \la_0$ the corresponding $\tilde{f}_\la$ is admissible, i.e., it is non-trivial and belongs to $C^\infty([0,1])$.  
What we know for sure is that $\tilde{f}_\la \in C^\infty(0,1)$, which follows from the assumption that $f_0$ is non-zero on $(0,1)$. However, smoothness at endpoints depends on precise asymptotic behavior of both $f_{\la_0}$ and $f_{\la}$. As a matter of fact, $\tilde{f}_\la$ can have a logarithmic singularity at an endpoint. Therefore, the endpoint behavior of $\tilde{f}_\la$ has to be checked for every particular problem separately.

As a curiosity, we mention that (as we already implied above) the eigenvalue correspondence failure can happen because the function $\tilde{f}_\la$ ends up being identically zero. Namely, this occurs precisely when $f_{\la}(\rho)=C(1-\rho^2)^{\frac{\la_0-\la}{2}}f_{\la_0}(\rho)$, which in fact happens at an instance we come across in this paper, see Remark \ref{Rem:SUSY_fails}.

\section{Explicit expressions}
\subsection{Expressions for Proposition \ref{prop:l >1_l=0}}\label{App:Expr_l>1}

$C_n(\ell,\la)=P_1(n,\ell,\la)/P_2(n,\ell,\la)$ and $\ve_n(\ell,\la)=P_3(n,\ell,\la)/P_2(n,\ell,\la)$ where

\begin{align*}
	P_1(n,\ell,\la)=16(2n+2\ell+7)(n+1)(2n+\la+\ell+7)(2n+\la+\ell-3),
\end{align*}	
	\begin{align*}
	P_2(n,\ell,\la)=[2\la^2+&4(4n+2\ell+9)\la+(2n+2\ell+9)(8n+3\ell)]\cdot\\&\cdot[2\la^2+4(4n+2\ell+5)\la+(2n+2\ell+7)(8n+3\ell-8)].
	\end{align*}
	\begin{align*}
	P_3(n,\ell,\la)=&(-16n^2-28n\ell-96n+10\ell-124)\la^2\\&+(-96n^2\ell-48n\ell^2-160n^2+24n\ell+104\ell^2-640n+20\ell-568)\la\\&-(2n+2\ell+7)(24n^2\ell+10n\ell^2+64n^2-96n\ell-47\ell^2+32n+250\ell-384).
	\end{align*}
Also, $\delta_3(\ell,\la)=R_1(\ell,\la)/R_2(\ell,\la)$ where
\begin{align*}
	R_1(\ell,\la)=&-(23\ell+190)\la^6-6(30\ell^2+387\ell+1094)\la^5\\&+(-447\ell^3-8670\ell^2-46097\ell-59746)\la^4\\&-4(62\ell^4+2935\ell^3+26998\ell^2+68573\ell+8898)\la^3\\&+(483\ell^5-1770\ell^4-97410\ell^3-473652\ell^2-209537\ell+890030)\la^2\\&+2(342\ell^6+3735\ell^5-8906\ell^4-171178\ell^3-181954\ell^2+914227\ell+77910)\la\\&+3(\ell-1) (81 \ell^6+1427 \ell^5+5140 \ell^4-21610 \ell^3-81983 \ell^2+226647 \ell+314586),
\end{align*}
\begin{align*}
	R_2(\ell,\la)=&2\la^8+16(2\ell+11)\la^7+(216\ell^2+2279\ell+5526)\la^6\\
	&+2  (400 \ell^3+6066 \ell^2+27801 \ell+35914) \la^5\\
	&+(1772\ell^4+34319\ell^3+223270\ell^2+532737\ell+294126)\la^4\\
	&+4  (600 \ell^5+13922 \ell^4+114695 \ell^3+382702 \ell^2+370389 \ell-173626) \la^3\\
	&+(1944\ell^6+51981\ell^5+510802\ell^4+2148322\ell^3+2913884\ell^2-2334671\ell-3619670)\la^2\\
	&+2  (432 \ell^7+12978 \ell^6+146745 \ell^5+738530 \ell^4+1304890 \ell^3-1197446 \ell^2-3998275 \ell+420306) \la\\
	&+3  (\ell-1) (3 \ell+16) (13+2 \ell) (9 \ell^5+201 \ell^4+1410 \ell^3+2594 \ell^2-4235 \ell-4971).
\end{align*}

$\delta_5(0,\la)=R_3(\la)/R_4(\la)$ where
\begin{align*}
R_3(\la)=&-191\la^8-18994\la^7-728158\la^6-14060890\la^5-149594764\la^4-900471766\la^3\\&-3005668466\la^2-4932933534\la-2726072037,
\end{align*}
\begin{align*}
R_4(\la)=&\la^{10}+184\la^9+13910\la^8+562738\la^7+13346440\la^6+191728906\la^5+1667459514\la^4\\&+8524836246\la^3+23936737079\la^2+31789410678\la+13392819504.
\end{align*}

\subsection{Expressions for Proposition \ref{prop:l=1}}\label{App:Expr_l=1}

\[
\delta_1(\la)=\frac{2(\la^3+12\la^2+140\la+144)}{\la^4+50\la^3+752\la^2+3280\la+4224},
\]
\[
C_n(\la)=\frac{-4(n+1)(2n+13)[\la^2+2(2n+5)\la+4(n+3)(n+2)]}{[\la^2+4(2n+7)\la+4(2n+11)(n+2)][\la^2+4(2n+5)\la+4(2n+9)(n+1)]},
\]
\[
\ve_n(\la)=\frac{2[\la^3-2(n^2+5n-5)\la^2-8(2n^2-n-12)\la-8(2n+3)(n+2)(n+1)]}{[\la^2+4(2n+7)\la+4(2n+11)(n+2)][\la^2+4(2n+5)\la+4(2n+9)(n+1)]}.
\]

\pagestyle{plain}
\bibliography{references_paper}
\bibliographystyle{plain}

\end{document}